\documentclass[11pt]{qjm}  %replace-string
\mathchardef\varSigma="0106  %http://arXiv.org
\usepackage{amssymb,latexsym,amsthm}
\font\medit=ptmri at 11pt
\font\smallit=ptmri at 10pt
\font\medsmallit=ptmri at 9pt

\font\smallmedbf=ptmbo at 11pt

\begin{document}
\theoremstyle{plain} %text of this environment is typesetted in italics
\newtheorem{theorem}{\bf Theorem}[section]
\newtheorem{lemma}[theorem]{\bf Lemma}
\newtheorem{corollary}[theorem]{\bf Corollary}
\newtheorem{proposition}[theorem]{\bf Proposition}
\theoremstyle{definition} %text of this environment is typeset in roman letters
\newtheorem{definition}[theorem]{\sc Definition}
\newtheorem{remark}[theorem]{\smallmedbf Remark}
\newtheorem{example}[theorem]{\smallmedbf Example}
%
%If a theorem-like environment should not be numbered,
%add * after \newtheorem, and delete the counter option such as [theorem].
%\newtheorem*{remark0}{\smallmedbf Remark}
\renewcommand{\proofname}{\smallmedbf Proof.}
\renewcommand{\theequation}{\thesection.\arabic{equation}}

\mathchardef\varSigma="0106
\mathchardef\varGamma="0100 
\mathchardef\varTheta="0102 
\mathchardef\varLambda="0103 
\mathchardef\varXi="0104 
\mathchardef\varPi="0105 
\mathchardef\varSigma="0106 
\mathchardef\varPhi="0108 
\mathchardef\varPsi="0109 
\mathchardef\varOmega="010A
\def\af{af\-fine}
\def\afs{af\-fine space}
\def\dr{\hbox{D\'\i az\hh-}\hskip0ptRa\-mos}
\def\fmf{\hbox{four\hh-}\hskip0ptman\-i\-fold}
\def\gr{Gar\-\hbox{c\'\i a\hs-}\hskip0ptR\'\i o}
\def\jos{Jor\-dan-Osser\-man}
\def\itjos{Jor\-\hbox{dan\hs-}\hskip0ptOsser\-man}
\def\ku{Ku\-pe\-li}
\def\nw{\hbox{non\hh-\hskip-1.5pt}\hskip0ptWalk\-er}
\def\itnw{\hbox{non\hh-\hskip-1.8pt}\hskip0ptWalk\-er}
\def\pse{pseu\-\hbox{do\hs-}\hskip0ptEuclid\-e\-an}
\def\psr{pseu\-\hbox{do\hs-\hn}\hskip0ptRiem\-ann\-i\-an}
\def\tvs{translation vector space}
\def\va{V\nh\'az\-que\hbox{z\hh-}\hskip0ptAbal}
\def\vl{\hbox{V\nh\'az}\-\hbox{quez\hh-}\hskip0ptLo\-ren\-zo}
\def\vs{vector space}
\def\dpi{\hbox{$(\hskip-.6ptd\pi\hskip-1pt)$}} 
\def\bbC{{\mathchoice {\setbox0=\hbox{$\displaystyle\mathrm{C}$}\hbox{\hbox 
to0pt{\kern0.4\wd0\vrule height0.9\ht0\hss}\box0}} 
{\setbox0=\hbox{$\textstyle\mathrm{C}$}\hbox{\hbox 
to0pt{\kern0.4\wd0\vrule height0.9\ht0\hss}\box0}} 
{\setbox0=\hbox{$\scriptstyle\mathrm{C}$}\hbox{\hbox 
to0pt{\kern0.4\wd0\vrule height0.9\ht0\hss}\box0}} 
{\setbox0=\hbox{$\scriptscriptstyle\mathrm{C}$}\hbox{\hbox 
to0pt{\kern0.4\wd0\vrule height0.9\ht0\hss}\box0}}}} 
\def\lie{\bbL}
\def\bbL{\mathrm{I\!L}}
\def\bbP{\mathrm{I\!P}} 
\def\bbR{\mathrm{I\!R}} 
\def\rto{\bbR\hskip-.5pt^2} 
\def\rtr{\bbR\hskip-.7pt^3} 
\def\rn{{\bbR}^{\hskip-.6ptn}} 
\def\rk{{\bbR}^{\hskip-.6ptk}} 
\def\tw{{\hskip.7pt\widetilde{\hskip-.7ptw\hskip-.4pt}\hskip.4pt}}
\def\bbT{{\mathchoice {\setbox0=\hbox{$\displaystyle\rm
T$}\hbox{\hbox to0pt{\kern0.3\wd0\vrule height0.9\ht0\hss}\box0}}
{\setbox0=\hbox{$\textstyle\mathrm{T}$}\hbox{\hbox
to0pt{\kern0.3\wd0\vrule height0.9\ht0\hss}\box0}}
{\setbox0=\hbox{$\scriptstyle\mathrm{T}$}\hbox{\hbox
to0pt{\kern0.3\wd0\vrule height0.9\ht0\hss}\box0}}
{\setbox0=\hbox{$\scriptscriptstyle\mathrm{T}$}\hbox{\hbox
to0pt{\kern0.3\wd0\vrule height0.9\ht0\hss}\box0}}}}
\def\tx{\mathrm{K}_\xi}
\def\bbZ{\mathbf{Z}} 
\def\md{{\vbox{\hbox{$\mathbf.$}\vskip1.3pt}}}
\def\hyp{\hskip.5pt\vbox 
{\hbox{\vrule width3ptheight0.5ptdepth0pt}\vskip2.2pt}\hskip.5pt} 
\def\ea{a}
\def\eb{b}
\def\ec{c}
\def\eg{g}
\def\ep{p}
\def\eq{q}
\def\er{r}
\def\es{s}
\def\et{t}
\def\eu{u}
\def\ev{v}
\def\ew{w}
\def\ex{x}
\def\ey{y}
\def\ez{z}
\def\yi{\mathrm{I}}
\def\ym{\mathrm{M}}
\def\yn{\mathrm{N}}
\def\yx{\mathrm{X}}
\def\fn{{f}}
\def\nsu{\nabla_{\hskip-2.2pt\eu}}
\def\nsv{\nabla_{\hskip-2.2pt\ev}}
\def\nsvu{\nsv\hh\eu}
\def\nsvw{\nsv\hh\ew}
\def\nsw{\nabla_{\hskip-2pt\ew}}
\def\nswu{\nsw\eu}
\def\nswv{\nsw\ev}
\def\nsww{\nsw\hn\ew}
\def\nav{\nabla\hskip-.4pt\ev}
\def\naw{\nabla\nh\ew}
\def\dvw{\mathrm{div}\hskip1.4pt\ew}
\def\df{d\hskip-.8pt\ef} 
\def\dfn{d\hskip-.8pt\fn} 
\def\dfh{d\hskip-.8pt\fh} 
\def\trf{[\hskip-1.7pt[\hskip-.4ptF\hskip.8pt]\hskip-1.7pt]} 
\def\trfk{[\hskip-1.7pt[\hskip-.4ptF^K\hskip0pt]\hskip-1.7pt]} 
\def\trfo{[\hskip-2pt[\hskip-.4ptF_0]\hskip-2pt]} 
\def\trz{[\hskip-2pt[\hskip.4pt\zeta\hskip.7pt]\hskip-2pt]} 
\def\trfz{[\hskip-2pt[\hskip.8ptf\zeta\hskip.7pt]\hskip-2pt]} 
\def\tb{T\hskip.3pt\bs} 
\def\tab{{T\hskip.2pt^*\hskip-.9pt\bs}} 
\def\tayb{{T_y^*\hskip-1pt\bs}} 
\def\taxb{{T_x^*\hskip-1pt\bs}} 
\def\tyb{{T\hskip-2pt_y\hskip-.2pt\bs}} 
\def\txm{{T\hskip-2.3pt_x\hskip-.5ptM}} 
\def\tym{{T\hskip-2pt_y\hskip-.9ptM}} 
\def\txn{{T\hskip-2pt_x\hskip-.9ptN}} 
\def\txtm{{T\hskip-2pt_{x(t)}\hskip-.9ptM}} 
\def\txom{{T\hskip-2pt_{x(0)}\hskip-.9ptM}} 
\def\txtsm{{T\hskip-2pt_{x(t,s)}\hskip-.9ptM}} 
\def\olm{\hskip2.1pt\overline{\hskip-2.1ptM}} 
\def\bw{\hskip1.8pt\overline{\hskip-1.8ptw\hskip-.8pt}\hskip.8pt}
\def\bu{\hskip1pt\overline{\hskip-1ptu\hskip-.4pt}\hskip.4pt} 
\def\bv{\hskip.8pt\overline{\hskip-.8ptv\hskip-.4pt}\hskip.4pt} 
\def\bc{\hskip1pt\overline{\hskip-1pt\mathrm{D}\hskip-1.8pt}\hskip1.8pt} 
\def\br{\hskip3.3pt\overline{\hskip-3.3ptR\hskip-.7pt}\hskip.7pt} 
\def\brh{\hskip1.3pt\overline{\hskip-1.3pt\rho\hskip-.1pt}\hskip.1pt} 
\def\hd{\widehat{\mathrm{D}}} 
\def\tm{{T\hskip0pt\ym}} 
\def\tn{{T\hskip-.3ptN}} 
\def\tu{{T\hskip-.3ptU}} 
\def\tam{{T^*\!M}} 
\def\so{\mathfrak{so}\hh} 
\def\gi{\mathfrak{g}} 
\def\hi{\mathfrak{h}} 
\def\fh{f} 
\def\hf{\varLambda} 
\def\line{\Lambda} 
\def\plane{\Pi} 
\def\ds{\mathrm{S}} 
\def\plne{\mathrm{S}} 
\def\spce{\mathrm{E}}
\def\spac{\mathrm{V}}
\def\cst{^{\mathrm{cst}}} 
\def\lin{^{\mathrm{lin}}} 
\def\qdr{^{\mathrm{qdr}}} 
\def\cub{^{\mathrm{cub}}} 
\def\qrt{^{\mathrm{qrt}}} 
\def\qnt{^{\mathrm{qnt}}} 
\def\rc{c} 
\def\rd{X} 
\def\kri{\mathrm{Ker}\hskip1pt\ri} 
\def\kr{\mathrm{Ker}\hskip1ptR} 
\def\kw{\mathrm{Ker}\hskip2ptW} 
\def\kb{\mathrm{Ker}\hskip1.5ptB} 
\def\xc{\mathcal{X}_c} 
\def\fd{F\hskip-2pt.} 
\def\ob{\mathcal{B}} 
\def\od{\mathcal{D}} 
\def\op{\mathcal{P}} 
\def\as{\mathcal{S}} 
\def\sop{\as^+}
\def\soz{\as^{\hh0}}
\def\xz{\mathcal{X}} 
\def\yz{\mathcal{Y}} 
\def\oz{\mathcal{Z}} 
\def\tim{\hskip1.5pt\widetilde{\hskip-1.5ptM\hskip-.5pt}\hskip.5pt} 
\def\hm{\hskip1.9pt\widehat{\hskip-1.9ptM\hskip-.2pt}\hskip.2pt} 
\def\hmt{\hskip1.9pt\widehat{\hskip-1.9ptM\hskip-.5pt}_t} 
\def\hmz{\hskip1.9pt\widehat{\hskip-1.9ptM\hskip-.5pt}_0} 
\def\hmp{\hskip1.9pt\widehat{\hskip-1.9ptM\hskip-.5pt}_p} 
\def\hg{\hskip1.2pt\widehat{\hskip-1.2ptg\hskip-.4pt}\hskip.4pt} 
\def\nao{\hbox{$\nabla\!\!^{^{^{_{\!\!\circ}}}}$}} 
\def\ro{\hbox{$R\hskip-4.5pt^{^{^{_{\circ}}}}$}{}} 
\def\mppp{\hbox{$-$\hskip1pt$+$\hskip1pt$+$\hskip1pt$+$}} 
\def\mpdp{\hbox{$-$\hskip1pt$+$\hskip1pt$\dots$\hskip1pt$+$}} 
\def\mmpp{\hbox{$-$\hskip1pt$-$\hskip1pt$+$\hskip1pt$+$}} 
\def\mmp{\hbox{$-$\hskip1pt$-$\hskip1pt$+$}} 
\def\mpp{\hbox{$-$\hskip1pt$+$\hskip1pt$+$}} 
\def\nl{\hbox{$-$\hskip1pt$+$}} 
\def\mmmp{\hbox{$-$\hskip1pt$-$\hskip1pt$-$\hskip1pt$+$}} 
\def\pppp{\hbox{$+$\hskip1pt$+$\hskip1pt$+$\hskip1pt$+$}} 
\def\mpmp{\hbox{$-$\hskip1pt$\pm$\hskip1pt$+$}} 
\def\mpmpp{\hbox{$-$\hskip1pt$\pm$\hskip1pt$+$\hskip1pt$+$}} 
\def\mmpmp{\hbox{$-$\hskip1pt$-$\hskip1pt$\pm$\hskip1pt$+$}} 
\def\q{q} 
\def\bq{\hat q} 
\def\p{p} 
\def\xh{\chi}
\def\ef{f}
\def\xf{\psi}
\def\x{v} 
\def\y{y} 
\def\vdx{\vd{}\hskip-4.5pt_x} 
\def\bz{b\hh} 
\def\cy{{y}} 
\def\rkwo{\,\hs\mathrm{rank}\hskip2.7ptW\hskip-2.7pt=\hskip-1.2pt1} 
\def\rkwho{\,\hs\mathrm{rank}\hskip2.2ptW^h\hskip-2.2pt=\hskip-1pt1} 
\def\rkw{\,\hs\mathrm{rank}\hskip2.4ptW\hskip-1.5pt} 
\def\rw{\varPsi} 
\def\nd{Q} 
\def\ft{\Psi} 
\def\js{J} 
\def\ism{H} 
\def\fe{F} 
\def\fy{f} 
\def\dfc{dF\hskip-2.3pt_\cy\hskip.4pt} 
\def\dfct{dF\hskip-2.3pt_\cy(t)\hskip.4pt} 
\def\dic{d\im\hskip-1.4pt_\cy\hskip.4pt} 
\def\qt{\mathcal{E}} 
\def\tqt{\tilde{\qt}} 
\def\vh{h} 
\def\mv{V} 
\def\sv{\mathcal{V}} 
\def\vy{\mathcal{V}} 
\def\xv{\mathcal{X}} 
\def\yv{\mathcal{Y}} 
\def\iv{\mathcal{I}} 
\def\gkp{\Sigma} 
\def\bs{\Sigma} 
\def\dbs{\dot\bs} 
\def\das{\hskip3pt\dot{\hskip-3pt\as}} 
\def\hs{\hskip.7pt} 
\def\hh{\hskip.4pt} 
\def\nh{\hskip-.7pt} 
\def\hn{\hskip-.4pt} 
\def\nnh{\hskip-1pt} 
\def\hrz{^{\hskip.5pt\mathrm{hrz}}} 
\def\vrt{^{\hskip.2pt\mathrm{vrt}}} 
\def\vt{\varTheta} 
\def\lf{\lambda} 
\def\zh{\zeta} 
\def\vg{\varGamma} 
\def\bvg{\hskip3.3pt\overline{\hskip-3.3pt\vg\hskip-.7pt}\hskip.7pt} 
\def\gm{g}%{\gamma} 
\def\gp{\mathrm{G}} 
\def\hp{\mathrm{H}} 
\def\kp{\mathrm{K}} 
\def\yp{\mathrm{Y}} 
\def\Gm{\Gamma} 
\def\Dt{\Delta} 
\def\sj{\sigma} 
\def\lg{\langle} 
\def\rg{\rangle} 
\def\lr{\lg\,,\rg} 
\def\uv{\underline{v\hskip-.8pt}\hskip.8pt} 
\def\uvp{\underline{v\hh'\hskip-.8pt}\hskip.8pt} 
\def\uw{\underline{w\hskip-.8pt}\hskip.8pt} 
\def\uxs{\underline{x_s\hskip-.8pt}\hskip.8pt} 
\def\vs{vector space} 
\def\vf{{q}} 
\def\tf{tensor field} 
\def\tvn{the vertical distribution} 
\def\dn{distribution} 
\def\pt{point} 
\def\tc{tor\-sion\-free connection} 
\def\rt{Ric\-ci tensor} 
\def\pde{partial differential equation} 
\def\pf{projectively flat} 
\def\pfs{projectively flat surface} 
\def\pfc{projectively flat connection} 
\def\pftc{projectively flat tor\-sion\-free connection} 
\def\su{surface} 
\def\sco{simply connected} 
\def\psr{pseu\-\hbox{do\hs-\hn}\hskip0ptRiem\-ann\-i\-an} 
\def\inv{-in\-var\-i\-ant} 
\def\trinv{trans\-la\-tion\inv} 
\def\feo{dif\-feo\-mor\-phism} 
\def\feic{dif\-feo\-mor\-phic} 
\def\feicly{dif\-feo\-mor\-phi\-cal\-ly} 
\def\Feicly{Dif\-feo\-mor\-phi\-cal\-ly} 
\def\nw{\hbox{non\hh-\hskip-1.5pt}\hskip0ptWalk\-er} 
\def\itnw{\hbox{non\hh-\hskip-1.8pt}\hskip0ptWalk\-er} 
\def\diml{-di\-men\-sion\-al} 
\def\prl{-par\-al\-lel} 
\def\skc{skew-sym\-met\-ric} 
\def\sky{skew-sym\-me\-try} 
\def\Sky{Skew-sym\-me\-try}
\def\dbly{-dif\-fer\-en\-ti\-a\-bly} 
\def\cs{con\-for\-mal\-ly symmetric} 
\def\cf{con\-for\-mal\-ly flat} 
\def\ls{locally symmetric} 
\def\kf{Killing field} 
\def\om{\omega} 
\def\vol{\varOmega} 
\def\dt{\delta} 
\def\ve{\varepsilon} 
\def\zt{\zeta}
\def\cn{\kappa} 
\def\mf{manifold} 
\def\mfd{-man\-i\-fold} 
\def\bmf{base manifold} 
\def\bd{\hs\overline{\nh\mathrm{D}\hskip-1.4pt}\hskip1.4pt}
\def\ed{T} 
\def\bna{\hs\overline{\nh\nabla\nh}\hs} 
\def\gs{s} 
\def\hj{t} 
\def\de{\dot\ea}
\def\ha{{H}}
\def\dh{\dot\ha}
\def\qa{{Q}}
\def\dq{\dot\qa}
\def\paf{{P}}
\def\dpaf{\dot\paf}
\def\qx{Q} 
\def\lx{L} 
\def\lp{\lx\nnh^+} 
\def\lm{\lx\nnh^-} 
\def\lmo{\lx_1^{\hskip-1.3pt-}} 
\def\lpt{\lx_2^{\hskip-1.3pt+}} 
\def\ty{T} 
\def\gy{G} 
\def\ny{N\nh} 
\def\oy{{I\hskip-3.3ptP\hskip-.8pt}} 
\def\apb{{\mathcal{A}}} 
\def\alb{{\mathcal{J}}} 
\def\pb{{\mathcal{P}}} 
\def\lb{{\mathcal{L}}} 
\def\ki{\mathcal{L}}
\def\tl{{\mathcal{D}}} 
\def\ow{{\mathcal{W}}} 
\def\tlp{\tl^\perp} 
\def\tbd{tangent bundle} 
\def\ctb{cotangent bundle} 
\def\bp{bundle projection} 
\def\pr{pseu\-d\hbox{o\hs-}\hskip0ptRiem\-ann\-i\-an} 
\def\prc{pseu\-d\hbox{o\hs-}\hskip0ptRiem\-ann\-i\-an metric} 
\def\prd{pseu\-d\hbox{o\hs-}\hskip0ptRiem\-ann\-i\-an manifold} 
\def\Prd{pseu\-d\hbox{o\hs-}\hskip0ptRiem\-ann\-i\-an manifold} 
\def\npd{null parallel distribution} 
\def\pj{-pro\-ject\-a\-ble} 
\def\pd{-pro\-ject\-ed} 
\def\lcc{Le\-vi-Ci\-vi\-ta connection} 
\def\vb{vector bundle} 
\def\vbm{vec\-tor-bun\-dle morphism} 
\def\kerd{\mathrm{Ker}\hskip2.7ptd} 
\def\ro{\rho} 
\def\sy{\varepsilon} 
\def\ts{t} 
\def\pmb{\pi} 
\def\ri{{\rho}}

\title[{{\smallit Noncompactness and maximum mobility}}]{NONCOMPACTNESS AND 
MAXIMUM MOBILITY OF TYPE III RICCI-FLAT SELF-DUAL NEUTRAL WALKER 
FOUR-MANIFOLDS}
%title of paper and the running head option
\author[A. DERDZINSKI]{ANDRZEJ DERDZINSKI}
\address{\medit(Department of Mathematics, The Ohio State University, 
Columbus, OH 43210, USA)}

%\urladdr{http:/\hskip-1.5pt/www.math.ohio-state.edu/\~{}andrzej}
\begin{abstract}
\begin{center}{\bf Abstract}\end{center}
\vskip1pt
It is shown that a self-dual neutral Einstein four-manifold of Petrov type 
III, admitting a two\hh-dimensional null parallel distribution compatible with 
the orientation, cannot be compact or locally homogeneous, and its maximum 
possible degree of mobility is \hskip2pt3. \hbox{D\'\i az\hh-}\hn Ra\-mos, 
Garc\'\i a\hs-\hn R\'\i o and V\hn\'azquez\hs-\hn Lorenzo found a general 
coordinate form of such manifolds. The present paper also provides a modified 
version of that coordinate form, valid in a suitably defined generic case and, 
in a sense, ``\hs more canonical\hskip-.2pt'' than the usual formulation. 
Moreover, the lo\-\hbox{cal\hh-}\hskip0ptisometry types of manifolds as above 
having the degree of mobility equal to \hs3\hs\ are classified. Further results 
consist in explicit descriptions, first, of the kernel and image of the 
Killing operator for any tor\-sion\-free surface connection with 
everywhere\hh-nonzero, \hbox{skew\hh-}\hskip0ptsymmetric Ric\-ci tensor, and, 
secondly, of a moduli curve for surface connections with the properties just 
mentioned that are, in addition, locally homogeneous. Finally, hyperbolic 
plane geometry is used to exhibit examples of co\-di\-men\-sion-two foliations 
on compact manifolds of dimensions greater than \hs2\hs\ admitting a transversal 
tor\-sion\-free connection, the Ric\-ci tensor of which is 
\hbox{skew\hh-}\hskip0ptsymmetric and nonzero everywhere. No such connection 
exists on any closed surface, so that there are no analogous examples in 
dimension 2.
\end{abstract}

\email{andrzej@math.ohio-state.edu}
%\subjclass{53C50, 53C12; 53B30, 53B05}
%\keywords{ %key words and phrases
%Cur\-va\-ture\hskip.8pt-ho\-mo\-ge\-ne\-ous neutral Ein\-stein metric, type 
%III Osser\-man metric}
\acknow{}

\maketitle

\setcounter{theorem}{0}

\voffset=28pt\hoffset=-7pt %centered at home

\baselineskip14pt
\section{Introduction}\label{intr}
\setcounter{equation}{0}
A trace\-less en\-do\-mor\-phism of a 
pseu\-d\hbox{o\hs-}\hskip0ptEuclid\-e\-an $\,3$-space is said to be of 
{\it Pe\-trov type\/} III if it is self-ad\-joint and sends some 
ordered basis %$\,X,Y\nnh,Z\,$ %$\,\mathbf{p},\mathbf{q},\mathbf{r}\,$ 
$\,p,q,r\,$ to %$\,0,X,Y\nnh$.%$\,\mathbf{0},\mathbf{p},\mathbf{q}$.
$\,0,p,q$.

By a {\it type\/}~III {\it SDNE manifold\/} we mean a {\it self-du\-al neutral 
Ein\-stein \hbox{four\hskip.4pt-}\hskip0ptman\-i\-fold\/} $\,(\ym,\gm)\,$ {\it 
of Pe\-trov type\/}~III. In other words, $\,(\ym,\gm)\,$ is assumed to be a 
self-du\-al oriented Ein\-stein \hbox{four\hskip.4pt-}\hskip0ptman\-i\-fold of 
the neutral metric signature $\,(\mmpp)$, such that the self-du\-al Weyl 
tensor $\,W^+$ of $\,(\ym,g)$, acting on self-du\-al $\,2$-forms, is of 
Pe\-trov type~III at every point.

This paper deals with type~III SDNE manifolds $\,(\ym,\gm)\,$ having the {\it 
Walker property}, that is, admitting a 
\hbox{two\hh-}\hskip0ptdi\-men\-sion\-al null parallel distribution which is 
compatible with the orientation in the sense explained immediately before 
Remark~\ref{liebr}. The main results, Theorems~\ref{nocpl} and~\ref{maxmo}, 
state that such $\,(\ym,\gm)\,$ is never compact, while its degree of mobility 
is at most $\,3$, and so $\,(\ym,\gm)\,$ cannot be locally homogeneous. In the 
case where the degree of mobility equals $\,3$, the lo\-cal-i\-som\-e\-try 
types of $\,(\ym,\gm)\,$ are explicitly classified (Theorem~\ref{three}).

Questions about type~III SDNE manifolds arise for two reasons. First, these 
are precisely the {\it type\/} III {\it 
Jor\-\hbox{dan\hskip.7pt-}\hskip0ptOsser\-man 
\hbox{four\hskip.4pt-}\hskip0ptman\-i\-folds\/} 
\cite[Remark 2.1]{diaz-ramos-garcia-rio-vazquez-lorenzo}, a subclass of the 
class of Jor\-\hbox{dan\hskip.4pt-}\hskip0ptOsser\-man manifolds, studied by 
many authors \cite{garcia-rio-kupeli-vazquez-lorenzo,gilkey}. Secondly, 
type~III SDNE metrics are all curvature homogeneous 
\cite[pp.\ 247--248]{blazic-bokan-rakic}, so that understanding their 
structure is a step towards a description of all 
cur\-va\-\hbox{ture\hskip1.3pt-}\hskip0ptho\-mo\-ge\-ne\-ous 
pseu\-\hbox{do\hs-\hn}\hskip0ptRiem\-ann\-i\-an Ein\-stein metrics in 
dimension four.%\hbox{four\hh-}\hskip0ptman\-i\-folds.

Of the two main results mentioned above, Theorem~\ref{nocpl} follows from the 
divergence formula: as shown in Lemma~\ref{dvten}, every type~III SDNE 
Walker manifold carries a natural vector field with nonzero constant 
divergence. The proof of Theorem~\ref{maxmo} uses, in turn, some conclusions 
about pairs $\,(\bs,\nnh\nabla)\,$ formed by a surface $\,\bs\,$ and a 
tor\-sion\-free connection $\,\nabla\hn\,$ on $\,\bs\,$ with 
eve\-ry\-\hbox{where\hskip1pt-}\hskip0ptnon\-zero, skew-sym\-met\-ric Ric\-ci 
tensor. Specifically, Sections~\ref{kers}, \ref{ioko} and~\ref{ngrs} contain a 
characterization of the image and kernel of the {\it Kil\-ling operator\/} of 
$\,(\bs,\nnh\nabla)$, which sends each $\,1$-form to its sym\-met\-riz\-ed 
\hbox{$\,\nabla\nnh$-\hh}\hskip0ptco\-var\-i\-ant derivative. The other  
conclusion, Theorem~\ref{modul}, describes a moduli curve for pairs 
$\,(\bs,\nnh\nabla)\,$ with the properties just listed that are, in addition, 
locally homogeneous. (A canonical coordinate form of such locally homogeneous 
pairs $\,(\bs,\nnh\nabla)\,$ was first found by Ko\-wal\-ski, O\-poz\-da and 
Vl\'a\-\v sek \cite{kowalski-opozda-vlasek-00}.)

Pairs $\,(\bs,\nnh\nabla)\,$ as above are naturally related to type~III 
SDNE Walker manifolds. Namely, \dr, \gr\ and \vl\ proved in 
\cite[Theorem~3.1(ii.3)]{diaz-ramos-garcia-rio-vazquez-lorenzo} that, locally, 
type~III SDNE Walker metrics are nothing else than Patter\-son and Walker's 
Rie\-mann-ex\-ten\-sion metrics associated with triples 
$\,(\bs,\nnh\nabla\nnh,\tau)\,$ consisting of any such pair 
$\,(\bs,\nnh\nabla)\,$ and an arbitrary symmetric $\,2$-ten\-sor 
$\,\tau\,$ on $\,\bs$. For details, see Section~\ref{tstd}. Although 
$\,\tau$, unlike $\,\bs\,$ and $\,\nabla\nnh$, is not a geometric invariant of 
the metric $\,\gm$, a canonical choice of $\,\tau\,$ is possible for a 
class of type~III SDNE Walker manifolds satisfying a 
gen\-er\-al-po\-si\-tion requirement introduced in Section~\ref{ttwg}.

The result of \cite[Theorem~3.1(ii.3)]{diaz-ramos-garcia-rio-vazquez-lorenzo} 
implies that every type~III SDNE Walker manifold carries a 
co\-di\-men\-sion-two foliation admitting a transversal tor\-sion\-free 
connection with eve\-ry\-\hbox{where\hskip1pt-}\hskip0ptnon\-zero, 
skew-sym\-met\-ric Ric\-ci tensor. The fact that type~III SDNE Walker 
manifolds are noncompact (Theorem~\ref{nocpl}) cannot be derived just from the 
presence of such a foliation. Namely, as shown in Proposition~\ref{genus} and 
Corollary~\ref{hidim}, foliations with the stated property exist 
on compact manifolds of all dimensions $\,n\ge3\,$ (though not for 
$\,n=2$).

It is unknown whether the conclusion about non\-com\-pact\-ness of all 
type~III SDNE Walker manifolds, established in Theorem~\ref{nocpl}(a), remains 
valid in the \hbox{non\hh-\hskip-1.5pt}\hskip0ptWalk\-er case. That the two 
cases differ in global properties other than compactness is exemplified by 
{\it vertical completeness}, introduced in Section~\ref{glpr}. Specifically, 
type~III SDNE manifolds $\,(\ym,\gm)\,$ are sometimes vertically complete, 
yet, according to Theorem~\ref{compl}, this can happen only if $\,\gm\,$ is a 
Walker metric.

\section{Preliminaries}\label{prel} 
\setcounter{equation}{0} 
Manifolds are by definition connected. All manifolds, bundles, their 
sections and sub\-bun\-dles, as well as connections and mappings, including 
bundle morphisms, are assumed to be $\,C^\infty\nnh$-dif\-fer\-en\-ti\-a\-ble. 
A bundle morphism always operates between two bundles with the same base 
manifold, and acts by identity on the base.

By the {\it degree of mobility\/} of a connection $\,\nabla\hn\,$ (or, of a 
pseu\-d\hbox{o\hs-}\hskip0ptRiem\-ann\-i\-an metric $\,\gm$) on a manifold 
$\,\bs\,$ we mean the function assigning to each $\,y\in\bs\,$ the dimension 
of the Lie algebra $\,\mathfrak{a}_y$ (or, $\,\mathfrak{i}\hh_y$) formed by 
the germs at $\,y\,$ of infinitesimal af\-fine transformations for 
$\,\nabla\hn\,$ (or, respectively, of Kil\-ling fields for $\,\gm$). This 
function is constant when $\,\nabla\hn\,$ (or, $\,\gm$) is locally homogeneous.

For a connection $\,\nabla\hn\,$ in a real vector bundle over a manifold 
$\,\bs$, sections $\,\alpha\,$ of the bundle and vector fields $\,\ev,\ew\,$ 
tangent to $\,\bs$, our sign convention about the curvature tensor $\,R\,$ of 
$\,\nabla\hn\,$ is
\begin{equation}\label{cur}
R\hh(\hn\ev,\ew)\hs\alpha\,=\,\nsw\nsv\hs\alpha\,-\,\nsv\nsw\hs\alpha\,
+\,\nabla_{\nh[\ev,\ew]}\alpha\hs.
\end{equation}
If $\,\nabla\nh\,$ is a connection on $\,\bs\,$ (that is, in the tangent  
bundle $\,\tb$), we treat the covariant derivative $\,\naw\,$ of any vector 
field $\,\ew\,$ as
\begin{equation}\label{nwt}
\mathrm{\ \ the\ bundle\ morphism\ \ }\naw:\tb\to\tb\mathrm{\ \ sending\ each\ 
vector\ field\ \ }\ev\mathrm{\ \ to\ \ }\nsvw\hh.
\end{equation}
Given a manifold $\,\bs\,$ and a bundle morphism $\,A:\tb\to\tb$,
\begin{equation}\label{ast}
A^{\nh*}\nh:\tab\to\tab\mathrm{\ \ denotes\ the\ bundle\ morphism\ dual\ to\ \ 
}A\hh,
\end{equation}
so that $\,A^{\nh*}$ sends any $\,1$-form $\,\xi\,$ to the composite 
$\,A^{\nh*}\hn\xi=\xi A\,$ in which $\,A:\tb\to\tb\,$ is followed by the 
morphism $\,\xi\,$ from $\,\tb\,$ to the product bundle $\,\bs\times\bbR\hh$.

For the Ric\-ci tensor $\,\rho\,$ of a tor\-sion\-free connection 
$\,\nabla\hn\,$ on a manifold $\,\bs$, and any tangent vector field $\,\ew$, 
the {\it Boch\-ner identity\/} states that
\begin{equation}\label{bch}
\mathrm{a)}\hskip9pt\rho\hs(\,\cdot\,,\ew)\,=\,\hs\mathrm{div}\hs[\naw]\,
-\,\hs d\hs[\dvw]\hh,\hskip8pt\mathrm{where}\hskip29pt\mathrm{b)}
\hskip9pt\dvw=\mathrm{tr}\hskip2.5pt\naw\hh.
\end{equation}
Cf.\ \cite[formula~(4.39) on p.\ 449]{dillen-verstraelen}. 
Here $\,\mathrm{tr}\hskip2.5pt\naw:\bs\to\bbR\,$ is the pointwise trace of 
(\ref{nwt}).

In fact, the coordinate form of (\ref{bch}.a), 
$\,R_{jk}\ew^{\hs k}\nh=\ew^{\hs k}{}_{,jk}-\ew^{\hs k}{}_{,\hs kj}$, arises 
by contraction in $\,l=k\,$ from the Ric\-ci identity $\,\ew^{\hs l}{}_{,jk}
-\ew^{\hs l}{}_{,\hs kj}=R_{jks}{}^l\ew^s\nnh$, which in turn is nothing else 
than (\ref{cur}).

For the tensor, exterior and symmetric products, and the exterior derivative 
of $\,1$-forms $\,\beta,\xi\,$ on a manifold, any tangent vector fields 
$\,\eu,\ew$, and any fixed tor\-sion\-free connection $\,\nabla\nnh$, we have
\begin{equation}\label{bwa}
\begin{array}{rl}
\mathrm{a)}&[\beta\otimes\hs\xi](\hn\eu,\ew)=\beta(\hn\eu)\hs\xi(\ew)\hh,
\hskip12pt\beta\wedge\hs\xi=\beta\otimes\hs\xi-\xi\otimes\beta\hh,
\hskip12pt2\hh\beta\odot\hs\xi=\beta\otimes\hs\xi+\xi\otimes\beta\hh,\\
\mathrm{b)}&[d\hh\beta\hh](\hn\eu,\ew)\,
=\,\hs d_\eu\hs[\hh\beta(\ew)]\,-\,\hs d_\ew\hs[\hh\beta(\hn\eu)]\,
-\,\beta([\eu,\ew])\hh,\\
\mathrm{c)}&[d\hh\beta\hh](\hn\eu,\ew)\,
=\,[\nabla_{\hskip-2.2pt\eu}\beta\hh](\ew)\,
-\,[\nabla_{\hskip-2.2pt\ew}\beta\hh](\hn\eu)\hh.
\end{array}
\end{equation}
Since $\,\lie_\ev\nh=d\hs\imath_\ev\nh+\imath_\ev d\,$ for the Lie derivative 
$\,\lie_\ev$ acting on differential forms, it follows that
\begin{equation}\label{lvz}
\lie_\ev\zeta\,\,=\,\,d\hs[\zeta(\hn\ev,\,\cdot\,)]
\end{equation}
whenever $\,\zeta\,$ is a $\,2$-form on a surface and $\,\ev\,$ is a 
tangent vector field.

Suppose now that $\,\spce\,$ is a 
fi\-\hbox{nite\hh-}\hskip0ptdi\-men\-sion\-al real vector space.
\begin{enumerate}
  \def\theenumi{{\rm\romen{enumi}}}
\item[(i)] Whenever a subspace $\,\spce\hh'$ of $\,\spce\,$ contains the image 
of an en\-do\-mor\-phism $\,\ed\,$ of $\,\spce$, the trace of 
$\,\ed:\spce\to\spce\,$ is obviously equal to the trace of the restriction 
$\,\ed:\spce\hh'\nnh\to\spce\hh'\nnh$.
\item[(ii)] Given an $\,m$-form $\,\zeta\in[\spce^*]^{\wedge m}\nnh$, where 
$\,m=\dim\hskip2pt\spce$, and any en\-do\-mor\-phism $\,\ed\,$ of $\,\spce$, 
the sum $\,\zeta(\ed\ev_1,\ev_2,\dots,\ev_m)
+\zeta(\hn\ev_1,\ed\ev_2,\ev_3,\dots,\ev_m)+\ldots
+\zeta(\hn\ev_1,\ev_2,\dots,\ev_{m-1},\ed\ev_m)\,$ equals 
$\,\zeta(\hn\ev_1,\ev_2,\dots,\ev_m)\,\mathrm{tr}\,\ed$, for any 
$\,\ev_1,\dots,\ev_m\in\spce$. One sees this using the matrix of $\,\ed\,$ in 
the basis $\,\ev_1,\dots,\ev_m$, if $\,\ev_1,\dots,\ev_m$ are linearly 
independent, and noting that both sides vanish for reasons of 
skew-sym\-me\-try, if $\,\ev_1,\dots,\ev_m$ are linearly dependent.
\item[(iii)] If $\,\dim\hskip1.6pt\spce=2\,$ and a tri\-lin\-e\-ar mapping 
$\,(\hn\ev,\ev\hh'\nnh,\ev\hh''\hh)\mapsto\chi(\hn\ev,\ev\hh'\nnh,\ev\hh'')\,$ 
from $\,\spce\,$ into any vector space is skew-sym\-met\-ric in 
$\,\ev\hh'\nnh,\ev\hh''\nnh$, then $\,\chi(\hn\ev,\ev\hh'\nnh,\ev\hh'')\,$ 
summed cyclically over $\,\ev,\ev\hh'\nnh,\ev\hh''$ yields $\,0$. In fact, the 
cyclic sum depends on $\,\ev,\ev\hh'$ and $\,\ev\hh''$ 
skew-sym\-met\-ri\-cal\-ly, so that it vanishes as $\,\spce\,$ is 
\hbox{two\hh-}\hskip0ptdi\-men\-sion\-al.
\end{enumerate}
It is well-known (see, e.g., 
\cite[Proposition~37.1(i) on p.\ 638]{dillen-verstraelen}) that any null 
\hbox{two\hh-}\hskip0ptdi\-men\-sion\-al subspace $\,\plne\,$ in a 
pseu\-d\hbox{o\hs-}\hskip0ptEuclid\-e\-an $\,4$-space $\,\spce\,$ of the 
neutral signature $\,(\mmpp)\,$ naturally distinguishes an orientation of 
$\,\spce$, namely, the one which, for some/\nnh any basis $\,u,v\,$ of 
$\,\plne$, makes the bi\-vec\-tor $\,u\wedge v\,$ self-du\-al. This makes it 
meaningful to say that a null distribution of dimension $\,2\,$ on a 
pseu\-d\hbox{o\hs-}\hskip0ptRiem\-ann\-i\-an 
\hbox{four\hh-}\hskip0ptman\-i\-fold $\,(\ym,\gm)\,$ of the neutral metric 
signature is, or is not, {\it compatible\/} with a prescribed orientation of 
$\,\ym$.
\begin{remark}\label{liebr}Let $\,\mathcal{V}\,$ be an integrable distribution 
on a manifold $\,\ym$. The maximal integral manifolds of $\,\mathcal{V}\hs$ 
will be simply referred to as the {\it leaves of\/} $\,\mathcal{V}\nnh$. (They 
are the leaves of the foliation on $\,\ym$, the tangent bundle of which is 
$\,\mathcal{V}$.) We will also speak of {\it sections of\/} 
$\,\mathcal{V}\nnh$, treating $\,\mathcal{V}\,$ as a vector sub\-bun\-dle of 
$\,\tm$. Finally, by a $\,\mathcal{V}\nh${\it-pro\-ject\-a\-ble local vector 
field\/} in $\,\ym\,$ we will mean any vector field $\,w\,$ defined on a 
nonempty open set $\,\,U\subset\ym\,$ and such that, whenever $\,v\,$ is a 
section of $\,\mathcal{V}\,$ defined on $\,\,U\nh$, so is $\,[w,v]$. If, in 
addition,% $\,\,U\,$ and $\,\mathcal{V}\hs$ satisfy the condition
\begin{equation}\label{bdl}
\begin{array}{l}
\mathrm{the\ \hskip.94ptleaves\ \hskip.94ptof\ 
\hskip.94pt}\,\mathcal{V}\hs\mathrm{\ \hskip.94ptrestricted\ \hskip.94ptto\ 
\hskip.94pt}\,U\mathrm{\ \hskip.94ptare\ \hskip.94ptall\ 
\hskip.94ptcontractible\ \hskip.94ptand\ \hskip.94ptconstitute}\\
\mathrm{the\ fibres\ of\ a\ 
bundle\ projection\ }\,\,
\pi:U\to\bs\,\,\mathrm{\ over\ some\ manifold\ }\,
\bs\hh,
\end{array}
\end{equation}
then $\,\mathcal{V}\nh$-pro\-ject\-a\-bil\-i\-ty of a vector field $\,w\,$ 
defined on $\,\,U\,$ is equivalent to its $\,\pi$-pro\-ject\-a\-bil\-i\-ty. 
(This is easily seen in suitable local coordinates.)
\end{remark}

\section{The Co\-daz\-zi and Kil\-ling operators}\label{cako}
\setcounter{equation}{0}
By a $\,k${\it-ten\-sor\/} on a manifold $\,\bs\,$ we always mean a $\,k\,$ 
times covariant tensor field on $\,\bs$. For instance, the curvature 
$\,4\hh$-ten\-sor $\,\br\,$ of a pseu\-d\hbox{o\hs-}\hskip0ptRiem\-ann\-i\-an 
manifold $\,(\ym,\gm)\,$ is characterized by 
$\,\br\hh(\hn\eu,\ev,\eu\hh'\nnh,\ev\hh'\hh)
=\gm(\br\hh(\hn\eu,\ev)\hh\eu\hh'\nnh,\ev\hh'\hh)$, where 
$\,\eu,\ev,\eu\hh'\nnh,\ev\hh'$ are any tangent vector fields and $\,\br\,$ on 
the right-hand side is defined as in (\ref{cur}) for the Le\-vi-Ci\-vi\-ta 
connection $\,\bna\,$ of $\,\gm$.

A connection $\,\nabla\hs$ on a manifold $\,\bs\,$ gives rise to two 
first-or\-der linear differential operators that will repeatedly appear in our 
discussion. One is the {\it Co\-daz\-zi operator\/} $\,d\hh^\nabla\nnh$, 
sending each symmetric $\,2$-ten\-sor $\,\tau\,$ on $\,\bs\,$ to the 
$\,3\hh$-ten\-sor equal to twice the skew-sym\-met\-ri\-za\-tion of the 
\hbox{$\,\nabla\nnh$-\hh}\hskip0ptco\-var\-i\-ant derivative of $\,\tau\,$ in 
the first two arguments. The other is the {\it Kil\-ling operator\/} $\,\ki$, 
which sends each $\,1$-form $\,\xi\,$ on $\,\bs\,$ to the symmetric 
$\,2$-ten\-sor obtained by sym\-met\-riz\-ing the 
\hbox{$\,\nabla\nnh$-\hh}\hskip0ptco\-var\-i\-ant derivative of $\,\xi$. 
Explicitly, for any tangent vector fields $\,\eu,\ev$,
\begin{equation}\label{cod}
\mathrm{a)}\hskip6pt[\hh d\hh^\nabla\hskip-2pt\tau](\hn\eu,\ev,\,\cdot\,)
=[\nsu\tau](\hn\ev,\,\cdot\,)-[\nsv\tau](\hn\eu,\,\cdot\,)\hh,\hskip12pt
\mathrm{b)}\hskip6pt2\hh[\hh\ki\hh\xi](\hn\eu,\ev)
=[\nsu\xi](\hn\ev)+[\nsv\xi](\hn\eu)\hh.
\end{equation}
We denote by $\,\mathrm{Ker}\,\ki\,$ the space of all 
$\,C^\infty\nnh$-dif\-fer\-en\-ti\-a\-ble $\,1$-forms $\,\xi\,$ with 
$\,\ki\hh\xi=0$.% By (\ref{bwa}.a),
%\begin{equation}\label{lps}
%\ki\hh(\psi\xi)=\psi\ki\hh\xi+\xi\odot d\psi\hskip7pt\mathrm{for\ 
%}\hs1\mbox{-}\mathrm{forms\ }\,\xi\,\mathrm{\ and\ functions\ 
%}\,\psi\,\mathrm{\ on\ }\,\bs\hs.
%\end{equation}

If $\,\nabla\hn\,$ is tor\-sion\-free, the second covariant derivative 
$\,\nabla\nabla\xi\,$ of any $\,1$-form $\,\xi\,$ on $\,\bs\,$ and the tensor 
field $\,\tau=\ki\hh\xi\,$ satisfy the following well-known relation, 
immediate from the Ric\-ci and Bian\-chi identities, cf.\ \cite[the bottom of 
p.\ 572]{derdzinski-roter}, in which both sides are $\,2$-ten\-sors:
\begin{equation}\label{nvn}
\nsv\nabla\xi\,=\,-\hs\xi[R(\,\cdot\,,\,\cdot\,)\hh\ev]\,
+\,[\hh d\hh^\nabla\hskip-2pt\tau](\,\cdot\,,\,\cdot\,,v)\,+\,\nsv\tau
\hskip10pt\mathrm{whenever}\hskip5pt\tau=\ki\hh\xi\hh.
\end{equation}
Here $\,\ev\,$ stands for an arbitrary vector field, $\,d\hh^\nabla\nnh$ 
denotes the Co\-daz\-zi operator with (\ref{cod}.a), and $\,\nabla\xi\,$ is 
treated as a $\,2$-ten\-sor acting on vector fields $\,\eu,\ev\,$ by 
$\,[\nabla\xi](\hn\eu,\ev)=[\nsu\xi](\hn\ev)$. In coordinates, (\ref{nvn}) 
reads 
$\,\xi_{\hh j,kl}=-R_{kjl}{}^s\hs\xi_s+\tau_{jl,k}-\tau_{kl,j}+\tau_{kj,l}$, 
with $\,\tau_{kj}\nh=(\xi_{\hh j,k}\nh+\xi_{\hh k,j})/2$.

\section{Riemann extensions}\label{riex}
\setcounter{equation}{0}
Let $\,\ym=\tab\,$ be the total space of the cotangent bundle of a manifold 
$\,\bs\,$ carrying a tor\-sion\-free connection $\,\nabla\nnh$, and let 
$\,\pi:\tab\to\bs\,$ denote the bundle projection. Following Pat\-ter\-son and 
Walker \cite[p.\ 26]{patterson-walker}, by a {\it Riemann extension metric\/} 
associated with $\,\nabla\hn\,$ we mean any $\,2$-ten\-sor on $\,\ym\,$ having 
the form $\,\gm\,=\,\gm\nnh^\nabla\nnh+\hs2\hh\pi^*\nnh\tau$, where $\,\tau\,$ 
is a symmetric $\,2$-ten\-sor on $\,\bs$, and $\,\,\gm\nnh^\nabla$ stands for 
the pseu\-d\hbox{o\hs-}\hskip0ptRiem\-ann\-i\-an metric on $\,\tab\,$ defined 
by requiring that all vertical and all $\,\nabla$-hor\-i\-zon\-tal vectors be 
$\,\gm\nnh^\nabla\nnh$-null, while 
$\,\gm_x\hskip-4.3pt^\nabla\hskip-2pt(\xi,w)=\xi(d\pmb_xw)\,$ for any 
$\,x\in\ym=\tab$, any $\,w\in T\hskip-2pt_x\hskip-.9pt\ym$, and any vertical 
vector $\,\xi\in\kerd\pi_x=\tayb$, with $\,y=\pi(x)$. Such $\,\gm\,$ is 
clearly a pseu\-d\hbox{o\hs-}\hskip0ptRiem\-ann\-i\-an metric of the neutral 
signature. In local coordinates $\,y^{\hs j}\nnh,q_j$ for $\,\tab\,$ arising 
from a coordinate system $\,y^{\hs j}$ for $\,\bs\,$ in which $\,\nabla\hn\,$ 
has the components $\,\vg_{\!kl}^j$,
\begin{equation}\label{rex}
\gm\,=\,2\hs dq_j\odot\hskip2ptdy^{\hs j}+2\hs(\hn\tau_{kl}
-q_j\vg_{\!kl}^j\hh)\,dy^k\odot\hs dy^{\hh l},
\end{equation}
cf.\ \cite[formula (28)]{patterson-walker}, with the symmetric multiplication 
$\,\odot\,$ given by (\ref{bwa}.a).

We use the term `Riemann extension' narrowly, as dictated by the specific 
applications described in Section~\ref{tstd}. Wider classes of Riemann 
extensions have been discussed by many authors, for instance, in 
\cite{patterson-walker}, \cite{afifi} and, most recently, 
\cite{calvino-louzao-garcia-rio-gilkey-vazquez-lorenzo}.

Part (a) of the following lemma goes back to Patter\-son and Walker 
\cite[\S8]{patterson-walker}.
\begin{lemma}\label{kxpgn}Let\/ $\,\bs,\nabla\hn\,$ and\/ $\,\tau\,$ have the 
properties listed above.
\begin{enumerate}
  \def\theenumi{{\rm\alph{enumi}}}
\item[(a)] Any\/ $\,1$-form\/ $\,\xi\,$ on\/ $\,\bs\,$ gives rise to a 
dif\-feo\-mor\-phism\/ $\,\tx:\ym\to\ym$, acting as the translation by\/ 
$\,\xi_y$ in the fibre $\,\tayb\,$ of\/ $\,\ym=\tab$, for every\/ $\,y\in\bs$, 
and the\/ $\,\tx$-pull\-back of\/ 
$\,\gm\nnh^\nabla\nnh+\hs2\hh\pi^*\nnh\tau\,$ is\/ 
$\,\gm\nnh^\nabla\nnh+\hs2\hh\pi^*\nh(\tau+\ki\hh\xi)$, with\/ $\,\ki\,$ as 
in\/ {\rm(\ref{cod}.b)}.
\item[(b)] In particular, for any\/ $\,1$-form\/ $\,\xi\,$ on\/ $\,\bs$, the 
metrics\/ $\,\gm\nnh^\nabla\nnh+\hs2\hh\pi^*\nnh\tau\,$ and\/ 
$\,\gm\nnh^\nabla\nnh+\hs2\hh\pi^*\nh(\tau+\ki\hh\xi)\,$ on\/ $\,\ym\,$ are 
isometric to each other.
\item[(c)] If\/ $\,\vt:\tab\to\tab\,$ is a vec\-tor-bun\-dle isomorphism 
and the\/ $\,\vt$-pull\-back of\/ $\,\gm\nnh^\nabla\nnh+\hs2\hh\pi^*\nnh\tau$, 
restricted to some nonempty open set in\/ $\,\tab\,$ interseting\/ $\,\tayb\,$ 
for each\/ $\,y\in\bs$, coincides with\/ 
$\,\gm\nnh^\nabla\nnh+\hs2\hh\pi^*\nnh\tau\hh'$ for some symmetric\/ 
$\,2$-ten\-sor\/ $\,\tau\hh'$ on\/ $\,\bs$, then\/ $\,\vt=\mathrm{Id}\,$ and\/ 
$\,\tau\hh'\nnh=\tau$.
\end{enumerate}
\end{lemma}
\begin{proof}This is immediate if one replaces the ingredients 
$\,q_j,y^{\hs j},\tau_{kl},\vg_{\!kl}^j$ of formula (\ref{rex}) with their 
pull\-backs under $\,\tx$ (or, under $\,\vt$), that is, with 
$\,q_j\nh+\xi_{\hs j}$ (or, $\,\vt_{\hskip-1.5ptj}^{\hs l}q_l^{\phantom i}$), 
$\,y^{\hs j},\tau_{kl}$ and $\,\vg_{\!kl}^j$.
\end{proof}
To avoid confusion caused by the presence of both upper and lower indices in 
$\,y^{\hs j}\nnh,q_j$, we fix a nonsingular square matrix 
$\,[\gm_{j\lambda}]\,$ of constants, and replace $\,y^{\hs j}\nnh,q_j$ with 
the new coordinates $\,y^{\hs j}\nnh,x^\lambda\nnh$, related to the old ones 
by $\,q_j=\gm_{j\lambda}x^\lambda\nnh$. Keeping the Greek indices 
$\,\lambda,\mu\,$ always separate from the Roman indices $\,j,k,l,p,q,s$, even 
though both sets of indices range over $\,\{1,\dots,\dim\bs\}$, we easily 
verify that the components of $\,\gm=\gm\nnh^\nabla\nnh+\hs2\hh\pi^*\nnh\tau$, 
its reciprocal metric, the Le\-vi-Ci\-vi\-ta connection $\,\bna\,$ of $\,\gm$, 
and its curvature $\,4\hh$-ten\-sor $\,\br$, in the coordinates 
$\,y^{\hs j}\nnh,x^\lambda\nnh$, are given by
\begin{equation}\label{gjk}
\begin{array}{l}
\gm_{jk}=\hs2\hs(\hn\tau_{jk}
-\gm_{s\lambda}x^\lambda\vg_{\hskip-3ptjk}^s)\hs,\hskip12pt\gm_{\lambda\mu}
=\hs0,\hskip12pt\gm_{\lambda j}
=\hs\gm_{j\lambda}\hskip7pt\mathrm{(the\ fixed\ constants),}\\
\gm^{jk}=\hs0,\hskip12pt[\gm^{j\lambda}]=\hs[\gm_{j\lambda}]^{-1},\hskip12pt
\gm^{\lambda\mu}=\hs-\hs\gm^{j\lambda}\gm^{\hh k\mu}\gm_{jk}\hh,
\phantom{1^{1^1}}\\
\bvg_{\!\lambda\mu}^{\hh\md}=\hs0\hh,\hskip12pt\bvg_{\!\lambda\md}^j=\hs0\hh,
\hskip12pt\bvg_{\hskip-3ptjk}^{\hs l}=\hs\vg_{\hskip-3ptjk}^{\hs l}\hh,
\hskip12pt\bvg_{\hskip-3ptj\lambda}^{\hs\mu}
=\hs-\gm_{s\lambda}\gm^{\hh k\mu}\vg_{\hskip-3ptjk}^s,\\
\gm_{k\mu}\bvg_{\!jl}^\mu=\hs\gm_{s\lambda}x^\lambda(R_{lkj}{}^s
-\partial_j\nnh\vg_{\!lk}^s
+\vg_{\!kp}^s\vg_{\hskip-3ptjl}^p+\vg_{\!lp}^s\vg_{\hskip-3ptjk}^p)
+\tau_{\hs lk,\hs j}+(d\hh^\nabla\hskip-2pt\tau)_{lkj}\hh,\\
\br_{\lambda\mu\hs\md\md}=\hs\br_{\lambda\md\mu\md}=\hs0\hh,
\hskip12pt\br_{jkl\lambda}=\hs\gm_{s\lambda}R_{jkl}{}^s,\phantom{_{j_j}}\\
\br_{jklp}=\hs\gm_{s\lambda}x^\lambda(R_{lpj}{}^s{}_{\nnh,\hs k}
-R_{lpk}{}^s{}_{\nnh,\hs j}
+\vg_{\hskip-3ptjq}^sR_{lpk}{}^q
-\vg_{\!kq}^sR_{lpj}{}^q
+\vg_{\!lq}^sR_{jkp}{}^q
-\vg_{\!pq}^sR_{jkl}{}^q)\\
\phantom{\br_{jklp}\hskip1.3pt}+\,R_{kjp}{}^s\tau_{sl}-R_{kjl}{}^s\tau_{sp}
+(d\hh^\nabla\hskip-2pt\tau)_{lpj,\hs k}
-(d\hh^\nabla\hskip-2pt\tau)_{lpk,\hs j}\hh.\phantom{1^{1^1}}
\end{array}
\end{equation}
Here the dots stand for indices of either kind, $\,\vg_{\hskip-3ptjk}^{\hs l}$ 
and $\,R_{jkl}{}^s$ for the components of $\,\nabla\hn\,$ and its curvature 
tensor $\,R\,$ in the coordinates $\,y^{\hs j}\nnh$, the commas for 
\hbox{$\,\nabla\nnh$-\hh}\hskip0ptco\-var\-i\-ant derivatives, 
$\,d\hh^\nabla\nnh$ for the Co\-daz\-zi operator of $\,\nabla\hn\,$ (see 
(\ref{cod}.a)), and $\,[\gm_{j\lambda}]^{-1}$ for the matrix inverse of 
$\,[\gm_{j\lambda}]$.

One easily obtains the following conclusion, due to Patter\-son and Walker 
\cite[p.\ 26]{patterson-walker}:
\begin{lemma}\label{rcflt}For\/ 
$\,\gm=\gm\nnh^\nabla\hn+\hs2\hh\pi^*\nnh\tau\,$ defined as above, the 
following three conditions are equivalent\/{\rm:}
\begin{enumerate}
  \def\theenumi{{\rm\roman{enumi}}}
\item[(i)] $\gm\,$ is an Ein\-stein metric,
\item[(ii)] $\gm\,$ is Ric\-ci-flat,
\item[(iii)] the Ric\-ci tensor\/ $\,\rho\,$ of\/ $\,\nabla\hn\,$ is 
\hbox{skew\hs-}\hskip0ptsym\-met\-ric at every point.
\end{enumerate}
\end{lemma}
\begin{proof}By (\ref{gjk}), the Ric\-ci tensor $\,\brh\,$ of $\,\gm\,$ has 
the components $\,\brh_{jk}=\gm^{s\lambda}\br_{jsk\lambda}
+\gm^{\lambda s}\br_{j\lambda ks}=\rho_{jk}+\rho_{kj}$ and 
$\,\brh_{j\lambda}=\brh_{\lambda\mu}=0$, and the so scalar curvature of 
$\,\gm\,$ is $\,\gm^{jk}\brh_{jk}=0$.
\end{proof}
The relations 
$\,\gm_{\lambda\mu}=\bvg_{\!\lambda\mu}^{\hh\md}=\bvg_{\!\lambda\md}^j
=\br_{\lambda\md\mu\md}=0\,$ in (\ref{gjk}) state that the vertical 
distribution $\,\mathcal{V}=\kerd\pi\,$ is $\,\gm$-null, $\,\gm$-par\-al\-lel, 
and satisfies the following curvature condition (cf.\ Remark~\ref{liebr}):
\begin{equation}\label{ccn}
\br\hh(\hn\eu,\,\cdot\,,\ev,\,\cdot\,)\,=\,0\hskip7pt\mathrm{for\ any\ sections\ 
}\,\eu,\ev\,\mathrm{\ of\ }\,\,\mathcal{V}\nh.
\end{equation}
The properties just listed form an intrinsic local characterization of Riemann 
extension metrics, which is a result of Afifi \cite{afifi}, stated below as 
Theorem~\ref{riext}. We preceed it by a more general discussion, beginning 
with a lemma phrased in  the language of Remark~\ref{liebr}:
\begin{lemma}\label{trvcn}Let an \hbox{$\,m$\hh-}\hskip0ptdi\-men\-sion\-al 
null parallel distribution\/ $\,\mathcal{V}\,$ on a 
pseu\-\hbox{do\hs-\hn}\hskip0ptRiem\-ann\-i\-an manifold\/ $\,(\ym,\gm)\,$ 
with\/ $\,\dim\ym=2m\,$ satisfy condition\/ {\rm(\ref{bdl})} for\/ $\,\,U=\ym$.
\begin{enumerate}
  \def\theenumi{{\rm\alph{enumi}}}
\item[(a)] The requirement that $\,\pi^*\xi=g(\hn\ev,\,\cdot\,)\,$ defines a 
natural bijective correspondence between sections\/ $\,\ev\,$ of\/ 
$\,\mathcal{V}\,$ parallel along $\,\,\mathcal{V}\hs$ and sections\/ 
$\,\xi\nh\,$ of\/ $\,\tab$.
\item[(b)] If\/ {\rm(\ref{ccn})} holds as well, then there exists a unique 
tor\-sion\-free connection\/ $\,\nabla\hn\,$ on\/ $\,\bs\,$ such that, for 
any\/ $\,\pi$-pro\-ject\-a\-ble vector fields\/ $\,\eu,\eu\hh'$ on\/ $\,\ym$, 
the covariant derivative\/ $\,\bna_{\hskip-2pt\eu}\hn\eu\hh'\nnh$, relative to 
the Le\-vi-Ci\-vi\-ta connection\/ $\,\bna\,$ of\/ $\,g$, is\/ 
$\,\pi$-pro\-ject\-a\-ble onto the vector field\/ 
$\,\nabla_{\hskip-2pt\ew}\hn\ew\hh'$ on\/ $\,\bs$, where\/ $\,\ew,\ew\hh'$ 
are the\/ $\,\pi$-im\-ages of\/ $\,\eu\,$ and\/ $\,\eu\hh'\nnh$.
\end{enumerate}
\end{lemma}
\begin{proof}The pull\-back $\,\pi^*\xi\,$ of any given section $\,\xi\,$ of 
$\,\tab\,$ determines $\,\xi\,$ uniquely (since $\,\pi\,$ is a submersion), 
and equals $\,\gm(\hn\ev,\,\cdot\,)\,$ for a unique vector field $\,\ev\,$ on 
$\,\ym$, which defines an injective assignment $\,\xi\mapsto\ev$. For 
sections $\,\eu\,$ of $\,\mathcal{V}=\kerd\pi\,$ we have 
$\,(\pi^*\xi)(\hn\eu)=0\,$ (so that $\,\ev\,$ is a section of 
$\,\mathcal{V}^\perp\nnh=\mathcal{V}$) and 
$\,\gm(\nsu\ev,\,\cdot\,)=\nsu(\pi^*\xi)=0\,$ (which gives $\,\nsu\ev=0$), as 
one sees noting that, since $\,\mathcal{V}\,$ is parallel, 
$\,[\nsu(\pi^*\xi)](\ew)=-\hs(\pi^*\xi)(\nsu\ew)=
-\hs(\pi^*\xi)(\nsw\eu)=0\,$ for any $\,\pi$-pro\-ject\-a\-ble vector 
field $\,\ew\,$ on $\,\ym$, in view of the Lei\-bniz rule and 
Remark~\ref{liebr}. Finally, the assignment $\,\xi\mapsto\ev\,$ is 
surjective: for a section $\,\ev\,$ of $\,\mathcal{V}\,$ parallel along 
$\,\mathcal{V}\nh$, we define a section $\,\xi\,$ of $\,\tab\,$ by 
$\,\xi(\tw)=\gm(v,w)$, for any vector field $\,\tw\,$ on $\,\bs$, where 
$\,w\,$ is any $\,\pi$-pro\-ject\-a\-ble vector field on $\,\ym\,$ with the 
$\,\pi$-im\-age $\,\tw$. Since $\,\mathcal{V}=\kerd\pi\,$ is a $\,\gm$-null 
distribution, $\,\xi(\tw)\,$ does not depend on the choice of $\,w$. Also, 
for any section $\,\eu\,$ of $\,\mathcal{V}\nh$, we have $\,\nsu\ev=0$, and so 
$\,d_u[\hs\xi(\tw)]=d_u[\hh \gm(v,w)]=\gm(v,\nsu\ew)=\gm(v,\nsw\eu)\,$ in view 
of Remark~\ref{liebr}, which in turn vanishes, as $\,\mathcal{V}\,$ is 
parallel and null. Therefore, $\,\xi(\tw)\,$ may be treated as a function 
$\,\bs\to\bbR\hs$, and so $\,\xi\,$ is well defined. This proves (a).

Assuming (\ref{ccn}), let us fix $\,\eu,\eu\hh'$ as in (b) and a section 
$\,\ev\,$ of $\,\mathcal{V}\nh$. Thus, 
$\,\bna_{\hskip-2.2pt\ev}\hn\eu\,\approx\,0$, where $\,\approx\,$ 
means {\it differ by a section of\/} $\,\mathcal{V}\nh$. (In fact, 
$\,[\ev,\eu]\,\approx\,0$, cf.\ Remark~\ref{liebr}, and 
$\,\bna_{\hskip-2pt\eu}\hn\ev\,\approx\,0\,$ since $\,\mathcal{V}\,$ is
parallel.) Also, $\,\br\hh(\hn\ev,\eu)\hh\eu\hh'$ is, by (\ref{ccn}), a 
section of $\,\mathcal{V}^\perp\nnh=\mathcal{V}\nh$. This, along with 
(\ref{cur}) and Remark~\ref{liebr}, gives 
$\,[\ev,\bna_{\hskip-2pt\eu}\hn\eu\hh'\hh]\,
\approx\,\bna_{\hskip-2.2pt\ev}\bna_{\hskip-2pt\eu}\hn\eu\hh'\hs\approx\,0$, 
so that $\,\bna_{\hskip-2pt\eu}\hn\eu\hh'$ is $\,\pi$-pro\-ject\-a\-ble, and 
(b) follows.
\end{proof}
For $\,\mathcal{V}\,$ as in Lemma~\ref{trvcn}, assertion (b) describes a 
transversal connection $\,\nabla\nnh$, in the sense of Molino \cite{molino}, 
for the foliation on $\,\ym\,$ tangent to $\,\mathcal{V}\nh$. (See also 
Section~\ref{trrs}.) We will refer to $\,\nabla\hn\,$ as the {\it transversal 
connection\/} on $\,\bs$, corresponding to $\,g\,$ and $\,\mathcal{V}\nh$.

Patter\-son and Walker \cite[p.\ 26]{patterson-walker} were the first to 
observe that, in the case where $\,\gm\,$ is a Riemann extension metric 
$\,\gm\nnh^\nabla\hn+\hs2\hh\pi^*\nnh\tau\,$ on $\,\tab$, both $\,\bs\,$ and 
$\,\nabla\hn\,$ (though not $\,\tau$) are, locally, determined just by 
$\,\gm\,$ and $\,\mathcal{V}=\kerd\pi$.

Using (\ref{gjk}), we now describe vertical Kil\-ling fields in Riemann 
extensions, cf.\  \cite{toomanian}:
\begin{lemma}\label{vrtki}Under the assumptions of Lemma\/~{\rm\ref{trvcn}}, a 
section\/ $\,\ev\,$ of\/ $\,\mathcal{V}\,$ is a Kil\-ling field for\/ 
$\,(\ym,\gm)\,$ if and only if\/ $\,\ki\hh\xi=0$, where\/ $\,\xi\,$ 
corresponds to\/ $\,\ev\,$ as in Lemma\/~{\rm\ref{trvcn}(a)} and\/ 
$\,\ki\,$ is the Kil\-ling operator, with\/ {\rm(\ref{cod}.b)}, of the 
connection\/ $\,\nabla\hn\,$ described in Lemma\/~{\rm\ref{trvcn}(b)}.
\end{lemma}
In fact, since $\,\ev^{\hs j}\nh=0\,$ and 
$\,\xi_{\hs j}\nh=\gm_{j\lambda}\ev^\lambda\nnh$, (\ref{gjk}) gives 
$\,\ev_j\nh=\xi_{\hs j}$ and $\,\ev_\lambda\nh=0\,$ (with 
$\,\gm$-low\-er\-ed indices), which implies the Lie-de\-riv\-a\-tive 
relation $\,\lie_\ev\gm=2\hh\pi^*\nh(\ki\hh\xi)$.

As mentioned earlier, the following intrinsic local characterization of 
Riemann extension metrics is a special case of a result of Afifi \cite{afifi}.
\begin{theorem}\label{riext}Let an\/ 
\hbox{$\,m$\hh-}\hskip0ptdi\-men\-sion\-al null parallel distribution\/ 
$\,\mathcal{V}\,$ on a pseu\-d\hbox{o\hs-}\hskip0ptRiem\-ann\-i\-an manifold\/ 
$\,(\ym,\gm)\,$ with\/ $\,\dim\ym=2m\,$ satisfy the curvature condition\/ 
{\rm(\ref{ccn})}. Then, for every point\/ $\,x\in\ym$, there exist a 
manifold\/ $\,\bs\,$ of dimension\/ $\,m$, a tor\-sion\-free connection 
$\,\nabla\hs$ on $\,\bs$, a symmetric $\,2$-ten\-sor\/ $\,\tau\,$ on\/ 
$\,\bs$, and a dif\-feo\-mor\-phism of a neighborhood of\/ $\,x\,$ onto an 
open subset of\/ $\,\tab$, which sends
\begin{enumerate}
  \def\theenumi{{\rm\roman{enumi}}}
\item[(i)] $\gm\,$ to the Riemann extension metric\/ 
$\,\gm\nnh^\nabla\hn+\hs2\hh\pi^*\nnh\tau$,
\item[(ii)] $\mathcal{V}\,$ to the vertical distribution\/ $\,\kerd\pi\hs$ of 
the bundle projection\/ $\,\pi:\tab\to\bs$,
\item[(iii)] the transversal connection described in 
Lemma\/~{\rm\ref{trvcn}(b)} to\/ $\,\nabla$.
\end{enumerate}
\end{theorem}
\begin{proof}By Lemma~\ref{trvcn}(a), the leaves of $\,\mathcal{V}\,$ 
coincide, in a neighborhood of any given point $\,x\in\ym$, with the fibres of 
an affine bundle which, through any fixed choice of a zero section, becomes 
identified with the vector bundle $\,\tab\,$ for a local leaf space $\,\bs\,$ 
of $\,\mathcal{V}\nh$. Let $\,\nabla\hn\,$ be the transversal connection on 
$\,\bs$, corresponding to $\,\gm\,$ and $\,\mathcal{V}\,$ as in 
Lemma~\ref{trvcn}(b). Clearly,
\begin{equation}\label{gmg}
(\gm-\gm\nnh^\nabla\hh)(\hn\ev,\,\cdot\,)\,=\,0\hskip12pt\mathrm{for\ every\ 
section\ }\,\,\ev\,\,\mathrm{of\ }\,\,\mathcal{V}.
\end{equation}
If $\,\bd\,$ is the Le\-vi-Ci\-vi\-ta connection of any 
pseu\-\hbox{do\hs-\hn}\hskip0ptRiem\-ann\-i\-an metric $\,h\,$ on a 
neighborhood of $\,x\,$ and $\,\ev,\ew\,$ are vector fields, 
$\,d_\ev[\hh h(\ew,\ew)]/2=h(\bd_\ev\hn\ew,\ew)
=h(\bd_\ew\ev,\ew)+h([\ev,\ew],\ew)$, and so 
\begin{equation}\label{dvh}
d_\ev[\hh h(\ew,\ew)]/2\,=\,d_\ew[\hh h(\hn\ev,\ew)]\,
-\,h(\hn\ev,\bd_\ew\ew)+h([\ev,\ew],\ew)\hs.
\end{equation}
Our two choices of $\,h\,$ are $\,h=\gm\,$ and $\,h=\gm\nnh^\nabla\nnh$. If 
$\,\ev\,$ is a section of $\,\mathcal{V}\,$ parallel along $\,\mathcal{V}\,$ 
and $\,\ew\,$ is $\,\mathcal{V}\nh$-pro\-ject\-a\-ble (Remark~\ref{liebr}), 
then each term on the right-hand side of (\ref{dvh}) is the same for 
$\,h=\gm\,$ as it is for $\,h=\gm\nnh^\nabla\nnh$. In the case of 
$\,d_\ew[\hh h(\hn\ev,\ew)]\,$ and $\,h([\ev,\ew],\ew)\,$ this is obvious from 
(\ref{gmg}) since, according to Remark~\ref{liebr}, $\,[\ev,\ew]\,$ is a 
section of $\,\mathcal{V}\nh$. On the other hand, the term 
$\,h(\hn\ev,\bd_\ew\ew)$, for either choice of $\,h$, equals 
$\,\xi(\nabla_{\hskip-2pt\pi\nh\ew}(\pi\hn\ew))$, where $\,\xi\,$ corresponds 
to $\,\ev\,$ as in Lemma~\ref{trvcn}(a). (That $\,\nabla\hn\,$ 
is also the transversal connection for $\,\gm\nnh^\nabla$ is immediate from 
the formula $\,\bvg_{\hskip-3ptjk}^{\hs l}=\vg_{\hskip-3ptjk}^{\hs l}$ in 
(\ref{gjk}) with $\,\tau=0$.)

Subtracting the two versions of (\ref{dvh}), for $\,h=\gm\,$ and 
$\,h=\gm\nnh^\nabla\nnh$, and using (\ref{gmg}), we see that 
$\,\gm-\gm\nnh^\nabla\nnh=2\hh\pi^*\nnh\tau\,$ for some symmetric 
$\,2$-ten\-sor $\,\tau\,$ on $\,\bs$, which completes the proof.
\end{proof}
The following lemma describes local isometries between two Riemann extension 
metrics, sending one vertical distribution onto the other. We use the symbol 
$\,\pi\,$ for both bundle projections $\,\tab\to\bs\,$ and 
$\,T\hskip.2pt^*\hskip-.9pt\plne\to\plne$, the meaning of $\,\tx\,$ is the 
same as in Lemma~\ref{kxpgn}(a), $\,\ki\,$ is the Kil\-ling operator for 
$\,\nabla\nnh$, given  by (\ref{cod}.b), while 
$\,F^*:T\hskip.2pt^*\hskip-.9pt\plne\to\tab\,$ denotes the 
dif\-feo\-mor\-phism induced by $\,F:\bs\to\plne$, and, at the same time, 
$\,F^*\nh\ts\,$ stands for the $\,F$-pull\-back of the $\,2$-ten\-sor $\,\ts$.
\begin{lemma}\label{ismtr}Suppose that a triple\/ $\,(\bs,\nabla\nnh,\tau)\,$ 
consists of a manifold\/ $\,\bs\,$ with a tor\-sion\-free connection\/ 
$\,\nabla\hn\,$ and a symmetric\/ $\,2$-ten\-sor\/ $\,\tau\,$ on\/ $\,\bs$. 
Let\/ $\,(\plne,\mathrm{D},\ts)\,$ be another such triple.
\begin{enumerate}
  \def\theenumi{{\rm\roman{enumi}}}
\item[(i)] For any dif\-feo\-mor\-phism $\,F:\bs\to\plne\,$ sending\/ 
$\,\nabla\hn\,$ to\/ $\,\mathrm{D}\,$ and any\/ $\,1$-form $\,\xi\,$ on\/ 
$\,\bs\,$ such that\/ $\,F^*\nh\ts=\tau+\ki\hh\xi$, the composite\/ 
$\,\varPhi=\tx\circ F^*$ is an isometry of\/ 
$\,(T\hskip.2pt^*\hskip-.9pt\plne,\gm^{\mathrm{D}}\hn
+\hs2\hh\pi^*\hn\ts)\,$ onto\/ 
$\,(\tab,\gm\nnh^\nabla\hn+\hs2\hh\pi^*\nnh\tau)\,$ sending one vertical 
distribution onto the other.
\item[(ii)] Conversely, if\/ $\,x\in T\hskip.2pt^*\hskip-.9pt\plne\,$ and\/ 
$\,\varPhi\,$ an isometry of a connected neighborhood of\/ $\,x\,$ in\/ 
$\,(T\hskip.2pt^*\hskip-.9pt\plne,\gm^{\mathrm{D}}\hn
+\hs2\hh\pi^*\hn\ts)\,$ onto an open sub\-man\-i\-fold of\/ 
$\,(\tab,\gm\nnh^\nabla\hn+\hs2\hh\pi^*\nnh\tau)$, sending one vertical 
distribution onto the other, then\/ $\,\varPhi\hs$ restricted to some 
neighborhood of\/ $\,x\,$ equals\/ $\,\tx\circ F^*\nnh$, where\/ 
$\,F:\bs\hh'\nnh\to\plne\hh'$ and\/ $\,\xi\,$ are defined on\/ $\,\bs\hh'$ 
and have the properties listed in\/ {\rm(i)}, for some open 
sub\-man\-i\-folds\/ $\,\bs\hh'\nnh\subset\bs\,$ and\/ 
$\,\plne\hh'\nnh\subset\plne$.
\end{enumerate}
\end{lemma}
\begin{proof}If $\,F^*\mathrm{D}=\nabla\nnh$, the $\,F^*\nnh$-pull\-backs of 
$\,\gm\nnh^\nabla$ and $\,\pi^*\nnh F^*\nh\ts\,$ are $\,\gm^{\mathrm{D}}$ and 
$\,\pi^*\hn\ts$. Thus, (i) is immediate from Lemma~\ref{kxpgn}(a).

Conversely, given $\,x\,$ and $\,\varPhi\,$ as in (ii), we may assume 
(\ref{bdl}) for suitable neighborhoods $\,\,U\,$ of $\,x\,$ in 
$\,T\hskip.2pt^*\hskip-.9pt\plne\,$ and $\,\,U\hh'$ of $\,\varPhi(x)\,$ in 
$\,\tab\,$ with some base manifolds which are open connected sets 
$\,\plne\hh'\nnh\subset\plne\,$ and $\,\bs\hh'\nnh\subset\bs$, in such a way 
that $\,\pi\nh\circ\varPhi=F^{-1}\nnh\circ\pi\,$ for some dif\-feo\-mor\-phism 
$\,F:\bs\hh'\nnh\to\plne\hh'\nnh$. 

In view of Lemma~\ref{trvcn} and the formula 
$\,\bvg_{\hskip-3ptjk}^{\hs l}=\vg_{\hskip-3ptjk}^{\hs l}$ in (\ref{gjk}), the 
af\-fine structures of the fibres in $\,T\hskip.2pt^*\hskip-.9pt\plne\,$ and 
$\,\tab$, as well as the connections $\,\mathrm{D}\,$ and $\,\nabla\hn\,$ are 
local geometric invariants associated with the metrics 
$\,\gm^{\mathrm{D}}\hn+\hs2\hh\pi^*\hn\ts$, 
$\,\gm\nnh^\nabla\hn+\hs2\hh\pi^*\nnh\tau\,$ and the respective vertical 
distributions. Therefore, $\,F^*\mathrm{D}=\nabla\hn\,$ and, on a neighborhood 
of $\,x$, we have $\,\varPhi=\vt\circ\tx\circ F^*$ for some vec\-tor-bun\-dle 
isomorphism $\,\vt:\tab\hh'\nnh\to\tab\hh'$ and some $\,1$-form $\,\xi\,$ 
on $\,\bs\hh'\nnh$. Now $\,F^*$ pushes 
$\,\gm^{\mathrm{D}}\hn+\hs2\hh\pi^*\hn\ts\,$ forward onto the metric 
$\,\gm\nnh^\nabla\hn+\hs2\hh\pi^*\hn F^*\nh\ts$, which, according to 
Lemma~\ref{kxpgn}(a), is pushed forward by $\,\tx\,$ onto 
$\,\gm\nnh^\nabla\hn+\hs2\hh\pi^*\nh(F^*\nh\ts-\ki\hh\xi)$. Since 
$\,\varPhi\,$ is an isometry, this last metric is the $\,\vt$-pull\-back of 
$\,\gm\nnh^\nabla\nnh+\hs2\hh\pi^*\nnh\tau$. By Lemma~\ref{kxpgn}(c), 
$\,\vt=\mathrm{Id}\,$ and $\,F^*\nh\ts-\ki\hh\xi=\tau$, which completes the 
proof.
\end{proof}

\section{RSTS connections}\label{rsts}
\setcounter{equation}{0}
By an {\it RSTS connection\/} we mean a `Ric\-ci skew-sym\-met\-ric 
tor\-sion\-free surface connection' or, more precisely, a tor\-sion\-free 
connection $\,\nabla\hn\,$ on a surface $\,\bs\,$ such that the Ric\-ci tensor of 
$\,\nabla\hn\,$ is skew-sym\-met\-ric at every point of $\,\bs$.

Wong \cite[Theorem~4.2]{wong} found a canonical coordinate form of RSTS 
connections. A simplified version of Wong's result can be phrased as follows.
\begin{theorem}\label{wongs}A tor\-sion\-free connection\/ $\,\nabla\hs$ 
on a surface\/ $\,\bs\,$ has skew-sym\-met\-ric Ric\-ci tensor if and only if, 
on some neighborhood of any point of $\hskip-8.5pt\phantom{^{1^{1^1}}}\bs$, 
there exist coordinates in which the component functions of 
$\hskip-8pt\phantom{^{1^{1^1}}}\nabla\hn\,$ are 
$\,\vg_{\hskip-2pt11}^1=-\hs\partial^{\phantom{j}}_1\nh\varphi$, 
$\,\vg_{\hskip-2pt22}^2=\hs\partial^{\phantom{j}}_{\hs2}\varphi\,$ for a 
function $\,\varphi$, and\/ $\,\vg_{\hskip-3ptjk}^{\hs l}=0\,$ unless 
$\,j=k=l$. The Ric\-ci tensor\/ $\,\rho\,$ of\/ $\,\nabla\hn\,$ then is given 
by\/ 
$\,\rho^{\phantom{j}}_{12}
=-\hs\partial^{\phantom.}_1\partial^{\phantom.}_{\hs2}\varphi$.
\end{theorem}
\begin{proof}See \cite[Section~6]{derdzinski-08}.
\end{proof}
All general local properties of RSTS connections could in principle be derived 
from Theorem~\ref{wongs}. However, such derivations are often tedious, which 
is why in this and the following sections direct arguments will be used.

Denoting by $\,R\,$ and $\,\rho\,$ the curvature and Ric\-ci tensors 
of any RSTS connection $\,\nabla\nnh$, we have
\begin{equation}\label{rer}
\mathrm{a)}\hskip9ptR\hh(\hn\eu,\ev)\hh\ev\hh'\nh\,
=\,\rho\hs(\nh\eu,\ev)\hh\ev\hh',\hskip29pt
\mathrm{b)}\hskip9pt\beta\wedge[\hs\rho\hs(\hn\eu,\,\cdot\,)]\,
=\,\beta(\hn\eu)\hs\rho\hh,
\end{equation}
for all tangent vector fields $\,\eu,\ev,\ev\hh'$ and $\,1$-forms $\,\beta\,$ 
(notations of (\ref{bwa}.a)). In fact, (\ref{rer}.a) is a well-known special 
case of the fact that the Ric\-ci tensor of a tor\-sion\-free surface 
connection uniquely determines its curvature tensor. (See, for instance, 
\cite[Lemma 4.1]{derdzinski-08}.) That both sides of (\ref{rer}.b) agree on 
any given pair $\,(\hn\ev,\ev\hh'\hh)\,$ of vector fields is in turn obvious 
from (\ref{bwa}.a), along with (iii) in Section~\ref{prel} applied to the 
expression $\,\beta(\hn\ev)\hs\rho\hs(\hn\eu,\ev\hh'\hh)$, tri\-lin\-e\-ar in 
$\,\ev,\eu,\ev\hh'\nnh$.

In the remainder of this section we assume that $\,\nabla\hn\,$ is a 
tor\-sion\-free connection on a surface $\,\bs\,$ and its Ric\-ci tensor 
$\,\rho$, in addition to being skew-sym\-met\-ric, is nonzero everywhere.

Since $\,\rho\,$ trivializes the bundle $\,[\tab]^{\wedge2}\nnh$, there exist 
a unique $\,1$-form $\,\phi$, called the {\it re\-cur\-rence $1$-form\/} of 
$\,\nabla\nnh$, and a unique vector field $\,\ew\,$ on $\,\bs\,$ such that
\begin{equation}\label{rec}
\mathrm{i)}\hskip9pt\nabla\nnh\rho\,=\,\phi\otimes\rho\hh,\hskip22pt
\mathrm{ii)}\hskip9pt\phi\,=\,\rho\hs(\ew,\,\cdot\,)\hh,\hskip22pt\mathrm{iii)}
\hskip9pt\phi(w)\,=\,0\hh,\hskip22pt\mathrm{iv)}\hskip9ptd\hh\phi=2\rho\hh.
\end{equation}
(Relation (\ref{rec}.iv) is an easy consequence of the Ric\-ci identity; see 
\cite[formula (8.1)]{derdzinski-08}.) Furthermore, for $\,\ew\,$ defined by 
(\ref{rec}.ii) and any vector field $\,\ev\,$ on $\,\bs$,
\begin{equation}\label{dfr}
\mathrm{a)}\hskip9pt\dvw\,=\,2\hh,\hskip22pt
\mathrm{b)}\hskip9ptd\hs[\hs\rho\hs(\hn\ev,\,\cdot\,)]\,
=\,[\hs\mathrm{div}\hskip1.4pt\ev+\phi(\hn\ev)]\hs\rho\hh.
\end{equation}
In fact, if $\,\eu,\eu\hh'$ are arbitrary vector fields, 
$\,(\nabla_{\hskip-2.2pt\eu}[\hh\rho\hs(\hn\ev,\,\cdot\,)])(\hn\eu\hh'\hh)
=\phi(\hn\eu)\hs\rho\hs(\hn\ev,\eu\hh'\hh)+
\rho\hs(\nabla_{\hskip-2.2pt\eu}\ev,\eu\hh'\hh)\,$ by (\ref{rec}.i), so that 
(\ref{bwa}.c) and (\ref{rer}.b) yield (\ref{dfr}.b), since, according to 
(ii) in Section~\ref{prel} and (\ref{bch}.b), skew-sym\-met\-riz\-ing 
$\,\rho\hs(\nabla_{\hskip-2.2pt\eu}\ev,\eu\hh'\hh)\,$ in $\,\eu,\eu\hh'$ we 
obtain one-half of $\,\mathrm{div}\hskip1.4pt\ev\,$ times 
$\,\rho\hs(\hn\eu,\eu\hh'\hh)$. Now (\ref{dfr}.a) follows if one 
applies (\ref{dfr}.b) to $\,\ev=\ew$, using (\ref{rec}.iii), (\ref{rec}.ii) 
and (\ref{rec}.iv). By (\ref{dfr}.a),
\begin{equation}\label{dns}
\mathrm{the\ set\ }\,\,\bs\hh'\nh\subset\bs\,\,\mathrm{\ on\ which\ 
}\,\,\ew\ne0\,\,\mathrm{\ is\ open\ and\ dense\ in\ }\,\,\bs\hh.
\end{equation}
For $\,\bs\,$ and $\,\nabla\hn\,$ as above, still assuming that the Ric\-ci 
tensor $\,\rho\,$ is skew-sym\-met\-ric and $\,\rho\ne0\,$ everywhere, we 
define a vec\-tor-bun\-dle morphism $\,\nd:\tb\to\tb\,$ by
\begin{equation}\label{qef}
\mathrm{i)}\hskip9pt\nd\,=\,4\,+\,\naw\,+\,3\hs\phi\otimes\ew/4\hh,\hskip14pt
\mathrm{so\ that}\hskip10pt\hskip12pt\mathrm{ii)}
\hskip9pt\mathrm{tr}\hskip2.5pt\nd\,=\,10\hh.
\end{equation}
Here $\,\naw:\tb\to\tb\,$ as in (\ref{nwt}), $\,4\,$ means $\,4\,$ times the 
identity, $\,\phi,\ew\,$ are characterized by (\ref{rec}), and (\ref{qef}.ii) 
is immediate from (\ref{dfr}.a) along with (\ref{rec}.iii). Finally, we denote 
by $\,\ob\,$ and $\,\od\,$ the first-or\-der linear differential operators, 
sending symmetric $\,2$-ten\-sors $\,\tau\,$ to $\,1$-forms on $\,\bs$, or, 
respectively, $\,1$-forms $\,\xi\,$ on $\,\bs\,$ to functions 
$\,\bs\to\bbR\hs$, which are given by
\begin{equation}\label{bta}
\mathrm{a)}\hskip9pt[(\ob\tau)(\hn\ev)]\hh\rho\,
=\,[\hh d\hh^\nabla\hskip-2pt\tau](\,\cdot\,,\,\cdot\,,\ev)\hh,\hskip18pt
\mathrm{b)}\hskip9pt2\hs[\od\hh\xi]\hh\rho\,=\,\xi\wedge\phi\,-\,\hh d\hs\xi
\end{equation}
for any vector field $\,\ev$, where $\,d\hh^\nabla\nnh$ is the Co\-daz\-zi 
operator with (\ref{cod}.a). (Note that 
$\,[\hh d\hh^\nabla\hskip-2pt\tau](\,\cdot\,,\,\cdot\,,v)\,$ is a section of 
the bundle $\,[\tab]^{\wedge2}\nnh$, trivialized by $\,\rho$.) Using these 
$\,\phi,\ew,\ob\,$ and $\,\od$, we also define a third-or\-der linear 
differential operator $\,\oz$, sending each symmetric $\,2$-ten\-sor 
$\,\tau\,$ to the $\,1$-form
\begin{equation}\label{zte}
\oz\tau\,=\,2\hs d\hh[\od(\ob\tau)]+4\hh\ob\tau-\tau(\ew,\,\cdot\,)
+3\hh[\od(\ob\tau)]\hs\phi/2\hh.
\end{equation}

\section{The vertical distribution of a type~III SDNE manifold}\label{vert}
\setcounter{equation}{0}
Every type~III SDNE manifold $\,(\ym,\gm)\,$ carries a distinguished 
\hbox{two\hh-}\hskip0ptdi\-men\-sion\-al null distribution $\,\mathcal{V}\nh$, 
which, in addition, is integrable and has totally geodesic leaves. Namely, 
$\,W^+$ acting on self-du\-al $\,2$-forms, at any point $\,x$, has rank $\,2$, 
and hence its kernel is \hbox{one\hh-}\hskip0ptdi\-men\-sion\-al. Thus, we 
may declare $\,\mathcal{V}_{\nh x}$ to be the null\-space of some, or any, 
self-du\-al $\,2$-form at $\,x\,$ spanning 
$\,\mathrm{Ker}\hskip2ptW_{\nnh x}^+\nnh$. That $\,\mathcal{V}\,$ has the 
properties just listed was shown in \cite[Lemma 5.1]{derdzinski-09}.

We refer to $\,\mathcal{V}\,$ as the {\it vertical distribution\/} of 
$\,(\ym,\gm)$, and say that $\,\gm\,$ is a {\it type\/}~III {\it SDNE Walker 
metric\/} if its vertical distribution $\,\mathcal{V}\,$ is parallel. 
Similarly, a type~III SDNE metric $\,\gm\,$ is called {\it strictly \itnw\/} 
\cite[Section~6]{derdzinski-09} if the fundamental tensor of 
$\,\mathcal{V}\nnh$, which measures its deviation from being parallel 
\cite[Section~24]{derdzinski-09}, is nonzero everywhere.

For $\,\gm\,$ as above, being a Walker metric is equivalent to having the 
{\it Walker property\/} mentioned in the Introduction. Namely, the vertical 
distribution $\,\mathcal{V}\,$ of every type~III SDNE manifold $\,(\ym,\gm)\,$ 
is compatible with the orientation \cite[Theorem~6.2(i)]{derdzinski-09} and, 
if $\,(\ym,\gm)\,$ admits any \hbox{two\hh-}\hskip0ptdi\-men\-sion\-al null 
parallel distribution compatible with the orientation, then $\,\mathcal{V}\,$ 
is such a distribution, that is, $\,\mathcal{V}\,$ itself must be parallel 
\cite[Theorem~6.2(ii),\hs(iv)]{derdzinski-09}.

Two constructions of type~III SDNE manifolds are known. One, discovered by 
\dr, \gr\ and \vl\ 
\cite[Theorem~3.1(ii.3)]{diaz-ramos-garcia-rio-vazquez-lorenzo}, always leads 
to Walker metrics. (See also the next section.) The other, described in 
\cite[Theorem~22.1]{derdzinski-09}, gives rise to strictly \nw\ metrics. The 
resulting examples serve as {\it universal models\/}: as shown in 
\cite[Theorem~3.1(ii.3)]{diaz-ramos-garcia-rio-vazquez-lorenzo} and 
\cite[Theorem~22.1]{derdzinski-09}, locally, at points in general position, 
up to isometries, every type~III SDNE manifold arises from one of the two 
constructions just mentioned.

\section{The structure theorem of \hbox{D\'\i az\hh-}\hskip0ptRa\-mos, 
Gar\-\hbox{c\'\i a\hs-}\hskip0ptR\'\i o and 
\hbox{V\nh\'az}\-\hbox{quez\hh-}\hskip0ptLo\-ren\-zo}\label{tstd}
\setcounter{equation}{0}
In \cite[Theorem~3.1(ii.3)]{diaz-ramos-garcia-rio-vazquez-lorenzo}, \dr, \gr\ 
and \vl\ described the local structure of all type~III SDNE Walker metrics. 
With compatibility defined as in the lines preceding Remark~\ref{liebr}, one 
can state their result as follows. (See also \cite[p.\ 238]{derdzinski-08}.)
\begin{theorem}\label{maith}Let there be given a surface\/ $\,\bs$, a 
tor\-sion\-free connection\/ $\,\nabla\hn\,$ on $\,\bs$ such that the 
Ric\-ci tensor\/ $\,\rho\hn\,$ of\/ $\,\nabla\hn\,$ is skew-sym\-met\-ric and 
nonzero everywhere, and a symmetric $\,2$-ten\-sor\/ 
$\,\tau\,$ on $\,\bs$. Then, for a suitable orientation of the 
\hbox{four\hh-}\hskip0ptman\-i\-fold\/ $\,\ym=\tab$, the Riemann extension 
metric\/ $\,\gm\,=\,\gm\nnh^\nabla\nnh+\hs2\hh\pi^*\nnh\tau\,$ on\/ $\,\ym$, 
with the neutral signature\/ $\,(\mmpp)$, is Ric\-ci-flat and self-du\-al of 
Pe\-trov type\/ {\rm III}, the distribution\/ $\,\mathcal{V}=\kerd\pi\,$ is 
$\,\gm$-null, $\,g$-par\-al\-lel, compatible with the orientation, and 
constitutes the vertical distribution of\/ $\,\gm$, cf.\ 
Section\/~{\rm\ref{vert}}, while\/ $\,\gm\,$ and\/ $\,\mathcal{V}\,$ satisfy 
the curvature condition\/ {\rm(\ref{ccn})}, and the corresponding transversal 
connection on\/ $\,\bs$, described in Lemma\/~{\rm\ref{trvcn}(b)}, coincides 
with our original\/ $\,\nabla\nh$.

Conversely, if\/ $\,(\ym,\gm)\,$ is a neu\-tral-sig\-na\-ture oriented 
self-du\-al Ein\-stein \hbox{four\hskip.4pt-}\hskip0ptman\-i\-fold of Pe\-trov 
type\/ {\rm III} admitting a \hbox{two\hh-}\hskip0ptdi\-men\-sion\-al null 
parallel distribution\/ $\,\mathcal{V}\,$ compatible with the orientation, 
then, for every $\,x\in\ym$, there exist $\,\bs,\nnh\nabla\nnh,\tau\,$ as 
above and a dif\-feo\-mor\-phism of a neighborhood of\/ $\,x\,$ onto an open 
subset of\/ $\,\tab$, under which\/ $\,\gm\,$ corresponds to the metric\/ 
$\,\gm\,=\,\gm\nnh^\nabla\nnh+\hs2\hh\pi^*\nnh\tau$, and\/ $\,\mathcal{V}\,$ 
to the vertical distribution\/ $\,\kerd\pi$.
\end{theorem}
According to Theorem~\ref{maith}, every type~III SDNE Walker metric $\,g\,$ 
can be locally identified with a Riemann extension metric 
$\,\gm\,=\,\gm\nnh^\nabla\nnh+\hs2\hh\pi^*\nnh\tau\,$ defined as at the 
beginning of Section~\ref{riex} in the special case where $\,\bs\,$ is a 
surface and the tor\-sion\-free connection $\,\nabla\hn\,$ on $\,\bs\,$ has 
the property that its Ric\-ci tensor $\,\rho\,$ is skew-sym\-met\-ric and 
nonzero at every point. By Lemma~\ref{rcflt}, $\,\gm\,$ is Ric\-ci-flat, while 
the formulae, appearing in (\ref{gjk}), for the components of its curvature 
$\,4\hh$-ten\-sor $\,\br\,$ in coordinates $\,y^{\hs j}\nnh,x^\lambda$ chosen 
for (\ref{gjk}) may be rewritten as
\begin{equation}\label{rlm}
%\begin{array}{l}
\br_{\lambda\mu\hs\md\md}=\hs\br_{\lambda\md\mu\md}=\hs0\hh,
\hskip12pt\br_{jkl\lambda}=\hs\gm_{\hs l\lambda}\rho_{jk},\hskip12pt
\br_{jkls}=[\gm_{p\lambda}x^\lambda\ew\hh^p
+2\hh\od(\ob\tau)]\hs\rho_{jk}\hs\rho_{\hs ls}\hh,
%\end{array}
\end{equation}
where $\,\ob\,$ and $\,\od\,$ are the operators defined by (\ref{bta}), and 
the index convention is the same as in (\ref{gjk}). In fact, by (\ref{rer}.a), 
$\,R_{jkl}{}^s\nh=\rho_{jk}\delta_l^s$, and so, on the right-hand side of the 
last equality in (\ref{gjk}), after interchaning the indices 
$\,p\,$ and $\,s$, we have 
$\,\vg_{\hskip-3ptjq}^pR_{lsk}{}^q-\vg_{\!kq}^pR_{lsj}{}^q
+\vg_{\!lq}^pR_{jks}{}^q-\vg_{\!sq}^pR_{jkl}{}^q
=(\vg_{\hskip-3ptjk}^p-\vg_{\!kj}^p)\rho_{\hs ls}
+(\vg_{\!ls}^p-\vg_{\!sl}^p)\rho_{jk}=0$, as well as 
$\,R_{kjs}{}^p\tau_{pl}-R_{kjl}{}^p\tau_{ps}
=(\tau_{sl}\nh-\tau_{\hs ls})\rho_{kj}=0$. Similarly, 
$\,R_{lsj}{}^p{}_{\nnh,\hs k}-R_{lsk}{}^p{}_{\nnh,\hs j}$ equals 
$\,\rho_{\hs ls,\hs k}\delta_j^p-\rho_{\hs ls,\hs j}\delta_k^p$, and hence 
$\,\ew\hh^p\rho_{jk}\hs\rho_{\hs ls}$, as $\,\rho_{\hs ls,\hs k}=\phi_{\hs 
k}\rho_{\hs ls}$ by (\ref{rec}.i), while $\,\phi_{\hs k}\delta_j^p
-\phi_j\delta_k^p=\ew\hh^p\rho_{jk}$ in view of (\ref{rec}.ii) and 
(\ref{rer}.b) for $\,\eu=\ew\,$ and the $\,1$-form $\,\beta\,$ with 
$\,\beta_j=\delta_j^p$. Finally, since 
$\,(d\hh^\nabla\hskip-2pt\tau)_{lsj}=(\ob\tau)_j\rho_{\hs ls}$ (cf.\ 
(\ref{bta}.a)) and $\,\rho_{\hs ls,\hs k}=\phi_{\hs k}\rho_{\hs ls}$ 
(see above), using (\ref{bta}.b), (\ref{bwa}.a) and (\ref{bwa}.c) we 
conclude that $\,(d\hh^\nabla\hskip-2pt\tau)_{lsj,\hs k}\nh
-(d\hh^\nabla\hskip-2pt\tau)_{lsk,\hs j}\nh
=2\hh\od(\ob\tau)\hs\rho_{jk}\hs\rho_{\hs ls}$.

\begin{remark}\label{trinv}Every type~III SDNE Walker metric $\,\gm$, 
restricted to a suitable neighborhood of any given point of the underlying 
\hbox{four\hh-}\hskip0ptman\-i\-fold, gives rise to a triple 
$\,(\bs,\nabla\nnh,[\tau])\,$ of invariants. Specifically, $\,\bs\,$ is a 
surface (a local leaf space of the vertical distribution $\,\mathcal{V}\nh$, 
cf.\ Section~\ref{vert}), $\,\nabla\hn\,$ is a tor\-sion\-free connection on 
$\,\bs\,$ with eve\-ry\-\hbox{where\hskip1pt-}\hskip0ptnon\-zero, 
skew-sym\-met\-ric Ric\-ci tensor (namely, the connection described in 
Theorem~\ref{maith}), and $\,[\tau]\,$ denotes a coset, in the vector space of 
all symmetric $\,2$-ten\-sors of class $\,C^\infty$ on $\,\bs$, of the image 
of the Kil\-ling operator $\,\ki\,$ for $\,\nabla\nnh$, given by 
(\ref{cod}.b). (Here the coset is chosen so as to contain the $\,2$-ten\-sor 
$\,\tau\,$ appearing in Theorem~\ref{maith}.) Although $\,\tau\,$ itself is 
not an invariant of $\,\gm$, the coset $\,[\tau]\,$ is, as one sees using 
Lemma~\ref{ismtr}(ii) for $\,\varPhi=\mathrm{Id}$, the local leaf space 
$\,\plne=\bs$, and $\,\mathrm{D}=\nabla\nnh$, with $\,\ts\,$ denoting the 
other choice of $\,\tau$.

Conversely, by Theorems~\ref{maith} and~\ref{imker}(c)), every triple 
$\,(\bs,\nabla\nnh,[\tau])\,$ with the properties just listed arises in this 
manner from some type~III generic SDNE Walker metric $\,\gm$, namely, the 
Riemann extension $\,\gm\,=\,\gm\nnh^\nabla\nnh+\hs2\hh\pi^*\nnh\tau$.

Finally, the original metric $\,\gm$, on a suitable neighborhood of the given 
point, is uniquely determined, up to an isometry, by the corresponding 
triple $\,(\bs,\nabla\nnh,[\tau])$. In fact, 
$\,\gm\,=\,\gm\nnh^\nabla\nnh+\hs2\hh\pi^*\nnh\tau\hh'$ for some 
$\,\tau\hh'$ that lies in the coset $\,[\tau]$, while, for any two choices 
of such $\,\tau\hh'\nnh$, the resulting metrics are isometric to each other 
(Lemma~\ref{kxpgn}(b)).
\end{remark}

\section{Some natural tensor fields on a type~III SDNE Walker 
manifold}\label{anvf}
\setcounter{equation}{0}
In the next two lemmas $\,(\ym,\gm)\,$ is assumed to be a type~III SDNE Walker 
manifold. By Theorem~\ref{maith}, $\,(\ym,\gm)\,$ may be identified, locally, 
with a Riemann extension $\,(\tab,\gm\nnh^\nabla\hn+\hs2\hh\pi^*\nnh\tau)\,$ 
for a surface $\,\bs\,$ with a tor\-sion\-free connection $\,\nabla\nnh$, the 
Ric\-ci tensor $\,\rho\,$ of which is skew-sym\-met\-ric and nonzero 
everywhere, and some symmetric $\,2$-ten\-sor $\,\tau\,$ on $\,\bs$. We will 
also choose, in $\,\ym$, local coordinates $\,y^{\hs j}\nnh,x^\lambda$ with 
(\ref{gjk}) and (\ref{rlm}).
\begin{lemma}\label{qntpl}For every type\/~{\rm III} SDNE Walker manifold\/ 
$\,(\ym,\gm)\,$ there exists a unique quintuple\/ 
$\,(\zeta,\eta,A,\gamma,\ev)\,$ of local geometric invariants of\/ $\,\gm\,$ 
consisting of\/ $\,2$-forms\/ $\,\zeta,\eta$, 
a bundle morphism\/ $\,A:\tm\to\tm$, a\/ $\,1$-form\/ $\,\gamma$, and a vector 
field\/ $\,\ev$, all defined globally on\/ $\,\ym$, such that
\begin{enumerate}
  \def\theenumi{{\rm\roman{enumi}}}
\item[(i)] $2\hh\br=\zeta\otimes\eta+\hs\eta\otimes\zeta$, where\/ $\,\br\,$ 
is the curvature $\,4\hh$-ten\-sor of\/ $\,(\ym,\gm)$,
\item[(ii)] $\bna\eta=2\gamma\otimes\zeta$, with\/ $\,\bna\,$ denoting the 
Le\-vi-Ci\-vi\-ta connection of\/ $\,(\ym,\gm)$,
\item[(iii)] $\zeta=-\hs2\hh\pi^*\nnh\rho$, the symbol\/ $\,\pi\,$ standing 
for the bundle projection\/ $\,\tab\to\bs$,
\item[(iv)] $\eta(\nh\eu,\,\cdot\,)=\gm(A\eu,\,\cdot\,)\,$ and\/ 
$\,\gm(\hn\ev,\eu)=4\hs[\hh\gamma(\hn\eu)-\gamma(A\eu)]\,$ for all vector 
fields\/ $\,\eu$.
\end{enumerate}
In coordinates\/ $\,y^{\hs j}\nnh,x^\lambda$ chosen as above, with\/ $\,\nd\,$ 
given by\/ {\rm(\ref{qef}.i)}, $\,\ev\,$ has the components
\begin{equation}\label{veq}
\ev^j\nh=0\hh,\hskip28pt\ev^\lambda=\gm^{j\lambda}(\gm_{k\mu}x^{\hs\mu}\nd_j^k
+\hs\xi_{\hs j})\hh,%+[\oz\tau]_j\hh,
\end{equation}
where\/ $\,\xi\,$ is a $\,1$-form on\/ $\,\bs\,$ which may depend on the 
choice of the special coordinates.
\end{lemma}
\begin{proof}By (\ref{rlm}), (i) and (iii) hold for the $\,2$-forms $\,\zeta\,$ 
and $\,\eta\,$ defined by 
$\,\zeta_{jk}=-2\hh\rho_{jk}$, $\,\eta_{jk}=-\hs[\hh\od(\ob\tau)
+\hs\gm_{p\lambda}x^\lambda\ew\hh^p\nnh/\hs2\hh]\hs\rho_{jk}$, 
$\,\eta_{\lambda j}=-\hs\eta_{j\lambda}=\gm_{j\lambda}$, 
$\,\zeta_{\lambda j}=\hs\zeta_{j\lambda}=\hs\zeta_{\lambda\mu}
=\eta_{\lambda\mu}=\hs0$. Both $\,\zeta\,$ and $\,\eta\,$ 
are local geometric invariants of the metric: $\,\zeta\,$ by (iii) and 
Theorem~\ref{maith}, $\,\eta\,$ in view of (i) and the fact that symmetric 
multiplication has no zero divisors.

We now establish (ii) for a $\,1$-form $\,\gamma\,$ with the components 
satisfying the conditions
\begin{equation}\label{egl}
8\hs\gamma_\lambda=\hs\gm_{k\lambda}\ew\hh^k\hskip12pt\mathrm{and}\hskip13pt
8\hs\gamma_j\,\sim\,\hs\gm_{k\mu}x^{\hs\mu}(\nd_j^k-\vg_{\!jl}^k\ew\hh^l
+\phi_j\ew\hh^k\nnh/\hh4)\hh,
\end{equation}
the relation $\,\sim\,$ meaning, in the rest of 
the proof, that the two expressions differ by a function in $\,\bs\,$ (which 
may itself depend both on some indices and on the choice of our coordinates); 
in other words, their difference is allowed to depend on the coordinates 
$\,y^{\hs j}\nnh$, but not on $\,x^\lambda\nnh$. In fact, with the semicolons 
standing for \hbox{$\,\bna\nnh$-\hh}\hskip0ptco\-var\-i\-ant derivatives, 
using (\ref{gjk}), we easily obtain 
$\,\eta_{\lambda\mu;\nu}=\hh\eta_{\lambda\mu;j}
=\hh\eta_{\lambda j;\mu}=\hh\eta_{\lambda j;\hs k}=\hs0\,$ and 
$\,2\hh\eta_{jk;\hs\lambda}=-\hs\gm_{s\lambda}\ew\hh^s\rho_{jk}$. Next, 
$\,-\bvg_{\!jl}^\mu\eta_{k\mu}-\bvg_{\!jk}^\mu\eta_{\mu l}
=\gm_{k\mu}\bvg_{\!jl}^\mu-\gm_{\hs l\mu}\bvg_{\!jk}^\mu$ which, as a 
consequence of the formula for $\,\gm_{k\mu}\bvg_{\!jl}^\mu$ in (\ref{gjk}), 
equals 
$\,2\hs\gm_{s\lambda}x^\lambda R_{lkj}{}^s+2(d\hh^\nabla\hskip-2pt\tau)_{lkj}$ 
(note the numerous cancellations due to symmetry in $\,k,l$). Since 
$\,\rho_{kl,\hs j}=\phi_j\rho_{kl}$, 
$\,R_{lkj}{}^s\nh=\rho_{\hs lk}\delta_j^s$ and 
$\,(d\hh^\nabla\hskip-2pt\tau)_{lkj}=(\ob\tau)_j\rho_{\hs lk}$ (see 
(\ref{rec}.i), (\ref{rer}.a) and (\ref{bta}.a)), this gives, by (\ref{gjk}), 
$\,2\hh\eta_{kl;\hs j}=-\hs\gamma_j\rho_{kl}$ with $\,\gamma_j$ as in 
(\ref{egl}), thus proving (ii) and (\ref{egl}).

Assertion (iv) is simply a definition of $\,A\,$ and $\,\ev$, stating that 
they are obtained from $\,\eta\,$ and $\,4(\gamma-A^{\nh*}\nh\gamma)\,$ by index 
raising. (Notation of (\ref{ast}).) Using (\ref{gjk}) we thus get 
$\,A_j^k=-\hs\delta_j^k$, $\,A_\lambda^j=0$, 
$\,A_\lambda^\mu=\delta_\lambda^\mu$ and 
$\,A_j^\lambda\,\sim\,-\hs\gm^{k\lambda}\gm_{s\mu}x^\mu(2\vg_{\hskip-3ptjk}^s
+\ew\hh^s\rho_{jk}\nh/\hh2)$. Consequently, 
$\,\ev^j\nh=4\hh\gm^{j\lambda}(\gamma-A^{\nh*}\nh\gamma)_\lambda\nh
=4\hh\gm^{j\lambda}(\gamma_\lambda\nh-A_\lambda^\mu\gamma_\mu)=0$. Similarly, 
$\,\ev^\lambda\nh=4\hh\gm^{j\lambda}(\gamma-A^{\nh*}\nh\gamma)_j\nh
=4\hh\gm^{j\lambda}(2\hh\gamma_j\nh-A_j^\mu\gamma_\mu)$. Now (\ref{egl}) and 
the equality $\,\ew\hh^k\rho_{jk}=-\hs\phi_j$ (see (\ref{rec}.ii)) yield 
(\ref{veq}).
\end{proof}
\begin{lemma}\label{vrtvo}Every type\/~{\rm III} SDNE Walker manifold\/ 
$\,(\ym,\gm)\,$ admits a globally defined section\/ $\,\theta\,$ of the real 
line bundle\/ $\,[\mathcal{V}^*]^{\wedge2}\nnh$, where\/ $\,\mathcal{V}\,$ 
denotes the vertical distribution, such that the restriction of\/ $\,\theta\,$ 
to each leaf\/ $\,\yn\,$ of\/ $\,\mathcal{V}\,$ is nonzero and parallel 
relative to the connection on\/ $\,\yn\,$ induced by the Le\-vi-Ci\-vi\-ta 
connection of\/ $\,\gm$.
\end{lemma}
\begin{proof}The $\,2$-form $\,\zeta\,$ appearing in Lemma~\ref{qntpl}(iii) 
may be treated as a no\-where-zero section of the vector bundle 
$\,\qt^{\wedge2}$ over $\,\ym$, where $\,\qt\,$ stands for the dual of the 
quotient bundle $\,(\tm)/\hs\mathcal{V}\nh$. As $\,\mathcal{V}\,$ is a 
$\,\gm$-null subbundle of $\,\tm$, the metric $\,\gm\,$ constitutes a 
vec\-tor-bun\-dle isomorphism $\,\mathcal{V}\to\qt\nnh$, under which 
$\,\zeta\,$ corresponds to a trivializing section of 
$\,\mathcal{V}^{\wedge2}\nnh$. Our $\,\theta\,$ is its dual trivializing 
section in $\,[\mathcal{V}^*]^{\wedge2}\nnh$. That $\,\theta\,$ is parallel in 
the direction of $\,\mathcal{V}\,$ is immediate, since so are $\,\zeta$, as 
one sees using (\ref{gjk}), and $\,\gm$.
\end{proof}
The local geometric invariants of type~III SDNE Walker manifolds, described in 
this section, can be naturally generalized to arbitrary type~III SDNE 
manifolds, with or without the Walker property. See 
\cite[Lemma 5.1(c),\hs(e) and Theorem 6.2(ii)]{derdzinski-09}.

\section{Noncompactness of type~III SDNE Walker manifolds}\label{nect}
\setcounter{equation}{0}
We begin with two lemmas. The first is obvious from (\ref{bch}.b) and (i) in 
Section~\ref{prel}.
\begin{lemma}\label{dvedv}If\/ $\,\mathcal{V}\,$ is a\/ 
$\,\bna\nh$-par\-al\-lel distribution on a manifold\/ $\,\ym\,$ endowed with a 
tor\-sion\-free connection $\,\bna\nnh$, and\/ $\,\ev\,$ is a section of\/ 
$\,\mathcal{V}\nh$, then\/ 
$\,\mathrm{div}\hskip1.4pt\ev=\mathrm{div}^{\mathcal{V}}\nh\ev$. Here\/ 
$\,\mathrm{div}\,$ is the $\,\bna\nnh$-di\-ver\-gence, given by\/ 
{\rm(\ref{bch}.b)} for\/ $\,\nabla=\bna\nnh$, while the function\/ 
$\,\mathrm{div}^{\mathcal{V}}\nh\ev:\ym\to\bbR\,$ is defined so as to 
coincide, on each leaf\/ $\,\yn\,$ of\/ $\,\mathcal{V}\nh$, with the\/ 
\hbox{$\,\mathrm{D}$-\hn}\hskip0ptdi\-ver\-gence of the restriction of\/ 
$\,\ev\,$ to\/ $\,\yn$, where\/ $\,\mathrm{D}\,$ denotes the connection on\/ 
$\,\yn\,$ induced by\/ $\,\bna\nnh$.
\end{lemma}
\begin{lemma}\label{dvten}If\/ $\,(\ym,\gm)\,$ is a type\/~{\rm III} SDNE 
Walker manifold and\/ $\,\ev\,$ denotes the vector field appearing in 
Lemma\/~{\rm\ref{qntpl}}, then\/ 
$\,\mathrm{div}\hskip1.4pt\ev=\mathrm{div}^{\mathcal{V}}\nh\ev=10$, with\/ 
$\,\mathrm{div}^{\mathcal{V}}$ as in Lemma\/~{\rm\ref{dvedv}} for the 
Le\-vi-Ci\-vi\-ta connection\/ $\,\bna\hn\,$ of\/ $\,\gm$, and the 
vertical distribution\/ $\,\mathcal{V}\,$ of\/ $\,(\ym,\gm)$, cf.\ 
Section\/~{\rm\ref{vert}}.
\end{lemma}
\begin{proof}We use the notations and identifications described at the 
beginning of Section~\ref{anvf}. The equality 
$\,\bvg_{\!\lambda\mu}^{\hh\md}=0\,$ in (\ref{gjk}) states that 
$\,x^\lambda$ are affine coordinates on each leaf $\,\tayb=\pi^{-1}(y)$. Thus, 
by (\ref{veq}) and (\ref{qef}.ii), $\,\mathrm{div}^{\mathcal{V}}\nh\ev
=\partial_\mu\ev^{\hs\mu}=\gm^{j\mu}\gm_{k\mu}\nd\hn_j^{\hs k}
=\nd\hn_k^{\hs k}=10$, 
and our assertion is immediate from Lemma~\ref{dvedv}.
\end{proof}
Lemma~\ref{dvten} leads to the following conclusion.
\begin{theorem}\label{nocpl}Suppose that\/ $\,(\ym,\gm)\,$ is a type\/~{\rm 
III} SDNE Walker manifold. Then
\begin{enumerate}
  \def\theenumi{{\rm\alph{enumi}}}
\item[(a)] $\ym\,$ is not compact,
\item[(b)] the vertical distribution\/ $\,\mathcal{V}\,$ has no compact leaves.
\end{enumerate}
\end{theorem}
\begin{proof}The divergence formula is well-known to remain valid for any 
compact manifold with a tor\-sion\-free connection admitting a global parallel 
volume element. (Cf.\ \cite[Remark 7.3]{derdzinski-roter}.) Thus, compactness 
of $\,\ym$, or of some leaf of $\,\mathcal{V}\nh$,  would contradict 
Lemmas~\ref{dvten} and~\ref{vrtvo}.
\end{proof}

\section{Left-in\-var\-i\-ant RSTS connections on a Lie group}\label{lirs}
\setcounter{equation}{0}
Ko\-wal\-ski, O\-poz\-da and Vl\'a\-\v sek \cite{kowalski-opozda-vlasek-00} 
found a canonical coordinate form of RSTS connections that are also {\it 
locally homogeneous}. More general results later appeared in \cite{opozda} and 
\cite{arias-marco-kowalski}.

It is convenient for us to rephrase the result of 
\cite{kowalski-opozda-vlasek-00} using left-in\-var\-i\-ant connections on a 
Lie group. This approach has the added benefit of providing a precise 
description of a local moduli space of the (non\-flat) connections in 
question, which turns out to be a moduli {\it curve}, namely, the union of two 
subsets ho\-me\-o\-mor\-phic to $\,\bbR\hs$, intersecting at one point. See 
Section~\ref{mcrs}.

We always identify the Lie algebra of a Lie group $\,\hp\,$ with the space 
$\,\hi\,$ of left-in\-var\-i\-ant vector fields on $\,\hp$. If $\,\hp\,$ is 
\hbox{two\hh-}\hskip0ptdi\-men\-sion\-al, 
\hbox{non\hs-\nh}\hskip0ptAbel\-i\-an, simply connected, and $\,\eu,\ew\,$ is 
a basis of $\,\hi\,$ such that $\,[\eu,\ew]=2\eu$, then there exists a 
function $\,\ef:\hp\to\bbR\,$ with
\begin{equation}\label{duf}
d_\eu\ef\,=\,0\hs,\hskip22ptd_\ew\ef\,=\,-2\hs\ef,\hskip22pt\ef\,>\,\hs0.
\end{equation}
Such functions $\,\ef\,$ are positive constant multiples of a specific 
Lie-group homomorphism from $\,\hp\,$ into the multiplicative group 
$\,(0,\infty)$. In fact, by (\ref{bwa}.b), the left-in\-var\-i\-ant $\,1$-form 
sending $\,\eu\,$ to $\,0\,$ and $\,\ew\,$ to $\,-2\,$ is closed, so that it 
equals $\,d\hskip2pt\mathrm{log}\hs\ef\,$ for some function $\,\ef>0$. 
Left-in\-var\-i\-ance of $\,d\hskip2pt\mathrm{log}\hs\ef\,$ means in turn that 
left translations act on $\,\ef\,$ via multiplications by constants, which 
characterizes nonzero multiples of homomorphisms $\,\hp\to(0,\infty)$.
\begin{lemma}\label{liegp}The left-in\-var\-i\-ant connections\/ $\,\nabla\hs$ 
with skew-sym\-met\-ric Ric\-ci tensor on any connected 
\hbox{two\hh-}\hskip0ptdi\-men\-sion\-al Lie group\/ $\,\hp\,$ are in a 
bijective correspondence with pairs $\,(\Psi,f\hs)\hs$ formed by a 
Lie-al\-ge\-bra homomorphism\/ 
$\,\Psi:\hi\to\mathfrak{sl}\hh(\hi)\,$ and a linear 
functional $\,\lf\in\hi^*\nnh$, where\/ $\,\hi\,$ is the Lie 
algebra of\/ $\,\hp$, consisting of left-in\-var\-i\-ant vector fields on\/ 
$\,\hp$, and\/ $\,\mathfrak{sl}\hh(\hi)\,$ stands for the Lie algebra 
of trace\-less vec\-tor-space en\-do\-mor\-phisms of\/ $\,\hi$.

The correspondence is given by 
$\,\nabla_{\hskip-2.2pt\eu}\ev=[\Psi\eu]\hh\ev+\lf(\hn\eu)\hh\ev\,$ for 
$\,\eu,\ev\in\hi$, and\/ $\,\nabla\hn\,$ has the Ric\-ci tensor 
$\,\rho\,$ with $\,\rho\hs(\nh\eu,\ev)=\lf([\eu,\ev])$.
\end{lemma}
\begin{proof}This is obvious from \cite[Theorem~7.2]{derdzinski-08} and 
\cite[Lemma 4.1]{derdzinski-08}.
\end{proof}
\begin{example}\label{nonab}Let us fix a 
\hbox{two\hh-}\hskip0ptdi\-men\-sion\-al \hbox{non\hs-\nh}\hskip0ptAbel\-i\-an 
simply connected Lie group $\,\hp\,$ along with a basis $\,\eu,\ew\,$ of its 
Lie algebra $\,\hi\,$ such that $\,[\eu,\ew]=2\eu$. Given real parameters 
$\,\ea,\eb\,$ with $\,\ea\eb=0$, we define a left-in\-var\-i\-ant 
tor\-sion\-free connection $\,\nabla=\nabla(\ea,\eb)\,$ on $\,\hp\,$ by
\begin{equation}\label{nuu}
\begin{array}{ll}
\nabla_{\hskip-2.2pt\eu}\eu\,=\,(3+\ea)\hh\eu-\ea\hh\ew\hh,
&\nabla_{\hskip-2.2pt\eu}\ew\,=\,\hs\ea\hh\eu+(3-\ea)\hh\ew\hh,\\
\nabla_{\hskip-2.2pt\ew}\eu\,=\,(\ea-2)\hh\eu+(3-\ea)\hh\ew\hh,\hskip10pt
&\nsww\,=\,(\ea+\eb-1)\hh\eu+(2-\ea)\hh\ew\hh.
\end{array}
\end{equation}
The Ric\-ci tensor $\,\rho\,$ of $\,\nabla\hn\,$ then is skew-sym\-met\-ric 
and, for the re\-cur\-rence $\,1$-form $\,\phi\,$ of $\,\nabla\nnh$,
\begin{equation}\label{ruw}
\mathrm{a)}\hskip9pt\rho\hs(\hn\eu,\ew)=6\hh,\hskip18pt
\mathrm{b)}\hskip9pt\phi(\hn\eu)=-\hh6\hh,\hskip9pt\phi(\ew)=0\hh,
\end{equation}
while $\,\ew\,$ coincides with the vector field in (\ref{rec}.ii). Whenever 
$\,\ef:\hp\to\bbR\,$ satisfies (\ref{duf}),
\begin{equation}\label{riv}
\mathrm{i)}\hskip9pt\ef\nh\rho\hskip9pt\mathrm{is\ 
right}\mbox{-}\mathrm{in\-var\-i\-ant,}\hskip18pt\mathrm{ii)}\hskip9pt
d(\ef^s\nh\phi)=2(1-s)\ef^s\hskip-1.5pt\rho\hskip9pt\mathrm{for\ 
any}\hskip5pts\in\bbR\hh.
\end{equation}
There exists a function $\,\psi:\hp\to\bbR\,$ with $\,3\hs d\psi=-\ef\phi$, 
and, for any such $\,\psi$,
\begin{equation}\label{vot}
\ev_1^{\phantom i}\nh=\ef^{-1}\eu\hskip4.5pt\mathrm{and}\hskip4.5pt 
\ev_2^{\phantom i}\nh=\ef^{-1}\psi\eu-\ew\hskip5.5pt\mathrm{are\ 
right}\mbox{-}\mathrm{in\-var\-i\-ant\ vector\ fields,\hskip4.5ptwhile}
\hskip4.5pt[\ev_1^{\phantom i},\ev_2^{\phantom i}]=2\hh\ev_1^{\phantom i}\hh.
\end{equation}
Finally, if $\,(\ea,\eb)=(1,0)\,$ and $\,\oz\,$ is the operator given by 
(\ref{zte}), we have
\begin{equation}\label{zff}
\oz(\phi\otimes\phi)\,=\,15\hs\phi\nh/2\,\ne\,0\hh.
\end{equation}
In fact, $\,\nabla\hn\,$ corresponds as in Lemma~\ref{liegp} to the pair 
$\,(\Psi,\lf)\,$ such that $\,\lf(\hn\eu)=3$, $\,\lf(\ew)=0$, and the matrices 
representing $\,\Psi\eu\,$ and $\,\Psi\ew\,$ in the basis $\,\eu,\ew\,$ are
\begin{equation}\label{bue}
\mathfrak{B}_\eu\,
=\hs\left[\begin{matrix}\ea&\ea\\-\ea&-\ea\hs\end{matrix}\right],\hskip27pt
\mathfrak{B}_\ew\,
=\hs\left[\begin{matrix}\hh\ea-2&\ea+\eb-1\\\hh3-\ea&2-\ea\end{matrix}\right]
\nnh.
\end{equation}
We have 
$\,\mathfrak{B}_\eu\mathfrak{B}_\ew-\mathfrak{B}_\ew\mathfrak{B}_\eu
=2\hh\mathfrak{B}_\eu$. Hence $\,\Psi\,$ is a Lie-al\-ge\-bra homomorphism, 
and so, according to Lemma~\ref{liegp},  $\,\nabla\hn\,$ has 
skew-sym\-met\-ric Ric\-ci tensor with (\ref{ruw}.a), while, evaluating 
$\,d_\eu[\hh\rho\hs(\hn\eu,\ew)]\,$ and $\,d_\ew[\hh\rho\hs(\hn\eu,\ew)]\,$ 
via the Leib\-niz rule and (\ref{nuu}), we obtain (\ref{rec}.i,\hs ii) for 
$\,\phi\,$ with (\ref{ruw}.b). Next, (\ref{duf}) and (\ref{ruw}.a) give 
$\,\rho\hs(\hn\eu,\,\cdot\,)=-3\,d\hskip2pt\mathrm{log}\hs\ef$, so that 
(\ref{riv}.ii) follows from (\ref{rec}.iv) and the relation 
$\,\phi\wedge\df=2\ef\nh\rho$, immediate from (\ref{rer}.b) and (\ref{ruw}.b). 
As $\,\rho\hs(\hn\eu,\,\cdot\,)=-3\,d\hskip2pt\mathrm{log}\hs\ef$, 
(\ref{lvz}), (\ref{rec}.ii) and (\ref{riv}.ii) with $\,s=1\,$ yield 
$\,\lie_\eu(\ef\nh\rho)=\lie_\ew(\ef\nh\rho)=0$. Since the flows of 
left-in\-var\-i\-ant vector fields consist of right translations, this proves 
(\ref{riv}.i). For $\,\ev_j$ as in (\ref{vot}), using (\ref{duf}) and 
(\ref{ruw}.b) we easily get 
$\,\lie_\eu\ev_j\nh=\lie_\ew\ev_j\nh=0\,$ for $\,j=1,2$, which implies 
(\ref{vot}), closedness of the $\,1$-form $\,\ef\phi\,$ being obvious from
(\ref{riv}.ii). Finally, by (\ref{nuu}) with $\,(\ea,\eb)=(0,1)$, the 
Leib\-niz rule and (\ref{ruw}), $\,(\nsu\phi)(\hn\eu)=24$, 
$\,(\nsu\phi)(\ew)=6$, $\,(\nsw\phi)(\hn\eu)=-\hs6$, $\,(\nsw\phi)(\ew)=0$, so 
that $\,\nsw\phi=\phi$. Now (\ref{ruw}.b) gives 
$\,[\nsu(\phi\otimes\phi)](\ew,\,\cdot\,)=6\hh\phi\,$ and 
$\,[\nsw(\phi\otimes\phi)](\hn\eu,\,\cdot\,)=-12\hh\phi$, which, combined 
with (\ref{bta}.a) and (\ref{ruw}.a), yields $\,\ob(\phi\otimes\phi)=3\hh\phi$. 
However, $\,\od\phi=-1\,$ by (\ref{bta}.b) and (\ref{rec}.iv). Therefore, 
(\ref{zff}) follows from (\ref{zte}) for $\,\tau=\phi\otimes\phi\,$ along with 
(\ref{rec}.iii).
\end{example}
\begin{remark}\label{uwinv}If $\,(\ea,\eb)\ne(1,0)$, the vector fields 
$\,\eu\,$ and $\,\ew\,$ are local geometric invariants of the connection 
$\,\nabla=\nabla(\ea,\eb)\,$ given by (\ref{nuu}).

For $\,\ew\,$ this is clear, also when $\,(\ea,\eb)=(1,0)$, from 
(\ref{rec}.ii). On the other hand, (\ref{ruw}.b) determines $\,\eu\,$ 
uniquely up to its replacement by $\,\eu+\chi\hh\ew$, where $\,\chi\,$ is any 
function. As $\,\ea\eb=0$, the requirement that the equality 
$\,\nabla_{\hskip-2.2pt\eu}\ew=\ea\hh\eu+(3-\ea)\hh\ew\,$ in (\ref{nuu}) 
remain valid, even after $\,\eu\,$ has been replaced by $\,\eu+\chi\hh\ew$, 
easily gives $\,(\ea,\eb)=(1,0)\,$ unless $\,\chi\,$ is identically zero.
\end{remark}
\begin{proposition}\label{wnonz}If\/ $\,\nabla\hn\,$ is a tor\-sion\-free 
connection on a surface\/ $\,\bs\,$ with 
eve\-ry\-\hbox{where\hskip1pt-}\hskip0ptnon\-zero, skew-sym\-met\-ric 
Ric\-ci tensor, while\/ {\rm(\ref{nuu})} holds on a nonempty open set\/ 
$\,\bs\hh'\nh\subset\bs$, for some constants\/ $\,\ea,\eb\,$ with\/ 
$\,\ea\eb=0\,$ such that\/ $\,\ea+\eb\ne1$, and some vector fields\/ 
$\,\eu,\ew\,$ defined on\/ $\,\bs\hh'\nnh$, which are linearly independent 
at each point of\/ $\,\bs\hh'\nnh$, then
\begin{enumerate}
  \def\theenumi{{\rm\roman{enumi}}}
\item[(i)] $\ew\,$ is the restriction to\/ $\,\bs\hh'$ of the vector field\/ 
$\,\ew\,$ given by\/ {\rm(\ref{rec}.ii)},
\item[(ii)] the vector field\/ $\,\ew\,$ with\/ {\rm(\ref{rec}.ii)} is 
nonzero everywhere in the closure of\/ $\,\bs\hh'\nnh$.
\end{enumerate}
\end{proposition}
\begin{proof}Assertion (i) was established in Example~\ref{nonab}. 
(Since $\,\nabla\hn\,$ is tor\-sion\-free, $\,[\eu,\ew]=2\eu$.) If we now had 
$\,\ew\to0\,$ on some sequence of points of $\,\bs\hh'$ converging in $\,\bs$, 
it would follow that $\,\nsww\to0\,$ as well, and so the last equality in 
(\ref{nuu}) would give $\,(\ea+\eb-1)\hh\eu\to0$, that is, $\,\eu\to0$, 
contradicting (\ref{ruw}.a).
\end{proof}
As shown by the next two examples, the assumption that 
$\,(0,1)\ne(\ea,\eb)\ne(1,0)\,$ (or, equivalently, $\,\ea+\eb\ne1$) is 
essential for conclusion (ii) in  Proposition~\ref{wnonz}. In 
Section\/~\ref{lore} the same connections $\,\nabla(0,1)\,$ and 
$\,\nabla(1,0)\,$ are realized on Lo\-rentz\-i\-an quadric surfaces in a 
$\,3$-space.
\begin{example}\label{slsgp}Let $\,y^{\hn1}\nnh,y^2$ be the Cartesian 
coordinates in $\,\bs=\rto\nnh$. For the vector fields 
$\,\eu=(0,1/y^{\hn1})\,$ on the open set $\,\bs\hh'\nh\subset\bs\,$ where 
$\,y^{\hn1}\nh\ne0$, and $\,\ew=(2y^{\hn1}\nnh,0)\,$ on $\,\bs$, we have 
$\,[\eu,\ew]=2\eu$. Furthermore, the connection $\,\nabla\,$ defined by 
(\ref{nuu}) with $\,(\ea,\eb)=(0,1)\,$ has a $\,C^\infty$ extension from 
$\,\bs\hh'$ to $\,\bs$, since $\,\nabla_{\hskip-2.2pt(1,0)}(1,0)=0$, 
$\,\nabla_{\hskip-2.2pt(1,0)}(0,1)=\nabla_{\hskip-2.2pt(0,1)}(1,0)
=(3y^{\hn1}\nnh,0)\,$ and $\,\nabla_{\hskip-2.2pt(0,1)}(0,1)=(0,3y^{\hn1})$, 
as one sees noting that $\,(1,0)=\ew/(2y^{\hn1})$, $\,(0,1)=y^{\hn1}\eu$, 
$\,d_\eu y^{\hn1}\nh=0\,$ and $\,d_\ew y^{\hn1}\nh=2y^{\hn1}\nnh$. Our 
$\,\bs,\bs\hh'$ and $\,\nabla\,$ thus satisfy the assumptions of 
Proposition~\ref{wnonz} except for the condition $\,\ea+\eb\ne1$, and 
conclusion (ii) fails to hold: $\,\ew=0\,$ on %the $\,y^2$ axis 
$\,\bs\smallsetminus\bs\hh'\nnh$. (The Ric\-ci tensor $\,\rho\,$ is nonzero 
everywhere in $\,\bs$, since, by (\ref{ruw}.a), 
$\,6=\rho\hs(\hn\eu,\ew)=2\rho\hs((0,1),(1,0))\,$ on $\,\bs\hh'\nnh$.)

Both vector fields $\,\eu,\ew$, and hence also the connection $\,\nabla\nh$, 
are easily seen to be invariant under the group $\,\hp\,$ of af\-fine 
transformations of $\,\rto$ having a diagonal linear part of determinant 
$\,1\,$ and a translational part parallel to the $\,y^2$ axis 
$\,\bs\smallsetminus\bs\hh'\nnh$.
\end{example}
\begin{example}\label{slinv}Let $\,\plane\,$ be a 
\hbox{two\hh-}\hskip0ptdi\-men\-sion\-al real vector space with a fixed area 
form $\,\varOmega$. Thus, $\,\varOmega\,$ is an element of 
$\,[\plane^*]^{\wedge2}\nnh\smallsetminus\{0\}$, treated as a constant 
$\,2$-form on $\,\plane$. Denoting by $\,\ew\,$ the radial (identity) vector 
field on $\,\plane\,$ and by $\,c\,$ a fixed nonzero real constant, we define 
a connection $\,\nabla\hn\,$ on $\,\plane\,$ by requiring that, for all vector 
fields $\,\eu\,$ and all constant vector fields $\,\ev\,$ on $\,\plane$,
\begin{equation}\label{nuv}
\nsu\ev\,\,
=\,\,2\hh c\hs[\hh\varOmega(\ew,\eu)\hh\ev\,+\,\varOmega(\ew,\ev)\hh\eu]\,
-\,c\hh^2\varOmega(\ew,\eu)\hs\varOmega(\ew,\ev)\hh\ew\hs.
\end{equation}
Using $\,\eu,\ev\,$ which are both constant, one sees that $\,\nabla\hn\,$ is 
tor\-sion\-free and its pull\-back under any linear iso\-mor\-phism 
$\,A:\plane\to\plane\,$ is an analogous connection corresponding, instead of 
$\,c$, to $\,c\hskip2.4pt\mathrm{det}\hskip1ptA$. Thus, $\,\nabla\hn\,$ is 
invariant under the action of the uni\-mod\-u\-lar group 
$\,\mathrm{SL}\hs(\plane)$, and, although $\,\nabla\hn\,$ varies with the 
parameter $\,c$, its dif\-feo\-mor\-phic e\-quiv\-a\-lence class is 
independent of $\,c$.

Next, (iii) in Section~\ref{prel} easily implies both that 
$\,\ew=\varOmega(\ew,\ev\hh'\hh)\hh\ev-\varOmega(\ew,\ev)\hh\ev\hh'$ for 
constant vector fields $\,\ev,\ev\hh'$ with 
$\,\varOmega(\hn\ev,\ev\hh'\hh)=1$, and that, as a result,
\begin{equation}\label{nuw}
\nabla_{\hskip-2.2pt\eu}\ew\,
=\,\eu\,+\,2\hh c\hh\varOmega(\ew,\eu)\hh\ew\hs\hskip16pt\mathrm{for\ all\ 
vector\ fields\ }\,\eu\hh,
\end{equation}
even if one replaces $\,c\hh^2$ in (\ref{nuv}) with just any constant 
$\,c\hh'\nnh$. Let us now fix a nonzero constant vector field $\,\ev\,$ on 
$\,\plane\,$ and set $\,\eu=[\hh c\hh\varOmega(\ew,\ev)]^{-1}\ev\,$ on the 
open subset $\,\bs\hh'\nh=\plane\smallsetminus\bbR\ev$. Then (\ref{nuv}) and 
(\ref{nuw}) yield (\ref{nuu}) for $\,(\ea,\eb)=(1,0)\,$ and our 
$\,\eu,\ew$. Again, the assumptions of Proposition~\ref{wnonz} hold in this 
case, with $\,\bs=\plane$, except for $\,\ea+\eb\ne1$, and conclusion (iii) 
fails, as $\,\ew=0\,$ at $\,0$. (The Ric\-ci tensor $\,\rho\,$ is nonzero 
everywhere in $\,\bs$, since, by (\ref{ruw}.a), 
$\,c\rho=6\hs\varOmega$.)
\end{example}
\begin{lemma}\label{slivc}If\/ $\,\plane\,$ is a 
\hbox{two\hh-}\hskip0ptdi\-men\-sion\-al real vector space with a fixed area 
form\/ $\,\varOmega$, and an RSTS connection\/ $\,\nabla\hn\,$ on a nonempty 
connected open set $\,\,U\subset\plane\,$ is invariant under the infinitesimal 
action of\/ $\,\mathrm{SL}\hs(\plane)$, then\/ $\,\nabla\hn\,$ satisfies\/ 
{\rm(\ref{nuv})} for some\/ $\,c\in\bbR\hh$, all vector fields $\,\eu$, and 
all constant vector fields\/ $\,\ev\,$ on\/ $\,\,U\nh$.
\end{lemma}
\begin{proof}Let $\,\hp\subset\mathrm{SL}\hs(\plane)\,$ be the isotropy 
subgroup of a fixed point $\,y\in U\smallsetminus\{0\}$. Thus, $\,\bbR y\,$ is 
the only line (\hbox{one\hh-}\hskip0ptdi\-men\-sion\-al vector sub\-space) in 
$\,T\hskip-2pt_yU=\plane\,$ with the property of 
$\,\hp${\it-in\-var\-i\-ance}, here meaning invariance under the infinitesimal 
action of $\,\hp$. 

Multiples of $\,\varOmega(y,\,\cdot\,)\otimes\varOmega(y,\,\cdot\,)\,$ are, in 
turn, the only $\,\hp$-in\-var\-i\-ant symmetric $\,2$-ten\-sors $\,\tau\,$ at 
$\,y$. In fact, such $\,\tau$, if nonzero, must be of rank $\,1$, so that its 
null\-space, being an $\,\hp$-in\-var\-i\-ant line, must coincide with 
$\,\bbR y$, the null\-space of 
$\,\varOmega(y,\,\cdot\,)\otimes\varOmega(y,\,\cdot\,)$. (The rank of 
$\,\tau\,$ cannot be $\,2$, or else $\,\tau\,$ would be a 
pseu\-d\hbox{o\hs-}\hskip0ptEuclid\-e\-an inner product, and so $\,\bbR y\,$ 
would give rise to a second $\,\hp$-in\-var\-i\-ant line: the 
$\,\tau$-or\-thog\-o\-nal complement of $\,\bbR y$, if $\,\bbR y\,$ is not 
$\,\tau$-null, or the other $\,\tau$-null line, if $\,\bbR y\,$ is 
$\,\tau$-null.)

Let $\,\mathrm{D}\,$ be the restriction to $\,\,U\,$ of the standard flat 
connection on $\,\plane$. The difference $\,\Xi=\nabla\nh-\mathrm{D}\,$ is an 
$\,\mathrm{SL}\hs(\plane)$-in\-var\-i\-ant section of 
$\,[T^*\nnh U]^{\odot2}\nnh\otimes TU\nh$, and its value at $\,y\,$ is an 
$\,\hp$-in\-var\-i\-ant symmetric bi\-lin\-e\-ar mapping 
$\,\Xi_y:\plane\times\plane\to\plane$. Since the $\,\hp$-in\-var\-i\-ant 
symmetric $\,2$-ten\-sor $\,\tau=\varOmega(y,\Xi_y(\,\cdot\,,\,\cdot\,))\,$ 
must equal 
$\,-\hs c\hs\varOmega(y,\,\cdot\,)\otimes\varOmega(y,\,\cdot\,)/4\,$ for 
some $\,c\in\bbR\hh$, so that 
$\,\Xi_y(\hn\eu,\ev)
-2\hh c\hs[\hh\varOmega(y,\eu)\hh\ev\,+\,\varOmega(y,\ev)\hh\eu]\,$ lies, 
for all $\,\eu,\ev\in\plane$, in 
$\,\mathrm{Ker}\hskip2pt\varOmega(y,\,\cdot\,)=\bbR y$. Therefore, if 
$\,\eu,\ev\,$ are {\it constant\/} vector fields, 
$\,\nsu\ev=\Xi(\hn\eu,\ev)\,$ is given by the formula obtained from 
(\ref{nuv}) by replacing the coefficient $\,c\hh^2$ with a constant $\,c\hh'$ 
unrelated to $\,c$. (The values of $\,\Xi\,$ at points other than $\,y\,$ are 
the images of $\,\Xi_y$ under the infinitesimal action of 
$\,\mathrm{SL}\hs(\plane)$.) Using (\ref{cur}), (\ref{nuw}) (still valid in 
this case), and (iii) in Section~\ref{prel} we get 
$\,R\hh(\hn\eu,\ev)\hh\ev\hh'\nh=-2\hh\varOmega(\hn\eu,\ev)\hs
[3\hh c\hh\ev\hh'\nh
+2(c\hh^2\nh-c\hh'\hh)\hs\varOmega(\ew,\ev\hh'\hh)\hh\ew\hh]\,$ for the 
curvature tensor $\,R\,$ of $\,\nabla\hn\,$ and all constant vector fields 
$\,\eu,\ev,\ev\hh'\nnh$. Thus, the Ric\-ci tensor of $\,\nabla\hn\,$ is 
skew-sym\-met\-ric if and only if $\,c\hh'\nh=c\hh^2\nnh$, which completes the 
proof.
\end{proof}
\begin{remark}\label{loclg}We will use the following well-known fact. Let 
$\,e_j$, $\,j=1,\dots,n\hs$, be vector fields on an 
\hbox{$\,n$-}\hskip0ptdi\-men\-sion\-al manifold $\,\bs$, trivializing the 
tangent bundle $\,\tb\,$ and spanning an 
\hbox{$\,n$-}\hskip0ptdi\-men\-sion\-al Lie algebra. (Thus, the Lie 
brackets $\,[e_j,e_k]\,$ are constant-coefficient combinations of 
$\,e_1^{\phantom i},\dots,e_n$.) Then, locally, $\,\bs\,$ may be 
dif\-feo\-mor\-phi\-cal\-ly identified with a Lie group so that 
$\,e_1^{\phantom i},\dots,e_n$ correspond to left-in\-var\-i\-ant vector 
fields. See, for instance, \cite[Appendix~B]{derdzinski-08}.
\end{remark}

\section{The moduli curve of locally homogeneous RSTS connections}\label{mcrs}
\setcounter{equation}{0}
Let $\,\nabla\hn\,$ be a tor\-sion\-free connection on a surface $\,\bs\,$ 
such that the Ric\-ci tensor $\,\rho\,$ of $\,\nabla\hn\,$ is 
skew-sym\-met\-ric and nonzero everywhere, and let $\,\ew\,$ be the vector 
field with (\ref{rec}.ii). As in Section~\ref{prel}, we denote by 
$\,\mathfrak{a}_y$ the Lie algebra of germs, at $\,y\in\bs$, of all 
infinitesimal af\-fine transformations of $\,\nabla\nnh$. If $\,\ew_y\ne0$, 
then $\,\dim\hs\mathfrak{n}_y\le1\,$ for 
$\,\mathfrak{n}_y=\{\ev\in\mathfrak{a}_y:\ev_y=0\}$, that is,
\begin{equation}\label{isa}
\mathrm{the\ isotropy\ subalgebra\ \ }\mathfrak{n}_y\,\mathrm{\ of\ \ 
}\mathfrak{a}_y\,\mathrm{\ is\ at\ most\ 
\hbox{one\hh-}\hskip0ptdi\-men\-sion\-al, \ and\ so\ }\,\dim\hs\mathfrak{a}_y
\le3\hh.
\end{equation}
Namely, the differentials at $\,y\,$ of af\-fine transformations in $\,\bs\,$ 
keeping $\,y\,$ fixed lie in the \hbox{one\hh-}\hskip0ptdi\-men\-sion\-al 
group of linear au\-to\-mor\-phism of $\,\tyb\,$ that preserve both the area 
form $\,\rho_y$ and the vector $\,\ew_y\ne0$, so that we get (\ref{isa}). By 
(\ref{isa}), for $\,y\in\bs\,$ with $\,\ew_y\ne0$,
\begin{equation}\label{oof}
\mathrm{the\ pair\ \ }(\dim\hs\mathfrak{a}_y,\hskip1pt\dim\hs\mathfrak{n}_y)
\mathrm{\ \ is\ one\ of\ \ }(3,1),\,(2,1),\,(2,0),\,(1,1),\,(1,0),\,(0,0)\hh.
\end{equation}
In addition, let $\,\ev_y\in\tyb\,$ be a vector naturally distinguished by 
$\,\nabla\nnh$. It follows that
\begin{equation}\label{let}
\mathrm{if\ }\,\ev_y\mathrm{\ and\ }\,\ew_y\mathrm{\ are\ linearly\ 
independent, \ then\ }\,\mathfrak{n}_y=\hs\{0\}\,\mathrm{\ and\ 
}\,\dim\hs\mathfrak{a}_y\le2\hh.
\end{equation}
In fact, the flows of elements of $\,\mathfrak{n}_y$ 
keep the basis $\,\ev_y,\ew_y$ of $\,\tyb\,$ fixed, while, in general,
\begin{equation}\label{aff}
\begin{array}{l}
\mathrm{an\ af\-fine\ transformation\ \ }F\,\mathrm{\ between\ two\ manifolds\ 
with\ tor\-sion\-free\ connec\mbox{-}}\\
\mathrm{tions\ is\ uniquely\ determined\ by\ its\ value\ and\ differential\ 
at\ any\ given\ point,}
\end{array}
\end{equation}
since, in geodesic coordinates, $\,F\,$ appears as a linear operator.
\begin{remark}\label{uslin}Let $\,\nabla\hn\,$ be a tor\-sion\-free connection 
on a surface $\,\bs\,$ such that the Ric\-ci tensor $\,\rho\,$ of 
$\,\nabla\hn\,$ is skew-sym\-met\-ric and nonzero at every point $\,y\in\bs$
\begin{enumerate}
  \def\theenumi{{\rm\alph{enumi}}}
\item[(a)] The inequality $\,\dim\hs\mathfrak{a}_y\le3\,$ in (\ref{isa}) 
remains valid, by (\ref{dns}), also when $\,\ew_y\nh=0$.
\item[(b)] Since $\,\dim\hs\mathrm{SL}\hs(\plane)=3$, (a) implies that the 
connections $\,\nabla\hn\,$ described in Example~\ref{slinv} have 
$\,\dim\hs\mathfrak{a}_y=3\,$ at each point $\,y$, while $\,\ew_y\nh=0\,$ if 
$\,y=0$.
\item[(c)] Conversely, if $\,\dim\hs\mathfrak{a}_z=3\,$ and $\,\ew_z\nh=0\,$ 
at a point $\,z\in\bs$, then the restriction of $\,\nabla\hn\,$ to some 
neighborhood of $\,z\,$ is dif\-feo\-mor\-phi\-cal\-ly equivalent to one of 
the connections in Example~\ref{slinv}.
\end{enumerate}
To verify (c), we first note that $\,(\naw)_z$ cannot have two distinct real 
eigenvalues: if it did, the same would be true of nearby points $\,y$, 
including, in view of (\ref{dns}), one with $\,\ew_y\nh\ne0\,$ and 
$\,\dim\hs\mathfrak{a}_y\nh=3$. The condition $\,\dim\hs\mathfrak{a}_y\nh=3\,$ 
would now contradict (\ref{let}) for $\,\ev_y$ chosen to be an eigenvector of 
$\,(\naw)_y$ such that $\,\rho\hs(\hn\ev_y,\ew_y)=1$. Consequently, 
$\,(\naw)_z$ must be a linear au\-to\-mor\-phism of 
$\,T\hskip-2pt_z\hskip-.2pt\bs$, for otherwise, according to (\ref{dfr}.a) and 
(\ref{bch}.b), $\,(\naw)_z$ would have the distinct real eigenvalues $\,0\,$ 
and $\,2$. Since the matrix $\,[(\partial_j\ew^{\hs k})(z)]\,$ of the 
components of $\,(\naw)_z$ in any local coordinates is nonsingular, the 
inverse mapping theorem implies that $\,z\,$ is an isolated zero of $\,\ew$. 
Thus, the flows of all elements of $\,\mathfrak{a}_y$ keep $\,z\,$ fixed, and 
so the isotropy subalgebra $\,\mathfrak{n}_y\nh=\mathfrak{a}_y$ is 
\hbox{three\hh-}\hskip0ptdi\-men\-sion\-al. The Lie algebra 
$\,\mathfrak{n}_y\nh=\mathfrak{a}_y$ treated as acting in 
$\,T\hskip-2pt_z\hskip-.2pt\bs\,$ preserves the area form $\,\rho_z$. Being 
\hbox{three\hh-}\hskip0ptdi\-men\-sion\-al, it must therefore coincide with 
$\,\mathfrak{sl}\hh(T\hskip-2pt_z\hskip-.2pt\bs)$, and (c) is immediate from 
Lemma~\ref{slivc} along with the line following (\ref{aff}).
\end{remark}
\begin{lemma}\label{dathr}Under the same assumptions as in 
Remark\/~{\rm\ref{uslin}}, let\/ $\,\phi\,$ and\/ $\,\ew\,$ be characterized 
by\/ {\rm(\ref{rec})}. If\/ $\,y\in\bs\,$ is a point with 
$\,\dim\hs\mathfrak{a}_y\nh=3\,$ and\/ $\,\ew_y\ne0$, then\/ $\,y\,$ has a 
connected neighborhood\/ $\,\,U\,$ such that\/ $\,\nabla\,$ restricted to\/ 
$\,\,U\,$ is locally homogeneous, and the only symmetric\/ $\,2$-ten\-sors 
on\/ $\,\,U\nh$, naturally distinguished by\/ $\,\nabla\nnh$, are constant 
multiples of\/ $\,\phi\otimes\phi$.
\end{lemma}
\begin{proof}Since $\,\ew_y\nh\ne0$, (\ref{isa}) implies local homogeneity of 
$\,\nabla\hn\,$ on some connected neighborhood $\,\,U\,$ of $\,y$. If 
$\,\sigma\,$ now is a symmetric $\,2$-ten\-sor with the stated properties, 
we must have $\,\sigma(\ew,\ew)=0$, for otherwise (\ref{let}) applied to the 
vector field $\,\ev\,$ with 
$\,\sigma(\ew,\,\cdot\,)=\rho\hs(\hn\ev,\,\cdot\,)\,$ would contradict the 
assumption that $\,\dim\hs\mathfrak{a}_y\nh=3$. Thus, if $\,\sigma\,$ were 
nondegenerate at $\,y$, it would have the Lo\-rentz\-i\-an signature 
$\,(\hbox{$-$\hskip1pt$+$})$, again leading to a contradiction with 
(\ref{let}), this time for $\,\ev\,$ chosen so that $\,\sigma(\hn\ev,\ev)=0\,$ 
and $\,\rho\hs(\hn\ev,\ew)=1$. Therefore, $\,\mathrm{rank}\hskip3pt\sigma$, 
obviously constant on $\,\,U\nh$, equals $\,1\,$ or $\,0$. If 
$\,\mathrm{rank}\hskip3pt\sigma=1$, we have 
$\,\sigma=\pm\hs\alpha\otimes\alpha\,$ for some $\,1$-form $\,\alpha\,$ 
without zeros, which is a constant multiple of $\,\phi\,$ due to the relation 
$\,\sigma(\ew,\ew)=0$, (\ref{rec}.iii) and local homogeneity of 
$\,\nabla\hn\,$ on $\,\,U\nh$.
\end{proof}
\begin{remark}\label{dwchi}For a tor\-sion\-free connection $\,\nabla\hn\,$ 
with eve\-ry\-\hbox{where\hskip1pt-}\hskip0ptnon\-zero, skew-sym\-met\-ric 
Ric\-ci tensor $\,\rho\,$ on a surface $\,\bs$, the vector field $\,\ew\,$ 
given by (\ref{rec}.ii), and a point $\,y\in\bs$, we have 
$\,\mathfrak{n}_y=\hs\{0\}\,$ and $\,\dim\hs\mathfrak{a}_y\le2\,$ whenever 
$\,d_\ew\psi\ne0\,$ at $\,y\,$ for some function $\,\psi\,$ which is 
naturally distinguished by $\,\nabla\nnh$, and defined on a neighborhood of 
$\,y$. (This is clear from (\ref{let}) for the vector field $\,\ev\,$ 
characterized by $\,d\psi=\rho\hs(\,\cdot\,,\ev)$.)
\end{remark}
\begin{theorem}\label{modul}The assignment\/ 
$\,(\ea,\eb)\mapsto\nabla\nh$, with\/ $\,\nabla=\nabla(\ea,\eb)\,$ defined as 
in Example\/~\ref{nonab}, establishes a bijective correspondence between 
\begin{enumerate}
  \def\theenumi{{\rm\romen{enumi}}}
\item[(i)] the union\/ $\,(\bbR\times\{0\})\cup(\{0\}\times\bbR)\,$ of the 
coordinate axes in the $\,(\ea,\eb)$-plane\/ $\,\rto\nnh$, and
\item[(ii)] the set of lo\-cal-e\-quiv\-a\-lence classes of locally 
homogeneous non\-flat tor\-sion\-free surface connections with 
skew-sym\-met\-ric Ric\-ci tensor.
\end{enumerate}
The degree of mobility of\/ $\,\nabla(\ea,\eb)$, defined as in 
Section\/~\ref{prel}, equals\/ $\,3\,$ for\/ $\,(\ea,\eb)=(1,0)$, and\/ 
$\,2\,$ when\/ $\,(\ea,\eb)\ne(1,0)$.
\end{theorem}
\begin{proof}That $\,\nabla=\nabla(\ea,\eb)\,$ all have the properties listed 
in (ii) is immediate from Example~\ref{nonab}. The final clause about the 
degrees of mobility $\,\dim\hs\mathfrak{a}_y$ is in turn an obvious 
consequence either of Remark~\ref{uwinv} combined with (\ref{let}), for 
$\,(\ea,\eb)\ne(1,0)$, or of Remark~\ref{uslin}(b), when $\,(\ea,\eb)=(1,0)$. 
Our assertion will thus follow if we show that, for any given non\-flat 
locally homogeneous RSTS connection $\,\nabla\nnh$,
\begin{enumerate}
  \def\theenumi{{\rm\roman{enumi}}}
\item[(iii)] $\nabla\hn\,$ is locally equivalent to $\,\nabla(\ea,\eb)\,$ for 
some $\,(\ea,\eb)\,$ with $\,\ea\eb=0$,
\item[(iv)] the pair $\,(\ea,\eb)\,$ in (iii) is uniquely determined by 
$\,\nabla\nnh$.
\end{enumerate}
To prove (iii) -- (iv), we first note that the degree of mobility 
$\,\dim\hs\mathfrak{a}_y$ is the same for all $\,y\in\bs\,$ and, by 
Remark~\ref{uslin}(a), equals $\,2\,$ or $\,3$. If 
$\,\dim\hs\mathfrak{a}_y=\hh2$, then $\,\nabla\hn\,$ is locally equivalent 
to a left-in\-var\-i\-ant connection on a connected Lie group $\,\hp$, and so 
$\,\nabla\hn\,$ has, locally, the form appearing in Lemma~\ref{liegp}, with 
suitable $\,\Psi\,$ and $\,\lf$, while $\,\hp\,$ is not Abel\-i\-an (for 
otherwise $\,\nabla\hn\,$ would be Ric\-ci-flat by Lemma~\ref{liegp}, and 
hence flat by (\ref{rer}.a)). Choosing a basis $\,\eu,\ew\,$ of the Lie 
algebra $\,\hi\,$ with $\,[\eu,\ew]=2\eu$, we get $\,\lf(\hn\eu)\ne0$, as 
$\,\nabla\hn\,$ is not Ric\-ci-flat. Thus, rescaling $\,\eu\,$ and adding to 
$\,\ew\,$ a multiple of $\,\eu$, we may also assume that $\,\lf(\hn\eu)=3\,$ and 
$\,\lf(\ew)=0$. If $\,\mathfrak{B}_\eu,\mathfrak{B}_\ew$ are the matrices 
representing $\,\Psi\eu\,$ and $\,\Psi\ew\,$ in the basis $\,\eu,\ew$, the 
equality $\,(\Psi\eu)\ew-(\Psi\ew)\eu=2\eu-3\ew$, meaning that $\,\nabla\hn\,$ 
is tor\-sion\-free, amounts to
\[
\mathfrak{B}_\eu\textstyle{\left[\begin{matrix}0\\1\end{matrix}\right]}
-\mathfrak{B}_\ew\textstyle{\left[\begin{matrix}1\\0\end{matrix}\right]}
=\textstyle{\left[\begin{matrix}2\\-3\end{matrix}\right]},\hskip11pt
\mathrm{and\ hence\ \ }\mathfrak{B}_\eu\,
=\hs\left[\begin{matrix}\ea&c\\-s&-\ea\end{matrix}\right],\hskip11pt
\mathfrak{B}_\ew\,
=\hs\left[\begin{matrix}c-2&\ea+\eb-1\\3-\ea&2-c\end{matrix}\right]
\]
for some $\,\ea,\eb,c,s\in\bbR\hh$. As $\,\Psi\,$ is a Lie-al\-ge\-bra 
homomorphism (see Lemma~\ref{liegp}),
\begin{equation}\label{apb}
\mathrm{i)}\hskip7pt(\ea+\eb-1)\hs s=(\ea-3)\hs c+2\ea\hh,\hskip11pt
\mathrm{ii)}\hskip7pt(\ea+\eb-1)\hs\ea=(c-1)\hs c\hh,\hskip11pt
\mathrm{iii)}\hskip7pt(\ea-3)\hs\ea=(c-3)\hs s\hh,
\end{equation}
or, in other words, 
$\,\mathfrak{B}_\eu\mathfrak{B}_\ew-\mathfrak{B}_\ew\mathfrak{B}_\eu
=2\hh\mathfrak{B}_\eu$. From (\ref{apb}) it follows that $\,c=s=\ea\,$ and 
$\,\ea\eb=0\,$ (which yields (\ref{bue}), and hence (\ref{nuu}), thus proving 
(iii) when $\,\dim\hs\mathfrak{a}_y=\hh2$). Namely, (\ref{apb}.i) and 
(\ref{apb}.ii) give $\,[(\ea-3)\hs c+2\ea]\hs\ea=(c-1)\hs c\hh s$, so that, 
by (\ref{apb}.iii), $\,(c-3)\hs c\hh s+2\ea^2\nh=(c-1)\hs c\hh s$, and, 
consequently, $\,c\hh s=\ea^2\nnh$. Again using (\ref{apb}.i) and 
(\ref{apb}.ii), we thus obtain 
$\,[(\ea-3)\hs c+2\ea]\hs c=(\ea+\eb-1)\hs c\hh s=(\ea+\eb-1)\hs\ea^2\nh
=(c-1)\hs c\hh\ea$, that is, $\,(\ea-c)\hs c=0$. There are now two cases: 
$\,c=0$, and $\,\ea=c$. If $\,c=0$, (\ref{apb}.ii) implies that 
$\,\ea+\eb=1\,$ or $\,\ea=0$, so that $\,\ea=0\,$ by (\ref{apb}.i) with 
$\,c=0$, and $\,s=0\,$ from (\ref{apb}.iii) with $\,c=0$. Hence $\,c=s=\ea\,$ 
and $\,\ea\eb=0$, as required. If $\,\ea=c$, (\ref{apb}.ii) yields 
$\,\ea\eb=0$, and either $\,\ea=c\ne3\,$ (which, by (\ref{apb}.iii), gives 
$\,\ea=s$), or $\,\ea=c=3\,$ (and, as the equality $\,\ea\eb=0\,$ now reads 
$\,\eb=0$, (\ref{apb}.i) becomes 
$\,6=(\ea-3)\hs c+2\ea=(\ea+\eb-1)\hs s=(\ea-1)\hs s=2s$, so that, again, 
$\,c=s=\ea$).

Next, let $\,\dim\hs\mathfrak{a}_y=\hh3\,$ at every point $\,y$. Then the 
tangent bundle of the underlying surface $\,\bs\,$ is, locally, trivialized by 
$\,\ew\,$ with (\ref{rec}.ii) and some vector field $\,\eu\,$ with (\ref{ruw}) 
such that
\begin{equation}\label{nue}
\begin{array}{ll}
\nabla_{\hskip-2.2pt\eu}\eu=4\hh\eu-\sy\ew\hh,
&\nabla_{\hskip-2.2pt\eu}\ew=\eu+2\hh\ew\hh,\\
\nabla_{\hskip-2.2pt\ew}\eu=-\hs\eu+2\hh\ew\hh,\hskip10pt
&\nsww=\ew
\end{array}
\end{equation}
for some function $\,\sy$. In fact, $\,\ew\ne0\,$ everywhere due to 
(\ref{dns}) and local homogeneity of $\,\nabla$, so that, by (\ref{rec}), we 
may, locally, choose $\,\eu\,$ with (\ref{ruw}). Such $\,\eu\,$ is unique up 
to being replaced by $\,\eu+\psi\hh\ew$, for an arbitrary function $\,\psi$. 
We may further require that $\,[\eu,\ew]=2\eu$, and then
\begin{equation}\label{unq}
\eu\,\mathrm{\ becomes\ unique\ up\ to\ replacement\ by\ 
}\,\eu+\psi\hh\ew\,\mathrm{\ for\ a\ function\ }\,\psi\,\mathrm{\ with\ 
}\,d_\ew\psi=-2\psi\hh.
\end{equation}
(Namely, (\ref{ruw}), (\ref{rec}.iv) and (\ref{bwa}.b) yield 
$\,12=2\rho\hs(\hn\eu,\ew)=(d\hh\phi)(\hn\eu,\ew)=-\hs\phi([\eu,\ew])$, and so 
$\,[\eu,\ew]-2\eu\,$ equals a function times $\,\ew$, which, locally, gives 
$\,[\eu,\ew]=2\eu\,$ provided that instead of $\,\eu\,$ one uses 
$\,\eu+\psi\hh\ew$, for suitable $\,\psi$.) On the other hand, as 
$\,\ew\ne0\,$ everywhere, we have $\,\nsww=\kappa\hh\ew\,$ for some function 
$\,\kappa:\bs\to\bbR\,$ (or else, setting $\,\ev=\nsww\,$ in (\ref{let}), we 
would get $\,\dim\hs\mathfrak{a}_y\le\hh2\,$ at some point $\,y$). Now, by 
(\ref{dfr}.a) and (\ref{bch}.b), the functions $\,\kappa\hs$ and 
$\,2-\kappa\,$ form the eigenvalues of $\,\naw:\tb\to\tb\,$ at every point of 
$\,\bs$, which implies that $\,\kappa=1\,$ everywhere, since otherwise 
(\ref{let}), applied to $\,\ev\,$ defined by the conditions 
$\,\nabla_{\hskip-2.2pt\ev}\ew=(2-\kappa)\hh\ev\,$ and 
$\,\rho\hs(\hn\ev,\ew)=1$, would again give $\,\dim\hs\mathfrak{a}_y\le\hh2\,$ 
at some $\,y$. Consequently, $\,\nsww=\ew\,$ and 
$\,\nabla_{\hskip-2.2pt\eu}\ew=\eu+\chi\hh\ew\,$ for some function $\,\chi$. 
Thus, by (\ref{ruw}.b), $\,\naw=\mathrm{Id}-\chi\hs\phi\otimes\ew/6$, which shows 
that $\,\chi\,$ is naturally distinguished by $\,\nabla\nnh$, and so 
$\,d_\ew\chi=0\,$ everywhere (or else Remark~\ref{dwchi} would give 
$\,\dim\hs\mathfrak{a}_y\le\hh2\,$ at some $\,y$). From (\ref{ruw}.a), 
(\ref{rer}.a) and (\ref{cur}) we now get 
$\,6\hh\ew=\rho\hs(\hn\eu,\ew)\hs\ew=R\hh(\hn\eu,\ew)\hh\ew
=\nsw(\hn\eu+\chi\hh\ew)-\nabla_{\hskip-2.2pt\eu}\ew+2\nabla_{\hskip-2.2pt\eu}\ew
=-\hs[\eu,\ew]+\chi\hh\ew+2\nabla_{\hskip-2.2pt\eu}\ew=3\chi\hh\ew$, and hence 
$\,\chi=2$. Therefore, $\,\nabla_{\hskip-2.2pt\eu}\ew=\eu+2\hh\ew$. As 
$\,\rho\hs(\hn\eu,\ew)=6\,$ is constant, cf.\ (\ref{ruw}.a), differentiation 
by parts gives $\,\rho\hs(\nsu\eu,\ew)=-\hs[\nsu\rho\hs]\hs(\hn\eu,\ew)
-\rho\hs(\hn\eu,\nsu\ew)=4\rho\hs(\hn\eu,\ew)$. (Note that 
$\,\nsu\rho=\phi(\nh\eu)\hh\rho=-\hs6\rho\,$ by (\ref{rec}.i) and 
(\ref{ruw}.b).) Hence $\,\nabla_{\hskip-2.2pt\eu}\eu=4\hh\eu-\sy\ew\,$ 
for some function $\,\sy$. Combined with the equality $\,[\eu,\ew]=2\eu$, this 
proves (\ref{nue}).

Also, as in the last paragraph, 
$\,6\hh\eu=\rho\hs(\hn\eu,\ew)\hs\eu=R\hh(\hn\eu,\ew)\hh\eu
=\nsw(4\hh\eu-\sy\ew)-\nabla_{\hskip-2.2pt\eu}(2\hh\ew-\eu)
+2\nabla_{\hskip-2.2pt\eu}\eu$, so that (\ref{nue}) yields 
$\,d_\ew\sy=4\hh(1-\sy)$. However, $\,\sy\,$ depends on the 
choice of $\,\eu$, and we are free to modify $\,\eu\,$ as in (\ref{unq}). 
Some such modification gives $\,\sy=1\,$ (and so (\ref{nue}) becomes 
(\ref{nuu}) with $\,(\ea,\eb)=(1,0)$, proving (iii) also in the case 
$\,\dim\hs\mathfrak{a}_y=\hh3$). Namely, the use of 
$\,\eu+\psi\hh\ew\,$ instead of $\,\eu$, with $\,d_\ew\psi=-2\psi$, 
causes $\,\sy\,$ to be replaced by $\,\sy+\psi^2\nh-d_\eu\psi$, 
while $\,\sy+\psi^2\nh-d_\eu\psi=1\,$ if and only if  $\,\psi\,$ is 
a solution of the system $\,d_\eu\psi=\psi^2\nh+\sy-1$, 
$\,d_\ew\psi=-2\psi$. As $\,d_\ew\sy=4\hh(1-\sy)$, this system is 
completely integrable. In fact, it implies its own integrability conditions, 
and so the graphs of its solutions are, locally, the integral manifolds of 
a distribution on $\,\bs\times\bbR\hh$, which happens to be integrable.

Finally, to obtain (iv), note that, according to the final clause of the 
theorem, the case $\,(\ea,\eb)=(1,0)\,$ is uniquely distinguished by the value 
of the degree of mobility, while, if $\,(\ea,\eb)\ne(1,0)$, the invariant 
character of both $\,\eu\,$ and $\,\ew$, established in Remark~\ref{uwinv}, 
allows us to treat the last equality in (\ref{nuu}) as an explicit geometric 
definition of $\,\ea\,$ and $\,\eb$.
\end{proof}

\section{The Kil\-ling equation for an RSTS connection}\label{kers}
\setcounter{equation}{0}
If the Ric\-ci tensor $\,\rho\,$ of a tor\-sion\-free connection 
$\,\nabla\hn\,$ on a surface $\,\bs\,$ is skew-sym\-met\-ric and non\-zero at 
every point, (\ref{rer}.a) and (\ref{bta}.a) allow us to rewrite (\ref{nvn}) as
\begin{equation}\label{nnx}
\nabla\nabla\xi\,=\,(\ob\tau-\hs\xi)\otimes\rho\,+\,\nabla\tau\hskip12pt
\mathrm{for\ any\ }\,1\hyp\mathrm{form}\hskip5.5pt\xi\hskip5pt\mathrm{on}
\hskip5pt\bs\hh,\hskip6pt\mathrm{with}\hskip5pt\tau=\ki\hh\xi\hh.
\end{equation}
\begin{lemma}\label{xiqeq}Let\/ $\,\nabla\hn\,$ be a tor\-sion\-free 
connection with eve\-ry\-\hbox{where\hskip1pt-}\hskip0ptnon\-zero, 
skew-sym\-met\-ric Ric\-ci tensor $\,\rho\,$ on a surface\/ $\,\bs$. For any\/ 
$\,1$-form\/ $\,\xi\,$ on $\,\bs$, setting\/ $\,\tau=\ki\hh\xi$, we then have
\begin{equation}\label{xqe}
\mathrm{a)}\hskip6pt\nd^*\xi\,=\,\oz\tau\hh,\hskip22pt\mathrm{b)}\hskip6pt
\nabla\xi\,=\,\tau\,+\,[\hh\xi(\ew)-2\hh\od(\ob\tau)]\hh\rho/4\hh,
\end{equation}
with\/ $\,\ew\,$ given by\/ {\rm(\ref{rec}.ii)}, $\,\oz,\ob,\od\,$ as in\/ 
{\rm(\ref{bta})} -- {\rm(\ref{zte})}, and\/ $\,\nd^*$ standing for the 
dual of the morphism\/ $\,\nd\,$ in\/ {\rm(\ref{qef}.i)}, cf.\ {\rm(\ref{ast})}.
\end{lemma}
\begin{proof}Being skew-sym\-met\-ric, 
$\,\nabla\xi-\ki\hh\xi=\nabla\xi-\tau\,$ equals $\,\psi\rho\,$ for some 
function $\,\psi$. Now
\begin{equation}\label{dps}
\mathrm{i)}\hskip9ptd\psi\,=\,\ob\tau\,-\,\hs\xi\,-\,\psi\hh\phi\hh,
\hskip12pt\mathrm{ii)}\hskip9ptd\hh\xi\,
=\,[\hh\xi(\ew)-2\psi-2\od(\ob\tau)]\hs\rho\hh.
\end{equation}
In fact, (\ref{nnx}) yields (\ref{dps}.i) since, by (\ref{rec}.i), 
$\,\nabla\nabla\xi=\nabla(\psi\rho+\tau)=(d\psi+\psi\hh\phi)\otimes\rho
+\nabla\tau$. Next, (\ref{dps}.i) gives 
$\,d(\ob\tau-\hs\xi-\psi\hh\phi)=0\,$ and so, again from (\ref{dps}.i), 
$\,d\hh\xi=d(\ob\tau)+\phi\wedge\hh d\psi-\psi\hs d\hh\phi=d(\ob\tau)
+\phi\wedge\ob\tau+\xi\wedge\phi-\psi\hs d\hh\phi$. Using (\ref{rec}.iv), 
(\ref{bta}.b) and the relation $\,\xi\wedge\phi=\xi(\ew)\hs\rho$, immediate 
from (\ref{rer}.b) with $\,\beta=\hh\xi$, $\,\eu=\ew\,$ and (\ref{rec}.ii), we 
thus get (\ref{dps}.ii).

However, (\ref{bwa}.c) and the equality $\,\nabla\xi=\psi\rho+\tau\,$ give 
$\,d\hh\xi=2\psi\rho$. Equating this expression for $\,d\hh\xi\,$ with 
(\ref{dps}.ii), we obtain $\,4\psi=\hh\xi(\ew)-2\od(\ob\tau)$, which, as 
$\,\nabla\xi=\psi\rho+\tau$, implies (\ref{xqe}.b). In view of 
(\ref{xqe}.b) and (\ref{rec}.ii), for any vector field $\,\ev\,$ we have
\begin{equation}\label{fnv}
4(\nsv\hh\xi)(\ew)\,=\,4\tau(\ew,\ev)
+[2\hh\od(\ob\tau)-\hh\xi(\ew)]\hh\phi(\hn\ev)\hh.
\end{equation}
Applying $\,\nsv$ to (\ref{xqe}.b), we see that $\,4\nsv(\nabla\xi-\tau)\,$ 
equals $\,\xi(\nsvw)-2\hs d_v\hh[\od(\ob\tau)]+\tau(\ew,\ev)
-3\hh[2\hh\od(\ob\tau)-\hh\xi(\ew)]\hh\phi(\hn\ev)/4\,$ times $\,\rho$, due to 
the Leib\-niz rule, (\ref{fnv}) and the relation $\,\nsv\rho=\phi(v)\hs\rho$, 
cf.\ (\ref{rec}.i). Since, by (\ref{nnx}), 
$\,4\nsv(\nabla\xi-\tau)=4\hh[(\ob\tau-\hs\xi)(v)]\hs\rho$, (\ref{xqe}.a) 
follows.
\end{proof}
Lemma~\ref{xiqeq} immediately leads to the following conclusion about 
$\,1$-forms $\,\xi\in\mathrm{Ker}\,\ki$, that is, $\,C^\infty$ solutions 
$\,\xi\,$ to the {\it Kil\-ling equation\/} $\,\ki\hh\xi=0$, where $\,\ki\,$ 
is the Kil\-ling operator with (\ref{cod}.b).
\begin{theorem}\label{kileq}For the Kil\-ling operator\/ $\,\ki\,$ of a 
surface\/ $\,\bs\,$ with a tor\-sion\-free connection\/ $\,\nabla\hn\,$ such 
that the Ric\-ci tensor\/ $\,\rho\,$ of\/ $\,\nabla\hn\,$ is 
skew-sym\-met\-ric and nonzero everywhere,
\begin{enumerate}
  \def\theenumi{{\rm\roman{enumi}}}
\item[(i)] $\dim\hskip2pt\mathrm{Ker}\,\ki\le1$,
\item[(ii)] each $\,1$-form $\,\xi\in\mathrm{Ker}\,\ki\,$ is either 
identically zero, or nonzero at every point.
\end{enumerate}
\end{theorem}
\begin{proof}Equality (\ref{xqe}.b) with $\,\tau=0\,$ states that any 
$\,\xi\in\mathrm{Ker}\,\ki$, restricted to any geodesic, satisfies a 
first-or\-der linear homogeneous ordinary differential equation, which proves 
(ii). If we now had $\,\dim\hskip2pt\mathrm{Ker}\,\ki\ge2$, formula 
(\ref{xqe}.a) with $\,\tau=0\,$ would, by (ii), imply that $\,\nd=0\,$ at 
every point, contradicting (\ref{qef}.ii). 
\end{proof}
We say that a tor\-sion\-free connection $\,\nabla\hn\,$ on a surface 
$\,\bs\,$ with skew-sym\-met\-ric Ric\-ci tensor $\,\rho\,$ is {\it generic\/} 
if $\,\rho\ne0\,$ at every point and $\,\nd\,$ defined by (\ref{qef}.i) is an 
isomorphism $\,\tb\to\tb$. For such a generic connection $\,\nabla\nnh$, 
formula (\ref{xqe}.a) leads to a complete description of both the kernel and 
the image of the Kil\-ling operator $\,\ki$. In fact, by 
(\ref{xqe}.a),
\begin{equation}\label{gkz}
\mathrm{if}\hskip6pt\nabla\hskip6pt\mathrm{is\ generic,}\hskip12pt
\mathrm{Ker}\,\ki=\{0\}\hh.
\end{equation}
About the image of $\,\ki$, see Section~\ref{ioko}. Here we just note that, 
for $\,\nabla(\ea,\eb)\,$ as in Example~\ref{nonab},
\begin{equation}\label{nab}
\nabla(\ea,\eb)\hskip7pt\mathrm{is\ generic\ except\ when}\hskip7pt
(\ea,\eb)=(-9,0)\hskip7pt\mathrm{or}\hskip7pt(\ea,\eb)=(0,-15)\hh.
\end{equation}
Namely, by (\ref{nuu}) and (\ref{ruw}.b), at each point $\,y$, in the basis 
$\,\eu_y,\ew_y$ of the tangent plane,
\begin{equation}\label{mat}
\nd\mathrm{\ \ \ is\ represented\ by\ the\ matrix\ \ }
\left[\begin{matrix}\ea\nh+\nh4&\ea\nh+\nh\eb\nh-\nh1\\
-\hs\ea\nh-\nh3/2&6\nh-\nh\ea\end{matrix}\right]\nnh.
\end{equation}
Thus, $\,2\hskip2.4pt\mathrm{det}\hskip1.9pt\nd=5\ea+3\eb+45$, and so 
(\ref{nab}) follows since $\,\ea\eb=0$.
\begin{example}\label{dkleo}For the non-ge\-ner\-ic locally homogeneous 
connection $\,\nabla(-9,0)\,$ defined in Example~\ref{nonab} with 
$\,(\ea,\eb)=(-9,0)$, we have $\,\dim\hskip2pt\mathrm{Ker}\,\ki=1$. In fact, 
choosing $\,\ef\,$ with (\ref{duf}), and then setting 
$\,\xi(\hn\eu)=3\ef^2\nnh$, $\,\xi(\ew)=2\ef^2\nnh$, we define a nonzero 
$\,1$-form $\,\xi\,$ such that $\,\ki\hh\xi=0$. By Theorem~\ref{kileq}(i), 
$\,\xi\,$ spans $\,\mathrm{Ker}\,\ki$.
\end{example}
\begin{example}\label{kerkz}The remaining non-ge\-ner\-ic locally homogeneous 
connection $\,\nabla(0,-15)\,$ in (\ref{nab}) has $\,\mathrm{Ker}\,\ki=\{0\}$. 
To see this, note that a nonzero $\,1$-form $\,\xi\,$ with $\,\ki\hh\xi=0$, 
if it existed, would give rise to the line sub\-bun\-dle 
$\,\mathrm{Ker}\,\xi\,$ in $\,\tb\,$ (cf.\ Theorem~\ref{kileq}(ii)), and, by 
(\ref{xqe}.a) with $\,\tau=0$, the image of $\,\nd\,$ would be contained in 
$\,\mathrm{Ker}\,\xi$. From (\ref{mat}) with $\,(\ea,\eb)=(0,-15)\,$ it would 
now follow that $\,8\eu-3\ew\,$ spans $\,\mathrm{Ker}\,\xi$, that is, 
$\,\xi(\hn\eu)=3\psi\,$ and $\,\xi(\ew)=8\psi\,$ for some function $\,\psi$, 
not identically equal to $\,0$. The relation $\,\ki\hh\xi=0\,$ would now yield 
$\,0=[\nsu\xi](\hn\eu)=d_\eu\hh[\hh\xi(\hn\eu)]
-\xi(\nsu\hh\eu)=3\hh(d_\eu\psi-3\psi)\,$ and, similarly, 
$\,0=[\nsw\xi](\ew)=8\hh(d_\ew\psi+4\psi)$. The resulting 
equalities $\,d_\eu\psi=3\psi$, $\,d_\ew\psi=-4\psi\,$ would in turn give 
$\,6\psi=2\hh d_\eu\psi=\hh d_{[\eu,\ew]}\psi=\hh d_\eu d_\ew\psi
-\hh d_\ew d_\eu\psi=0$, contradicting our earlier conclusion that 
$\,\psi\ne0\,$ somewhere.
\end{example}
The following lemma will be needed in Section~\ref{cmax}. As usual, $\,\ki\,$ 
denotes the Kil\-ling operator.
\begin{lemma}\label{lilim}Let\/ $\,\nabla\hn\,$ be one of the connections\/ 
$\,\nabla(\ea,\eb)\,$ described in Example\/~{\rm\ref{nonab}}. If a 
left-in\-var\-i\-ant symmetric\/ $\,2$-ten\-sor on the underlying 
\hbox{two\hh-}\hskip0ptdi\-men\-sion\-al Lie group\/ $\,\hp\,$ is the\/ 
$\,\ki$-im\-age of some\/ $\,1$-form, then it is also the\/ $\,\ki$-im\-age of 
some left-in\-var\-i\-ant\/ $\,1$-form.
\end{lemma}
\begin{proof}With $\,\lie_j$ denoting the Lie derivatives with respect to the 
right-in\-var\-i\-ant vector fields $\,\ev_j$, $\,j=1,2$, chosen as in 
(\ref{vot}), the left-in\-var\-i\-ance of the symmetric $\,2$-ten\-sor 
$\,\ki\hh\beta$, for a given $\,1$-form $\,\beta$, means that 
$\,\lie_j\ki\hh\beta=0$, since the flows of right-in\-var\-i\-ant vector 
fields on $\,\hp\,$ consist of left translations. The last fact also implies 
that both $\,\lie_j$ commute with $\,\ki$. Left-in\-var\-i\-ance of 
$\,\ki\hh\beta\,$ thus gives $\,\lie_j\beta\in\mathrm{Ker}\,\ki$. If 
$\,(\ea,\eb)\ne(-9,0)$, then, according to (\ref{gkz}), (\ref{nab}) and 
Example~\ref{kerkz}, $\,\mathrm{Ker}\,\ki=\{0\}$, so that $\,\beta\,$ is 
left-in\-var\-i\-ant, which yields our assertion. On the other hand, if 
$\,(\ea,\eb)=(-9,0)$, we may use the basis $\,\alpha,\phi\,$ of 
left-in\-var\-i\-ant $\,1$-forms given by $\,\alpha(\hn\eu)=3$, 
$\,\alpha(\ew)=2\,$ (for $\,\eu,\ew\,$ as in Example~\ref{nonab}) and 
(\ref{rec}.i), and fix a function $\,\ef\,$ with (\ref{duf}). The $\,1$-form 
$\,\xi=\ef^2\alpha\,$ then spans $\,\mathrm{Ker}\,\ki\,$ (see 
Example~\ref{dkleo}), while $\,\beta=\mu\hs\alpha+\chi\hs\phi\,$ for some 
functions $\,\chi\,$ and $\,\mu$. Since $\,\lie_j\beta\in\mathrm{Ker}\,\ki$, 
there exist $\,c_j\in\bbR\,$ such that 
$\,\lie_j\beta=4\hs c_j\xi=4\hs c_j\ef^2\alpha$. At the same time, 
$\,\lie_j\beta=(\lie_j\mu)\hs\alpha+(\lie_j\chi)\hs\phi$. The resulting 
system of equations, rewritten with the aid of (\ref{vot}), states that 
$\,\chi\,$ is constant and $\,d_\eu\mu=4\hs c_1^{\phantom i}\ef^3\nnh$, 
$\,d_\ew\mu=4\hs c_1^{\phantom i}\ef^2\psi-4\hs c_2^{\phantom i}\ef^2\nnh$. 
From (\ref{duf}) and (\ref{ruw}.b) with $\,3\hs d\psi=-\ef\phi\,$ we obtain 
$\,0=d_\eu d_\ew\mu-d_\ew d_\eu\mu-2\hs d_\eu\mu
=24\hs c_1^{\phantom i}\ef^3\nnh$, that is, $\,c_1^{\phantom i}=0$. The 
equations imposed on $\,\mu$, combined with (\ref{duf}), now give 
$\,\mu=c_2^{\phantom i}\ef^2\nnh+c\hh'$ for some $\,c\hh'\nh\in\bbR\hh$. Thus, 
$\,\beta=\mu\hs\alpha+\chi\hs\phi
=c_2^{\phantom i}\xi+c\hh'\nnh\alpha+\chi\hs\phi$, and $\,\beta\,$ has the 
same $\,\ki$-im\-age as the left-in\-var\-i\-ant $\,1$-form 
$\,c\hh'\nnh\alpha+\chi\hs\phi$.
\end{proof}

\section{Degree of mobility for type~III SDNE Walker manifolds}\label{demo}
\setcounter{equation}{0}
Let $\,(\ym,\gm)\,$ be a type~III SDNE Walker manifold. Since our discussion 
is local, we may use Theorem~\ref{maith} to identify $\,\ym\,$ with $\,\tab\,$ 
for some surface $\,\bs\,$ with a tor\-sion\-free connection $\,\nabla\hn\,$ 
such that the Ric\-ci tensor $\,\rho\,$ of $\,\nabla\hn\,$ which is 
skew-sym\-met\-ric and nonzero at each point. At any $\,x\in\ym=\tab$, the 
bundle projection $\,\pi:\tab\to\bs\,$ induces a Lie-al\-ge\-bra 
homomorphism
\begin{equation}\label{hom}
\pi_x:\mathfrak{i}\hh_x\to\hs\mathfrak{a}_y\hh,\hskip10pt\mathrm{where}
\hskip5pty=\pi(x)\hh,
\end{equation}
$\,\mathfrak{a}_y$ and $\,\mathfrak{i}\hh_x$ being as in Section~\ref{prel} 
for $\,\nabla\hn\,$ (on $\,\bs$) and $\,\gm\,$ (on $\,\ym$). Thus, if 
$\,\delta=\dim\hskip2pt\mathrm{Ker}\,\pi_x$, and $\,(\mathrm{Ker}\,\ki)_y$ is 
the space of germs at $\,y\,$ of $\,1$-forms $\,\xi\,$ with $\,\ki\hh\xi=0\,$ 
on a neighborhood of $\,y\,$ in $\,\bs$,
\begin{equation}\label{did}
\mathrm{a)}\hskip6pt\dim\hs\mathfrak{i}\hh_x\,=\,\hs\delta\,
+\,\hs\mathrm{rank}\hskip2pt\pi_x\,
\le\,\delta+\dim\hs\mathfrak{a}_y\hh,\hskip22pt\mathrm{b)}\hskip6pt
\delta\,=\,\dim\hskip2pt(\mathrm{Ker}\,\ki)_y\,\le\,1\hh,
\end{equation}
relation (\ref{did}.b) being immediate from Lemma~\ref{vrtki} and 
Theorem~\ref{kileq}(i), as $\,\mathrm{Ker}\,\pi_x$ consists of germs, at 
$\,x$, of those Kil\-ling fields for $\,g\,$ which are vertical (tangent to 
$\,\mathcal{V}=\kerd\pi$). Thus,
\begin{equation}\label{thr}
\mathrm{if\ \ }\dim\hskip2pt\mathfrak{i}\hh_x\ge\hh3,\mathrm{\ \ the\ pair\ \ 
}(\dim\hs\mathfrak{a}_y,\delta\hh)\mathrm{\ \ is\ one\ of\ \ }(3,0)\,
\mathrm{\ and\ }\,(2,1)\hh.
\end{equation}
Namely, according to Remark~\ref{uslin}(a), we just need to show that 
$\,(\dim\hs\mathfrak{a}_y,\delta\hh)\,$ cannot equal $\,(3,1)$. If it did, 
however, we would be free to assume that $\,\ew_y\nh\ne0\,$ (replacing $\,y\,$ 
with a point arbitrarily close to it, cf.\ (\ref{dns})). Then $\,\nabla\hn\,$ 
would be locally homogeneous at $\,y\,$ as a consequence of (\ref{isa}), and 
the final clause of Theorem~\ref{modul} would imply that $\,\nabla\hn\,$ is 
locally equivalent to the connection $\,\nabla(1,0)\,$ of Example~\ref{nonab}, 
contradicting in turn relations (\ref{nab}) and (\ref{gkz}) (as 
$\,\delta=\dim\hskip2pt(\mathrm{Ker}\,\ki)_y\,=\,1$).
\begin{theorem}\label{maxmo}Let\/ $\,(\ym,\gm)\,$ be a neu\-tral-sig\-na\-ture 
oriented Ric\-ci-flat self-du\-al Walker 
\hbox{four\hskip.4pt-}\hskip0ptman\-i\-fold of Pe\-trov type\/ {\rm III}. 
Then its degree of mobility\/ $\,\dim\hskip2pt\mathfrak{i}\hh_x$, defined as 
in Section\/~\ref{prel}, does not exceed\/ $\,3\,$ at any point\/ $\,x\in\ym$. 
In particular, $\,(\ym,\gm)\,$ cannot be locally homogeneous.
\end{theorem}
In fact, if $\,\dim\hskip2pt\mathfrak{i}\hh_x\ge\hh3$, (\ref{thr}) 
gives $\,\delta+\dim\hs\mathfrak{a}_y=3$, and so, by (\ref{did}.a), 
$\,\dim\hs\mathfrak{i}\hh_x=3$.

\section{The image of the Kil\-ling operator}\label{ioko}
\setcounter{equation}{0}
Let a tor\-sion\-free connection $\,\nabla\hn\,$ with skew-sym\-met\-ric 
Ric\-ci tensor $\,\rho\,$ on a surface $\,\bs\,$ be generic, as defined in the 
lines preceding formula (\ref{gkz}). We use the symbol $\,\op\,$ for the 
fourth-or\-der linear differential operator sending any symmetric 
$\,2$-ten\-sor $\,\tau\,$ on $\,\bs\,$ to the symmetric $\,2$-ten\-sor 
\begin{equation}\label{pte}
\op\tau\,\hs=\hs\,\tau\,\hs-\hs\,\ki[(\nd^*)^{-1}\oz\tau]\hs.
\end{equation}
Here $\,\ki,\oz\,$ and $\,\nd\,$ are given by (\ref{cod}.b), (\ref{zte}) and 
(\ref{qef}), while $\,(\nd^*)^{-1}\nnh:\tb\to\tb\,$ denotes the inverse 
of the dual of $\,\nd$, defined as in (\ref{ast}), so that 
$\,(\nd^*)^{-1}\oz\tau$ is the $\,1$-form $\,\xi\,$ with 
$\,\nd^*\xi=\oz\tau$, in the notation of (\ref{xqe}.a). Finally, we let 
$\,\as\nh_1$ (or, $\,\as\nh_2$) stand for the space of all $\,1$-forms (or, 
symmetric $\,2$-ten\-sors) of class $\,C^\infty$ on $\,\bs$.
\begin{theorem}\label{imker}Suppose that\/ $\,\nabla\hn\,$ is a generic 
tor\-sion\-free connection with skew-sym\-met\-ric Ric\-ci tensor on a 
surface\/ $\,\bs$, while\/ $\,\ki,\oz\,$ and\/ $\,\op\,$ are given by\/ 
{\rm(\ref{cod}.b)}, {\rm(\ref{zte})} and\/ {\rm(\ref{pte})}.
\begin{enumerate}
  \def\theenumi{{\rm\alph{enumi}}}
\item[(a)] $\ki(\as\nh_1)=\mathrm{Ker}\,\op$, that is, the image of the 
Kil\-ling operator\/ $\,\ki:\as\nh_1\to\as\nh_2$ is at the same time the 
kernel of the fourth-or\-der operator\/ $\,\op:\as\nh_2\to\as\nh_2$.
\item[(b)] We have a di\-rect-sum decomposition\/ 
$\,\as\nh_2=\ki(\as\nh_1)\oplus\mathrm{Ker}\,(\op\nh-\mathrm{Id})$, for which 
$\,\op:\as\nh_2\to\as\nh_2$ serves as the projection onto the second summand.
\item[(c)] The image of\/ $\,\op:\as\nh_2\to\as\nh_2$ coincides both with\/ 
$\,\mathrm{Ker}\,(\op\nh-\mathrm{Id})\,$ and with the space of all\/ 
$\,C^\infty$ solutions\/ $\,\tau\in\as\nh_2$ to the third-or\-der linear 
differential equation\/ $\,\oz\tau=0$.
\end{enumerate}
\end{theorem}
\begin{proof}If $\,\tau=\ki\hh\xi$, (\ref{xqe}.a) gives $\,\op\tau=0$, while, 
if $\,\op\tau=0$, then $\,\tau\in\ki(\as\nh_1)\,$ by (\ref{pte}), which yields 
(a). Next, (a) and (\ref{pte}) imply that $\,\op^2\nnh=\op\nnh$. Thus, 
$\,\op\,$ is a projection onto its image, and (b) follows from (a). 
Finally, by (b), $\,\mathrm{Ker}\,(\op\nh-\mathrm{Id})\,$ is the image of 
$\,\op\nnh$, while $\,\tau\in\as\nh_2$ lies in 
$\,\mathrm{Ker}\,(\op\nh-\mathrm{Id})\,$ if and only if 
$\,\tau=\op\tau=\tau-\ki[(\nd^*)^{-1}\oz\tau\hh]$, which amounts to 
$\,\ki[(\nd^*)^{-1}\oz\tau\hh]=0$. Since $\,\mathrm{Ker}\,\ki=\{0\}\,$ (see 
(\ref{gkz})), this is equivalent to $\,\oz\tau=0$, as required in (c).
\end{proof}

\section{Type III~SDNE generic Walker metrics}\label{ttwg}
\setcounter{equation}{0}
Generic RSTS connections were defined in the lines preceding formula 
(\ref{gkz}). We will now refer to a type~III SDNE Walker manifold 
$\,(\ym,\gm)\,$ as {\it generic\/} if so is, at every point $\,x\in\ym$, the 
RSTS connection $\,\nabla\hn\,$ associated, in the sense of 
Theorem~\ref{maith}, with the restriction of $\,\gm\,$ to a neighborhood of 
$\,x$. The connection $\,\nabla\hn\,$ is a part of the triple 
$\,(\bs,\nabla\nnh,[\tau])\,$ of local invariants of $\,\gm$, introduced in 
Remark~\ref{trinv}. In the generic case, however, the coset $\,[\tau]\,$ 
contains a distinguished element $\,\sigma\,$ given by $\,\sigma=\op\tau$, 
with the operator $\,\op\,$ determined by $\,\nabla\hn\,$ via (\ref{pte}). 
That $\,\sigma\,$ is independent of the choice of $\,\tau\,$ in the coset is 
immediate from Theorem~\ref{imker}(a). In view of Theorem~\ref{imker}(c), 
$\,\oz\sigma=0$.

Thus, by Theorem~\ref{imker}(b), if a type~III SDNE Walker metric 
$\,\gm\,$ is generic, the invariant $\,\sigma\,$ described above constitutes 
a canonical choice of the $\,2$-ten\-sor $\,\tau\,$ appearing in 
Theorem~\ref{maith}. The two final paragraphs of Remark~\ref{trinv} then 
remain valid also after $\,[\tau]\,$ has been replaced by a $\,2$-ten\-sor 
$\,\sigma\,$ on $\,\bs$, subject only to the condition $\,\oz\sigma=0\,$ 
(cf.\ Theorem~\ref{imker}(c)).

\section{The case of maximum mobility}\label{cmax}
\setcounter{equation}{0}
According to Theorem~\ref{maxmo}, the degree of mobility of a type~III SDNE 
Walker manifold cannot exceed $\,3\,$ at any point. This section provides a 
complete characterization of the case where it is $\,3$. We begin by 
constructing, in Examples~\ref{dhthr} and~\ref{minnn}, a single manifold and, 
respectively, a \hbox{one\hh-}\hskip0ptpa\-ram\-e\-ter family of type~III SDNE 
Walker manifolds having the degree of mobility equal to $\,3$. These 
mutually non-i\-so\-met\-ric manifolds are shown, in 
Theorem~\ref{three}, to represent all local isometry classes of type~III SDNE 
Walker metrics with maximum mobility. 
\begin{example}\label{dhthr}Let $\,\nabla\hn\,$ be one of the connections, 
described in Example~\ref{slinv}, on a 
\hbox{two\hh-}\hskip0ptdi\-men\-sion\-al real vector space $\,\plane$. (They 
are all dif\-feo\-mor\-phi\-cal\-ly equivalent.) Using the first part of 
Theorem~\ref{maith}, with $\,\tau=0$, we see that $\,\gm=\gm\nnh^\nabla$ then 
is a type~III SDNE Walker metric on $\,\ym=T\hskip.2pt^*\hskip-.9pt\plane$. 
Invariance of $\,\gm\,$ under the cotangent action of 
$\,\mathrm{SL}\hs(\plane)\,$ implies, in view of Remark~\ref{uslin}(a), that 
the degree of mobility of $\,\gm\,$ equals $\,3\,$ at every point.
\end{example}
\begin{example}\label{minnn}Given the connection $\,\nabla=\nabla(-9,0)\,$ 
defined as in Example~\ref{nonab}, with $\,(\ea,\eb)=(-9,0)$, on a 
\hbox{two\hh-}\hskip0ptdi\-men\-sion\-al \hbox{non\hs-\nh}\hskip0ptAbel\-i\-an 
simply connected Lie group $\,\hp$, and any left-in\-var\-i\-ant symmetric 
$\,2$-ten\-sor $\,\tau\,$ on $\,\hp$, the Riemann extension 
$\,\gm\,=\,\gm\nnh^\nabla\nnh+\hs2\hh\pi^*\nnh\tau\,$ is, according to the 
first part of Theorem~\ref{maith}, a type~III SDNE Walker metric on 
$\,\ym=T\hskip.2pt^*\hskip-.9pt\hp$, invariant under the cotangent left action 
of $\,\hp$, the orbits of which are obviously transverse to the fibres of 
$\,\ym=T\hskip.2pt^*\hskip-.9pt\hp$. In view of Example~\ref{dkleo} and 
Lemma~\ref{vrtki}, $\,(\ym,\gm)\,$ also admits a Kil\-ling vector field 
tangent to the fibres, which is nonzero everywhere (Theorem~\ref{kileq}(ii)), 
and, so by Theorem~\ref{maxmo}, its degree of mobility is $\,3\,$ at every 
point.

Even though the tensors $\,\tau\,$ used here form a 
\hbox{three\hh-}\hskip0ptdi\-men\-sion\-al space, the construction gives rise 
only to \hbox{one\hh-}\hskip0ptpa\-ram\-e\-ter family of nonequivalent 
metrics. In fact, the Kil\-ling operator $\,\ki\,$ restricted to the space of 
left-in\-var\-i\-ant $\,1$-forms $\,\alpha\,$ is injective 
(Example~\ref{dkleo}), while the metrics 
$\,\gm\,=\,\gm\nnh^\nabla\nnh+\hs2\hh\pi^*\nnh\tau\,$ 
and $\,\gm\hh'\hh=\,\gm\nnh^\nabla\nnh+\hs2\hh\pi^*\nnh\tau\hh'$ corresponding 
to the tensors $\,\tau\,$ and $\,\tau\hh'\nh=\tau+\ki\hh\alpha\,$ are 
isometric to each other in view of Lemma~\ref{kxpgn}(b).
\end{example}

\begin{theorem}\label{three}If\/ $\,(\ym,\gm)\,$ is a neu\-tral-sig\-na\-ture 
oriented Ric\-ci-flat self-du\-al Walker 
\hbox{four\hskip.4pt-}\hskip0ptman\-i\-fold of Pe\-trov type\/ {\rm III}, and 
the degree of mobility of\/ $\,\gm\,$ at a point\/ $\,x\,$ equals\/ $\,3$, 
then\/ $\,x\,$ has a connected neighborhood isometric to an open submanifold 
of one of the manifolds described in Examples\/~{\rm\ref{dhthr}} 
and\/~{\rm\ref{minnn}}.
\end{theorem}
\begin{proof}We use the same assumptions and identifications as in the lines 
preceding (\ref{hom}), so that $\,\ym=\tab\,$ for a surface $\,\bs\,$ carrying 
a tor\-sion\-free connection $\,\nabla\hn\,$ with 
eve\-ry\-\hbox{where\hskip1pt-}\hskip0ptnon\-zero, skew-sym\-met\-ric Ric\-ci 
tensor $\,\rho$. We also fix a point $\,x\in\ym\,$ at which 
$\,\dim\hs\mathfrak{i}\hh_x=3$, and set $\,y=\pi(x)$, where 
$\,\pi:\tab\to\bs\,$ is the bundle projection. By (\ref{thr}), the pair 
$\,(\dim\hs\mathfrak{a}_y,\delta\hh)\,$ equals $\,(3,0)\,$ or $\,(2,1)$.

If $\,(\dim\hs\mathfrak{a}_y,\delta\hh)=(3,0)$, then the conclusion of 
Remark~\ref{uslin}(c) holds. In fact, when $\,\ew_y\nh=0$, this is explicitly 
stated in Remark~\ref{uslin}(c), while, in the case $\,\ew_y\nh\ne0$, 
Lemma~\ref{dathr} implies local homogeneity of $\,\nabla\hn\,$ at $\,y$, and 
our claim follows from the final clause of Theorem~\ref{modul}. Thus, 
$\,\nabla\hn\,$ is generic, as one sees using (\ref{nab}) with 
$\,(\ea,\eb)=(1,0)\,$ if $\,\ew_y\nh\ne0$, and noting that the same is true 
when $\,\ew_y\nh=0\,$ since points $\,y\,$ with $\,\ew_y\nh\ne0\,$ form a 
dense set (see (\ref{dns})), and so, by (\ref{mat}), 
$\,\mathrm{det}\hskip1.9pt\nd\,$ is constant, namely, equal to $\,25$. The 
invariant $\,\sigma\,$ introduced in Section~\ref{ttwg} may thus be treated, 
locally, as a symmetric $\,2$-ten\-sor on an open connected subset of 
$\,\plane$, invariant under the infinitesimal action of 
$\,\mathrm{SL}\hs(\plane)\,$ (notation of Example~\ref{slinv}). Hence, by 
Lemma~\ref{dathr}, $\,\sigma\,$ is a constant multiple of $\,\phi\otimes\phi$. 
However, as $\,\oz\sigma=0$, cf.\ Section~\ref{ttwg}, (\ref{zff}) now gives 
$\,\sigma=0$, and our claim follows in this case since $\,\gm\,$ and 
$\,\gm\nnh^\nabla\nnh+\hs2\hh\pi^*\nnh\sigma\,$ are locally isometric 
(Section~\ref{ttwg}).

Now let $\,(\dim\hs\mathfrak{a}_y,\delta\hh)=(2,1)$. By (\ref{did}.b), this is 
also the case if $\,y\,$ is replaced with any point of some connected 
neighborhood $\,\,U\,$ of $\,y\,$ in $\,\bs$. We fix 
$\,\xi\in\mathrm{Ker}\,\ki$, defined on $\,\,U\nh$, so that $\,\xi\ne0\,$ 
everywhere in $\,\,U\,$ (cf.\ (\ref{did}.b) and Theorem~\ref{kileq}(ii)). On 
any open subset of $\,\,U\,$ on which $\,\xi(\ew)=0$, (\ref{xqe}.b) and 
(\ref{nnx}) with $\,\tau=0\,$ give $\,\nabla\xi=0\,$ and 
$\,\xi\otimes\rho=-\nabla\nabla\xi=0$. Hence that such a subset must be empty, 
and so $\,\xi(\ew)\ne0\,$ at all points of a dense open subset $\,\,U'$ of 
$\,\,U\nh$. Applying (\ref{let}) to $\,\ev\,$ given by 
$\,\rho\hs(\hn\ev,\,\cdot\,)=\xi\,$ we see that $\,\nabla\hn\,$ restricted to 
$\,\,U'$ is locally homogeneous. As 
$\,\dim\hskip2pt(\mathrm{Ker}\,\ki)_y=\delta=1$, (\ref{gkz}), (\ref{nab}) and 
Example~\ref{kerkz} show that $\,\nabla\hn\,$ represents, on $\,\,U'\nnh$, the 
point $\,(\ea,\eb)=(-9,0)\,$ of the moduli curve in Theorem~\ref{modul}(i). 
Therefore, on $\,\,U'\nnh$, (\ref{mat}) with $\,(\ea,\eb)=(-9,0)\,$ yields 
$\,\nd\hh\ew=15\hh\ew-10\hh\eu$, and, by (\ref{ruw}.a), 
$\,\rho\hs(\ew,\nd\hh\ew)=60$, which, due to denseness of $\,\,U'$ in $\,\,U\nh$, 
holds on $\,\,U\,$ as well. Consequently, $\,\ew\,$ and $\,\nd\hh\ew\,$ are 
linearly independent at each point of $\,\,U\nh$, so that the same is true 
of $\,\ew\,$ and $\,\eu$, where $\,\eu\,$ has now been extended to $\,\,U\,$ 
via the formula $\,10\hh\eu=15\hh\ew-\nd\hh\ew$. By continuity, we have 
(\ref{nuu}) everywhere in $\,\,U\nh$, and hence $\,[\eu,\ew]=2\eu$. This 
allows us to treat $\,\,U\nh$, locally, as an open set in a 
\hbox{two\hh-}\hskip0ptdi\-men\-sion\-al \hbox{non\hs-\nh}\hskip0ptAbel\-i\-an 
simply connected Lie group $\,\hp$, while $\,\eu,\ew\,$ then become 
left-in\-var\-i\-ant vector fields on $\,\hp$, and elements of 
$\,\mathfrak{a}_y$ are the germs at $\,y\,$ of right-in\-var\-i\-ant vector 
fields on $\,\hp\,$ (the flows of which consist of left translations).

Making $\,\,U\,$ smaller, if necessary, and using Theorem~\ref{maith}, we may 
assume that $\,\gm\,=\,\gm\nnh^\nabla\nnh+\hs2\hh\pi^*\nnh\tau\,$ for some 
symmetric $\,2$-ten\-sor $\,\tau\,$ on $\,\,U\nh$. Since (\ref{hom}) is 
surjective (by (\ref{did}.b)), for every left translation $\,F\,$ close to the 
identity Lemma~\ref{ismtr}(ii) yields $\,F^*\tau=\tau+\ki\hh\beta$, where 
$\,\beta\,$ is a $\,1$-form depending on $\,F$. (Local isometries of 
$\,(\ym,\gm)\,$ leave the vertical distribution $\,\mathcal{V}\,$ invariant.) 
Infinitesimally, this gives $\,\lie_\ev\tau=\ki\hh\beta_\ev$ for every 
right-in\-var\-i\-ant vector field $\,\ev$, with a $\,1$-form $\,\beta_v$ that 
depends linearly on $\,\ev$. Let us now choose a basis 
$\,\ev_1^{\phantom i},\ev_2^{\phantom i}$ of right-in\-var\-i\-ant vector 
fields such that 
$\,[\ev_1^{\phantom i},\ev_2^{\phantom i}]=2\hh\ev_1^{\phantom i}$ and write 
$\,\lie_j,\beta_j$ instead of $\,\lie_\ev$ and $\,\beta_\ev$ for 
$\,\ev=\ev_j$. As $\,\lie_j$ commute with $\,\ki$, for the $\,1$-form 
$\,\xi\hh'\nh=\lie_1^{\phantom i}\beta_2^{\phantom i}\nh
-\lie_2^{\phantom i}\beta_1^{\phantom i}\nh-2\beta_1^{\phantom i}$ we have 
$\,\ki\hh\xi\hh'\nh=\lie_1^{\phantom i}\lie_2^{\phantom i}\tau
-\lie_2^{\phantom i}\lie_1^{\phantom i}\tau-2\lie_1^{\phantom i}\tau=0$. On 
the other hand, using the basis $\,\alpha,\phi\,$ of left-in\-var\-i\-ant 
$\,1$-forms given by $\,\alpha(\hn\eu)=3$, $\,\alpha(\ew)=2\,$ and 
(\ref{rec}.i), we get $\,\beta_j=\Delta(\hn\ev_j)\hs\alpha
+\varXi(\hn\ev_j)\hs\phi\,$ with some $\,1$-forms $\,\Delta\,$ and $\,\varXi$. 
The equality $\,\xi\hh'\nh=\lie_1^{\phantom i}\beta_2^{\phantom i}\nh
-\lie_2^{\phantom i}\beta_1^{\phantom i}\nh-2\beta_1^{\phantom i}$ now reads 
$\,\xi\hh'\nh=(d\Delta)_{12}^{\phantom i}\hs\alpha
+(d\varXi)_{12}^{\phantom i}\hs\phi\,$ (cf.\ (\ref{bwa}.b)), where, for any 
$\,2$-form $\,\zeta$, we use the subscript convention 
$\,\zeta\hh_{12}^{\phantom i}
=\zeta(\hn\ev_1^{\phantom i},\ev_2^{\phantom i})$. Also, as 
$\,\ki\hh\xi\hh'\nh=0$, if we fix a function $\,\ef\,$ with (\ref{duf}), 
$\,\xi\hh'$ equals a constant times $\,\xi=\ef^2\alpha\,$ (see 
Example~\ref{dkleo}). Thus, $\,(d\varXi)_{12}^{\phantom i}=0\,$ and 
$\,(d\Delta)_{12}^{\phantom i}=4\hs c\ef^3\nnh\rho_{12}^{\phantom i}$ for 
some $\,c\in\bbR\hh$. (By (\ref{riv}.i), $\,f\rho_{12}^{\phantom i}$ is 
constant.) Hence $\,d\varXi=0\,$ and 
$\,d\Delta=\hs 4\hs c\ef^3\nnh\rho=-\hs c\hskip1ptd(\ef^3\phi)$, cf.\ 
(\ref{riv}.ii), so that $\,\varXi=d\chi\,$ and $\,\Delta=d\mu-c\ef^3\phi\,$ 
for some functions $\,\chi\,$ and $\,\mu$. We may in addition assume that, for 
a suitable function $\,\psi$,
\begin{equation}\label{ljt}
\mathrm{i)}\hskip9pt\lie_j[\tau-\ki\hh(\mu\hs\alpha+\chi\hs\phi)]\,
=\,-\hs c\hs\ki\hh[\ef\phi(\ev_j)\hs\xi]\hh,\hskip12pt\mathrm{ii)}\hskip9pt
\ki\hh[\ef\phi(\ev_j)\hs\xi]\,=\,
-\hs\lie_j\ki\hh(\psi\hs\xi)\hs.
\end{equation}
Namely, (\ref{ljt}.i) follows in any case since $\,\lie_j\tau=\ki\hh\beta_j$ 
and $\,\ki\hh\beta_j\nh
=\ki\hh[\Delta(\hn\ev_j)\hs\alpha+\varXi(\hn\ev_j)\hs\phi]
=\ki\hh[(\lie_j\mu)\hs\alpha+(\lie_j\chi)\hs\phi]
-c\hh\ki\hh[\ef^3\phi(\ev_j)\hs\alpha]$, while $\,\xi=\ef^2\alpha$, and 
$\,\lie_j$ commute with $\,\ki$. To obtain (\ref{ljt}.ii), instead of letting 
the right-in\-var\-i\-ant fields $\,\ev_j$ with 
$\,[\ev_1^{\phantom i},\ev_2^{\phantom i}]=2\hh\ev_1^{\phantom i}$ be 
otherwise arbitrary, we choose them as in (\ref{vot}). Then 
$\,\ef\phi(\ev_1^{\phantom i})=-\hh6$, 
$\,\ef\phi(\ev_2^{\phantom i})=-\hh6\hh\psi\,$ (cf.\ (\ref{ruw}.b)), and so 
$\,\ki\hh[\ef\phi(\ev_1^{\phantom i})\hs\xi]=0$, as 
$\,\xi\in\mathrm{Ker}\,\ki$, while 
$\,\ki\hh[\ef\phi(\ev_2^{\phantom i})\hs\xi]=-\hh6\hh\ki\hh(\psi\hs\xi)$. 
However, $\,\lie_1^{\phantom i}(\ef^2\psi)=2\ef^2\,$ and 
$\,\lie_2^{\phantom i}(\ef^2\psi)=6\ef^2\psi\,$ by (\ref{vot}), (\ref{duf}) 
and (\ref{ruw}.b) with $\,3\hs d\psi=-\ef\phi$. Thus, 
$\,\lie_j(\psi\hs\xi)=\lie_j(\ef^2\psi\hs\alpha)
=[\lie_j(\ef^2\psi)]\hs\alpha\,$ and 
$\,\lie_j\ki\hh(\psi\hs\xi)=\ki\hh\lie_j(\psi\hs\xi)$, so that 
$\,\lie_1^{\phantom i}\ki\hh(\psi\hs\xi)=2\hh\ki\hh(\ef^2\alpha)
=2\hh\ki\hh\xi=0=-\hs\ki\hh[\ef\phi(\ev_1^{\phantom i})\hs\xi]$, and 
$\,\lie_2^{\phantom i}\ki\hh(\psi\hs\xi)=6\hh\ki\hh(\ef^2\psi\hs\alpha)
=6\hh\ki\hh(\psi\hs\xi)=-\hs\ki\hh[\ef\phi(\ev_2^{\phantom i})\hs\xi]$, as 
required.

Setting $\,\tau\hh'\nh=\tau-\ki\hh(\mu\hs\alpha+\chi\hs\phi+c\hs\psi\hs\xi)\,$ 
we see that, by (\ref{ljt}), $\,\lie_j\tau\hh'\nh=0$, and so $\,\tau\hh'$ is 
left-in\-var\-i\-ant. Therefore, in view of Lemma~\ref{kxpgn}(b), the 
restriction of $\,\gm\,$ to some neighborhood of $\,x\,$ is isometric to one 
of the metrics of Example~\ref{minnn}.
\end{proof}
The single manifold of Example~\ref{dhthr} and the 
\hbox{one\hh-}\hskip0ptpa\-ram\-e\-ter family of Example~\ref{minnn} together 
form a collection of type~III SDNE Walker manifolds that are mutually 
non-i\-so\-met\-ric, even locally. More precisely, an open sub\-man\-i\-fold 
in one of them is never isometric to an open sub\-man\-i\-fold of another.

In fact, the manifolds of Example~\ref{dhthr} differs from those in 
Example~\ref{minnn} by the value of the local invariant 
$\,(\dim\hs\mathfrak{a}_y,\delta\hh)\,$ (see the proof of 
Theorem~\ref{three}). Thus, we may restrict our discussion to the latter 
manifolds, assuming that $\,\tau\,$ and $\,\tau\hh'$ are left-in\-var\-i\-ant 
symmetric $\,2$-ten\-sors on $\,\hp$, while the restrictions of 
$\,\gm\nnh^\nabla\nnh+\hs2\hh\pi^*\nnh\tau\,$ and 
$\,\gm\nnh^\nabla\nnh+\hs2\hh\pi^*\nnh\tau\hh'$ to some open 
sub\-man\-i\-folds are isometric. Since the vertical distribution is a local 
invariant of the metric (Section~\ref{vert}), and so is the transversal 
connection $\,\nabla\hn\,$ (cf.\ Lemma~\ref{trvcn}(b) and 
Theorem~\ref{maith}), applying Lemma~\ref{ismtr}(ii) we conclude that the 
left-in\-var\-i\-ant symmetric $\,2$-ten\-sor $\,\tau-\tau\hh'$ is the 
$\,\ki$-im\-age of some $\,1$-form on $\,\hp$. We used here the fact that the 
left translations in $\,\hp\,$ are the only dif\-feo\-mor\-phisms $\,F\,$ 
between open sub\-man\-i\-folds of $\,\hp$, satisfying the condition 
$\,F^*\nabla\nh=\nabla\nnh$, which itself easily follows from (\ref{aff}) and 
Remark~\ref{uwinv}.

As a consequence of Lemma~\ref{lilim}, the two metrics represent the same 
element of the \hbox{one\hh-}\hskip0ptpa\-ram\-e\-ter family in 
Example~\ref{minnn}.

\section{Non\hh-ge\-ner\-ic RSTS connections and type III~SDNE Walker 
metrics}\label{ngrs}
\setcounter{equation}{0}
We refer to an RSTS connection $\,\nabla\hn\,$ as {\it special\/} if, at every 
point, the Ric\-ci tensor $\,\rho\,$ is nonzero and the bundle morphism 
$\,\nd\,$ given by (\ref{qef}.i) is noninjective. This is the extreme opposite 
of the case where $\,\nabla\hn\,$ is generic, defined in the lines preceding 
(\ref{gkz}). For an RSTS connection with $\,\rho\ne0\,$ at every point of the 
underlying surface $\,\bs$, being generic {\it or\/} special is a 
gen\-er\-al-po\-si\-tion requirement: $\,\bs\,$ obviously contains a dense 
open subset $\,\,U\,$ such that the restriction of $\,\nabla\hn\,$ to each 
connected component of $\,\,U\,$ is either generic or special.

According to (\ref{qef}.ii), if an RSTS connection $\,\nabla\hn\,$ is special, 
$\nd\,$ and $\,\nd^*$ have, at each point, the eigenvalues $\,0\,$ and 
$\,10$, so that
\begin{equation}\label{qmt}
(\nd-10)\hh\nd\hs=\hs0\hh,\hskip26pt(\nd^*\nh-10)\hh\nd^*=\hs0\hh.
\end{equation}
The next result may be viewed as a counterpart of Theorem~\ref{imker} for RSTS 
connections which, this time, are assumed special rather than generic. As 
before, $\,\as\nh_1$ (or, $\,\as\nh_2$) is the space of all $\,1$-forms (or, 
symmetric $\,2$-ten\-sors) of class $\,C^\infty$ on the surface $\,\bs\,$ in 
question. In view of (\ref{qmt}), $\,\tab
=\mathrm{Ker}\,(\nd^*\nh-10)\oplus\hs\mathrm{Ker}\,\nd^*\nnh$, so 
that $\,\as\nh_1=\sop\hskip-2.1pt\oplus\soz$, where $\,\sop$ (or, $\,\soz$) is 
the space of all $\,C^\infty$ sections of the line bundle 
$\,\mathrm{Ker}\,(\nd^*\nh-10)\,$ (or, $\,\mathrm{Ker}\,\nd^*\nh$). We 
also define a fourth-or\-der linear differential operator
$\,\ow:\as\nh_2\to\as\nh_2$ by $\,10\hs\ow=\hs\ki\oz$, with $\,\ki\,$ and 
$\,\oz\,$ as in (\ref{cod}.b) and (\ref{zte}).
\begin{theorem}\label{imspc}For any special RSTS connection,
\begin{enumerate}
  \def\theenumi{{\rm\alph{enumi}}}
\item[(a)] $\ow^{\hh3}\nh=\ow^{\hh2}\nnh$, that is, 
$\,(\ow\nh-\mathrm{Id})\hs\ow^{\hh2}\nh=0$, and\/ 
$\,\as\nh_2=\mathrm{Ker}\,\ow^{\hh2}\nh
\oplus\mathrm{Ker}\,(\ow\nh-\mathrm{Id})$,
\item[(b)] $\ki(\as\nh_1)=\ki(\sop)\oplus\ki(\soz)\,$ and\/ 
$\,\ki(\sop)=\mathrm{Ker}\,(\ow\nh-\mathrm{Id})$,
\item[(c)] $\ki(\soz)\hs\subset\hs\mathrm{Ker}\,\oz\hs
\subset\hs\mathrm{Ker}\,\ow\hs\subset\hs\mathrm{Ker}\,\ow^{\hh2}
=\hs\mathrm{Ker}\,\nd^*\hskip-2pt\oz$.
\end{enumerate}
\end{theorem}
\begin{proof}By (\ref{xqe}.a) and (\ref{qmt}), $\,\oz\ki=\nd^*$ and 
$\,(\nd^*\nnh)^2\nh=10\hs\nd^*\nnh$, which yields 
$\ow^{\hh3}\nh=\ow^{\hh2}\nnh$, and hence (a). (Explicitly, 
$\,\tau=\ow^{\hh2}\nh\tau-(\ow\nh+\mathrm{Id})(\ow\nh-\mathrm{Id})\hh\tau$, 
and $\,(\ow\nh+\mathrm{Id})(\ow\nh-\mathrm{Id})\tau
\in\mathrm{Ker}\,\ow^{\hh2}\nnh$, 
$\,\ow^{\hh2}\nh\tau\in\mathrm{Ker}\,(\ow\nh-\mathrm{Id})\,$ for any 
$\,\tau\in\as\nh_2$.) The relation $\,\ki(\sop)\cap\ki(\soz)=\{0\}$, that is, 
the first part of (b), is clear since $\,\oz\ki=\nd^*$ is zero on 
$\,\soz\nnh$, and injective on $\,\sop\nnh$. Next, 
$\,10\hs\ow\ki\alpha=\ki\oz\ki\alpha=\ki\nd^*\nnh\alpha
=10\hs\ki\hh\alpha\,$ whenever $\,\alpha\in\sop$, so that 
$\,\ki(\sop)\subset\mathrm{Ker}\,(\ow\nh-\mathrm{Id})$. On the other hand, if 
$\,\ow\tau=\tau$, setting $\,\alpha=\oz\tau\,$ we obtain 
$\,\nd^*\nnh\alpha=\oz\ki\alpha=\oz\ki\oz\tau=10\hs\oz\ow\tau=10\hs\oz\tau
=10\hs\alpha$, and so $\,\alpha\in\sop$, while 
$\,\ki\alpha=\ki\oz\tau=10\hs\ow\tau=10\hs\tau$. Consequently, 
$\,\mathrm{Ker}\,(\ow\nh-\mathrm{Id})\subset\ki(\sop)$, which proves (b). 
The first two inclusions in (c) are immediate as $\,\oz\ki=\nd^*$ 
and $\,10\hs\ow=\hs\ki\oz$, and the third one is obvious. Finally, 
$\,100\hs\ow^{\hh2}\nh=\ki\oz\ki\oz=\ki\nd^*\nnh\oz$. Thus, 
$\,\mathrm{Ker}\,\nd^*\hskip-2pt\oz\hs\subset\hs\mathrm{Ker}\,\ow^{\hh2}\nnh$. 
The opposite inclusion follows since, if $\,\ki\nd^*\nnh\oz\tau=0$, then 
$\,0=\oz\ki\nd^*\nnh\oz\tau=(\nd^*\nnh)^2\nh\oz\tau=10\hs\nd^*\nnh\oz\tau$.
\end{proof}
By analogy with Section~\ref{ttwg}, we will also say that a type~III SDNE 
Walker manifold $\,(\ym,\gm)\,$ is {\it special\/} if so is, at every point 
$\,x\,$ of $\,\ym$, the RSTS connection $\,\nabla\hn\,$ determined, as in 
Theorem~\ref{maith}, by the restriction of $\,\gm\,$ to a neighborhood of 
$\,x$. 

The assumption that $\,(\ym,\gm)\,$ is special allows us, as in 
Section~\ref{ttwg}, to replace $\,(\bs,\nabla\nnh,[\tau])\,$ by a more 
tangible triple $\,(\bs,\nabla\nnh,\sigma)\,$ of local invariants. Here 
$\,\sigma\,$ is the $\,\mathrm{Ker}\,\ow^{\hh2}$ component of $\,\tau\,$ 
relative to the decomposition $\,\as\nh_2=\mathrm{Ker}\,\ow^{\hh2}\nh
\oplus\mathrm{Ker}\,(\ow\nh-\mathrm{Id})\,$ in Theorem~\ref{imspc}(a), so 
that, due to the second equality in Theorem~\ref{imspc}(b), the coset 
$\,[\sigma]=[\tau]\,$ remains unchanged, and 
$\,\,\nd^*\hskip-2pt\oz\hh\sigma=0$, cf.\ Theorem~\ref{imspc}(c).

In contrast with the situation discussed in Section~\ref{ttwg}, $\,\sigma\,$ 
does not constitute a unique, canonical choice of a representative from the 
coset $\,[\tau]$. In fact, according to Theorem~\ref{imspc}(c), by replacing 
$\,\tau\,$ with $\,\sigma\,$ we have merely reduced the freedom of choosing 
$\,\tau$, which originally ranged over a coset of the subspace 
$\,\ki(\as\nh_1)\,$ in the space $\,\as\nh_2$, to the freedom of selecting 
$\,\sigma\,$ out of a fixed coset of the subspace $\,\ki(\soz)\,$ in the space 
$\,\mathrm{Ker}\,\nd^*\hskip-2pt\oz$.

The remainder of this section is devoted to a description of the local 
structure of an arbitrary RSTS connection with 
$\,\dim\hskip2pt\mathrm{Ker}\,\ki=1$, where $\,\ki\,$ is the Kil\-ling 
operator. By (\ref{xqe}.a), all such connections are special. We begin with 
some examples.

Let $\,\bs\,$ be a surface with fixed vector fields $\,\ev,\ew\,$ that 
trivialize the tangent bundle $\,\tb\,$ and functions 
$\,\xf,\xh:\bs\to\bbR\hh$, satisfying the conditions 
\begin{equation}\label{bvw}
[\ev,\ew]\hs=\hs6\ev+2\xf\ew\hh,\hskip10ptd_\ev\xf=4\xf^2,\hskip10pt
d_\ew\xf=\hs-\hs4\xf\hh,\hskip10ptd_\ev\xh
=\hs4\hs,\hskip10pt\xf\ne0\hskip6pt\mathrm{everywhere.} 
\end{equation}
Such $\,\ev,\ew,\xf,\xh\,$ are in a bijective correspondence with quadruples 
$\,\eu,\ew,\xf,\xh\,$ in which $\,\eu,\ew\,$ again trivialize the tangent 
bundle, while $\,[\eu,\ew]=6\eu$, $\,d_\eu\xf=0$, $\,d_\ew\xf=\hs-\hs4\xf$, 
$\,d_{\eu-\xf\ew}\xh=\hs4\,$ and $\,\xf\ne0\,$ everywhere. (The correspondence 
is given by $\,\eu=\ev+\xf\ew$.) Locally, the triples $\,\eu,\ew,\xf\,$ within 
the latter quadruples arise precisely when $\,\eu,\ew\,$ constitute a suitably 
chosen basis of left-in\-var\-i\-ant vector fields on a 
\hbox{two\hh-}\hskip0ptdi\-men\-sion\-al 
\hbox{non\hs-\nh}\hskip0ptAbel\-i\-an Lie group $\,\hp\,$ and $\,\xf\,$ is a 
nonzero constant multiple of a specific Lie-group homomorphism from $\,\hp\,$ 
into the multiplicative group $\,(0,\infty)$. (Cf.\ Remark~\ref{loclg} and 
(\ref{duf}.)

For $\,\bs\,$ and $\,\ev,\ew,\xf,\xh\,$ as above, we define a tor\-sion\-free 
connection $\,\nabla\hn\,$ on $\,\bs\,$ by 
\begin{equation}\label{nvv}
\nabla_{\hskip-2.2pt\ev}\ev\,=\,5\xf\hh\ev\hh,\hskip10pt
\nabla_{\hskip-2.2pt\ew}\ev\,=\,\xf\hh\ew\hh,\hskip10pt
\nabla_{\hskip-2.2pt\ev}\ew\,=\,6\hh\ev+3\xf\hh\ew\hh,\hskip10pt\nsww\,
=\,15\hh\xh\hh\ev-4\hh\ew\hh.
\end{equation}
One easily verifies that the Ric\-ci tensor $\,\rho\,$ of $\,\nabla\hn\,$ is 
skew-sym\-met\-ric, $\,\rho\hs(\hn\ev,\ew)=4\xf$, and $\,\ew\,$ coincides with 
the vector field characterized by (\ref{rec}.ii). Furthermore, 
$\,\dim\hskip2pt\mathrm{Ker}\,\ki=1\,$ due to Theorem~\ref{kileq}(i) and the 
fact that $\,\ki\hh\xi=0\,$ for the nonzero $\,1$-form $\,\xi\,$ with 
$\,\xi(\hn\ev)=0\,$ and $\,\xi(\ew)=4\xf\nnh$. In addition, 
$\,60\xf\xh=\rho\hs(\nsww,\ew)\,$ is an af\-fine invariant and 
$\,d_\ev(\xf\xh)=4\hh(\xf\xh+1)\xf\nnh$. Thus, $\,\nabla\hn\,$ is 
locally homogeneous if and only if $\,\xf\xh\,$ is constant, namely, equal 
to $\,-1$. (The `if' part is immediate from Example~\ref{nonab}, since, 
setting $\,\eu=3\hh(\ew-\xh\hh\ev)/2$, we may then rewrite (\ref{nvv}) as 
(\ref{nuu}) with $\,(\ea,\eb)=(-9,0)$.)
\begin{theorem}\label{dkero}Every RSTS connection with\/ 
$\,\dim\hskip2pt\mathrm{Ker}\,\ki=1\,$ is, locally, at points in general 
position, given by\/ {\rm(\ref{nvv})} for some quadruple\/ 
$\,\ev,\ew,\xf,\xh\,$ as above, with\/ {\rm(\ref{bvw})}.
\end{theorem}
\begin{proof}Let us define $\,\ew\,$ by (\ref{rec}.ii) and $\,\xf\,$ by 
$\,4\xf\nnh=\xi(\ew)$, where $\,\xi\,$ is a fixed nontrivial $\,1$-form with 
$\,\ki\hh\xi=0$. Thus, $\,\xf\ne0\,$ at all points of a dense open subset: in 
fact, if we had $\,\xf=0\,$ on some nonempty open set $\,\,U\nh$, 
(\ref{xqe}.b) with $\,\tau=0\,$ would give $\,\nabla\xi=0\,$ on $\,\,U\nh$, 
and so (\ref{nnx}) with $\,\tau=0\,$ would imply that $\,\xi=0\,$ on 
$\,\,U\nh$, contradicting Theorem~\ref{kileq}(ii).

From now on we assume that $\,\xf\ne0\,$ everywhere. Using (\ref{xqe}.b) with 
$\,\tau=0\,$ we obtain $\,\nabla\xi=\xf\rho$, and so 
$\,\nabla\nabla\xi=(d\xf+\xf\hh\phi)\otimes\rho\,$ (cf.\ (\ref{rec}.i)), while 
(\ref{nnx}) with $\,\tau=0\,$ yields $\,\nabla\nabla\xi=-\hs\xi\otimes\rho$. 
Hence $\,d\xf=-\hs\xi-\xf\hh\phi$.

For the vector field  $\,\ev\,$ characterized by 
$\,\xi=\rho\hs(\hn\ev,\,\cdot\,)$, we see that $\,\xi(\hn\ev)=0\,$ and 
$\,\phi(\hn\ev)=\rho\hs(\ew,\ev)=-\hs\xi(\ew)=-\hs4\hh\xf$. Since 
$\,d\xf=-\hs\xi-\xf\hh\phi$, (\ref{rec}.iii) now gives 
$\,d_\ev\xf=4\xf^2\nnh$, $\,d_\ew\xf=\hs-\hs4\xf$, as required in (\ref{bvw}). 
From the relation $\,\nabla\xi=\xf\rho\,$ and the Leib\-niz rule, 
$\,\rho\hs(\nsu\ev,\,\cdot\,)=\nsu\xi-\phi(\hn\eu)\hs\xi
=\xf\rho\hs(\hn\eu,\,\cdot\,)-\phi(\hn\eu)\rho\hs(\hn\ev,\,\cdot\,)\,$ 
for any vector field $\,\eu$, so that 
$\,\nsu\ev=\xf\hh\eu-\phi(\hn\eu)\hh\ev$. In particular, setting $\,\eu=\ev\,$ 
or $\,\eu=\ew$, we obtain the first two equalities in (\ref{nvv}).

As $\,\ki\hh\xi=0$, (\ref{xqe}.a) shows that $\,\nd^*\xi=0$. Thus, by 
(\ref{qef}.i), $\,0=\xi(\nd\ew)=\xi(4\ew+\nsww)=16\hh\xf+\xi(\nsww)\,$ and 
$\,0=\xi(\nd\ev)=\xi(4\ev+\nsvw-3\hh\xf\ew)=\xi(\nsvw)-12\hh\xf^2\nnh$. We 
used here the fact that $\,\xi(\hn\ev)=0$, which now also implies the last 
equality in (\ref{nvv}), for some function $\,\xh$, as well as the relation 
$\,\nabla_{\hskip-2.2pt\ev}\ew=\mu\hh\ev+3\xf\hh\ew\,$ for some function 
$\,\mu$. At the same time, by (\ref{rec}.iv) and (\ref{bwa}.b), 
$\,8\hh\xf=2\hs\xi(\ew)=2\rho\hs(\hn\ev,\ew)=(d\hh\phi)(\hn\eu,\ew)
=-\hs d_\ew[\hh\phi(\hn\ev)]
-\phi(\nabla_{\hskip-2.2pt\ev}\ew-\nabla_{\hskip-2.2pt\ew}\ev)
=4\hs d_\ew\xf-\mu\hh\phi(\hn\ev)=(-16+4\hh\mu)\hh\xf$, so that $\,\mu=6$. We 
thus have (\ref{nvv}), as well as (\ref{bvw}) except for the relation 
$\,d_\ev\xh=\hs4$. To establish it, we use (\ref{rer}.a) and (\ref{cur}), 
obtaining $\,4\hh\xf\ew=\rho\hs(\hn\ev,\ew)\hh\ew=R\hh(\hn\ev,\ew)\hh\ew
=4\hh\xf\ew+15\hh(4-d_\ev\xh)\hh\ev$, which completes the proof.
\end{proof}

\section{RSTS connections associated with a Lo\-rentz\-i\-an 
3\hs-space}\label{lore}
\setcounter{equation}{0}
Let $\,\plane\,$ be a \hbox{two\hh-}\hskip0ptdi\-men\-sion\-al real vector 
space with a fixed area form $\,\varOmega\,$ (cf.\ Example~\ref{slinv}). The 
uni\-mod\-u\-lar group $\,\mathrm{SL}\hs(\plane)\,$ acts, by conjugation, on 
the \hbox{three\hh-}\hskip0ptdi\-men\-sion\-al vector space $\,\spac\,$ of all 
trace\-less en\-do\-mor\-phisms of $\,\plane$, and its action preserves the 
Lo\-rentz\-i\-an $\,(\mpp)\,$ inner product $\,\lr\,$ in $\,\spac\,$ 
characterized by $\,\lg A,A\rg=-\det\hs A\,$ for $\,A\in \spac\nh$. In other 
words, $\,\spac\,$ is the Lie algebra of $\,\mathrm{SL}\hs(\plane)$, the 
action amounts to the adjoint representation, and $\,\lr\,$ is, up to a 
factor, the Kil\-ling form of $\,\mathrm{SL}\hs(\plane)$. The action of 
$\,\mathrm{SL}\hs(\plane)\,$ is not effective: its kernel is the center 
$\,\bbZ_2\nh=\{\mathrm{Id},-\hs\mathrm{Id}\}\,$ of $\,\mathrm{SL}\hs(\plane)$, 
and $\,\mathrm{SL}\hs(\plane)/\bbZ_2$ acting on $\,\spac\,$ is nothing else 
than the identity component $\,\mathrm{SO}^\uparrow(\spac)\,$ of the Lo\-rentz 
group of $\,(\spac\nh,\lr)$.

The $\,\mathrm{SL}\hs(\plane)$-equi\-var\-i\-ant quadratic mapping 
$\,\Phi:\plane\to\spac\,$ defined by 
$\,\Phi(y)=\varOmega(y,\,\cdot\,)\otimes y\,$ sends 
$\,\plane\smallsetminus\{0\}\,$ onto the {\it future null cone\/} $\,\bs\,$ 
in $\,\spac\nh$, which is a specific connected component $\,\bs\,$ of the set 
of nonzero $\,\lr$-null vectors.
\begin{proposition}\label{rstsc}Under the above hypotheses, the future null 
cone\/ $\,\bs\,$ in\/ $\,\spac\,$ admits a 
\hbox{one\hh-}\hskip0ptpa\-ram\-e\-ter family of\/ 
tor\-sion\-free connections invariant under the transitive action of\/ 
$\,\mathrm{SO}^\uparrow(\spac)\,$ on\/ $\,\bs\,$ and having nonzero, 
skew-sym\-met\-ric Ric\-ci tensor. All these connections represent the 
point\/ $\,(\ea,\eb)=(1,0)\,$ of the moduli curve in 
Theorem\/~{\rm\ref{modul}(i)}.
\end{proposition}
\begin{proof}Since $\,\Phi:\plane\smallsetminus\{0\}\to\bs\,$ is a 
\hbox{two\hs-}\hskip0ptfold covering, we may choose the connections in 
question to be the $\,\Phi$-im\-ages of the 
$\,\mathrm{SL}\hs(\plane)$-in\-var\-i\-ant connections $\,\nabla\hn\,$ on 
$\,\plane\,$ described in Example~\ref{slinv}. (The deck transformation 
$\,-\hs\mathrm{Id}\in\mathrm{SL}\hs(\plane)\,$ leaves any such $\,\nabla\hn\,$ 
invariant.)
\end{proof}
As before, let $\,\bs\,$ be a future null cone in a $\,3$-space $\,\spac\,$ 
endowed with a Lo\-rentz\-i\-an $\,(\mpp)\,$ inner product $\,\lr$, and let 
$\,\yp\,$ be the sheet, adjacent to $\,\bs$, of the 
\hbox{two\hh-}\hskip0ptsheet\-ed hyperboloid formed by all $\,\lr$-unit 
time\-like vectors in $\,\spac\nnh$. We refer to $\,\yp\,$ as the {\it hyperbolic 
plane}, since $\,\lr\,$ induces on $\,\yp\,$ a Riemannian metric of constant 
curvature $\,-1$. Similarly, $\,\lr\,$ induces a Lo\-rentz\-i\-an 
$\,(\hbox{$-$\hskip1pt$+$})\,$ metric of constant curvature $\,1\,$ on the 
\hbox{one\hh-}\hskip0ptsheet\-ed hyperboloid $\,\ds\,$ of all $\,\lr$-unit 
space\-like vectors in $\,\spac\nnh$.

The unit tangent bundle $\,T^1\nnh\yp\,$ of $\,\yp\,$ may be 
identified with the submanifold of $\,\yp\times\ds\,$ consisting of all 
$\,\lr$-or\-thog\-o\-nal pairs $\,(p,q)\in\yp\times\ds$. The formula 
$\,F(p,q)=p+q\,$ defines a mapping $\,F:T^1\nnh\yp\to\bs\,$ which is a 
fibration, as the connected Lo\-rentz group $\,\mathrm{SO}^\uparrow(\spac)\,$ acts 
transitively on both $\,T^1\nnh\yp\,$ and $\,\bs$, while $\,F\,$ is obviously 
$\,\mathrm{SO}^\uparrow(\spac)$-equi\-var\-i\-ant. The fibres of $\,F\,$ are the 
leaves of the {\it horo\-cycle foliation\/} on $\,T^1\nnh\yp$, called so 
because they are easily verified to be the natural lifts to $\,T^1\nnh\yp\,$ 
of oriented horo\-cycles in $\,\yp$.
\begin{remark}\label{horoc}The horo\-cycle foliation descends from 
$\,T^1\nnh\yp\,$ to the unit tangent bundle $\,T^1\bs\,$ of any closed 
orientable surface $\,\bs\,$ of genus greater than $\,1\,$ endowed with a 
hyperbolic metric. This is due to its invariance under the action on 
$\,T^1\nnh\yp\,$ of the group $\,\mathrm{SO}^\uparrow(\spac)\,$ of all 
o\-ri\-en\-ta\-tion-pre\-serv\-ing isometries of the hyperbolic plane $\,\yp$. 
The invariance follows in turn from the 
$\,\mathrm{SO}^\uparrow(\spac)$-equi\-var\-i\-ance of $\,F\nh$, mentioned above.
\end{remark}
\begin{remark}\label{nlgeo}If $\,\line\,$ is a $\,\lr$-null 
\hbox{one\hh-}\hskip0ptdi\-men\-sion\-al subspace of our Lo\-rentz\-i\-an 
$\,3$-space $\,\spac\nh$, then $\,\lr\,$ restricted to the plane 
$\,\line\nnh^\perp$ is positive sem\-i\-def\-i\-nite and degenerate. Thus, the 
set $\,\line\nnh^\perp\nnh\cap\hs\ds\,$ of all $\,\lr$-unit space\-like 
vectors in $\,\line\nnh^\perp$ is the union of two parallel lines, cosets of 
$\,\line$, each of which is a null geodesic in the 
\hbox{one\hh-}\hskip0ptsheet\-ed hyperboloid $\,\ds\,$ with its submanifold 
metric. Consequently, as $\,\mathrm{SO}^\uparrow(\spac)\,$ acts on $\,\ds\,$ 
transitively, $\,\ds\,$ carries two foliations, the leaves of which are 
maximal null geodesics in $\,\ds\,$ and, simultaneously, straight lines in 
$\,\spac\nh$, in such a way that each leaf of one foliation is disjoint with 
(and, as a line, parallel to) exactly one leaf of the other foliation.
\end{remark}
\begin{proposition}\label{oshhp}For\/ $\,\spac\nh,\lr\,$ and the 
\hbox{one\hh-}\hskip0ptsheet\-ed hyperboloid\/ $\,\ds\subset \spac\,$ as 
above, let\/ $\,\line\nnh^\pm$ be the two parallel lines forming the set\/ 
$\,\line\nnh^\perp\nnh\cap\hs\ds$, where\/ $\,\line\,$ is a fixed\/ 
$\,\lr$-null \hbox{one\hh-}\hskip0ptdi\-men\-sion\-al subspace of\/ 
$\,\spac\nh$. Then the surface\/ $\,\ds'\nh=\ds\smallsetminus\line\nnh^-$ 
admits a tor\-sion\-free connection\/ $\,\nabla\hn\,$ invariant under the 
action of the  \hbox{two\hh-}\hskip0ptdi\-men\-sion\-al Lie group\/ 
$\,\hp=\{C\in\mathrm{SO}^\uparrow(\spac):C(\line)=\line\}\,$ and having 
eve\-ry\-\hbox{where\hskip1pt-}\hskip0ptnon\-zero, skew-sym\-met\-ric Ric\-ci 
tensor. Furthermore, the 
restriction of\/  $\,\nabla\hn\,$ to the open subset\/ 
$\,\ds'\nh\smallsetminus\line\nnh^+\nh=\ds\smallsetminus\line\nnh^\perp$ is 
locally homogeneous and represents the point\/ $\,(\ea,\eb)=(0,1)\,$ on the 
moduli curve of Theorem\/~{\rm\ref{modul}(i)}.
\end{proposition}
\begin{proof}Define vector fields $\,\ev_\pm^{\phantom i}$ on the open 
set $\,\spac\smallsetminus\line\nnh^\perp$ in $\,\spac\,$ by 
$\,\ev_\pm^{\phantom i}=y\pm q-\lg y,y\pm q\rg\lg y,p\rg^{-1}p$, where 
$\,y\,$ denotes the radial (identity) vector field on $\,\spac\,$ and 
$\,p,q\in\spac\,$ are constant vector fields with 
$\,p\in\line\smallsetminus\{0\}\,$ and $\,q\in\line\nnh^+\nnh$. Both 
$\,\ev_\pm^{\phantom i}$ are easily seen to remain unchanged when a different 
choice of $\,p\,$ or $\,q\,$ is made, so that they depend only on 
$\,\line\nnh^+\nnh$, which makes $\,\ev_+^{\phantom i}$ and 
$\,\ev_-^{\phantom i}$ invariant under the action of $\,\hp$. (As $\,\hp\,$ is 
connected, $\,C(\line\nnh^+)=\line\nnh^+$ for all $\,C\in\hp$.) Also, 
$\,\lg\ev_\pm^{\phantom i},y\rg=0$, 
$\,\lg\ev_\pm^{\phantom i},\ev_\pm^{\phantom i}\rg=1-\lg y,y\rg$, 
$\,\lg\ev_+^{\phantom i},\ev_-^{\phantom i}\rg=-1-\lg y,y\rg$, so that, at 
every point of the surface 
$\,\ds''\nnh=\ds\smallsetminus\line\nnh^\perp\nnh$, 
the vector fields $\,\ev_\pm^{\phantom i}$ are tangent to $\,\ds''\nnh$, 
null, and linearly independent. Therefore, the vector fields 
$\,\eu=3\hh(\hn\ev_-^{\phantom i}\nh-\ev_+^{\phantom i})\,$ and 
$\,\ew=2\hh\ev_-^{\phantom i}$ trivialize the tangent bundle of $\,\ds''\nnh$. 
As $\,[\eu,\ew]=2\eu$, the connection $\,\nabla\hn\,$ defined by (\ref{nuu}) 
with these $\,\eu,\ew\,$ and $\,(\ea,\eb)=(0,1)\,$ has all the required 
properties except for being defined just on $\,\ds''\nnh$, rather than 
everywhere in $\,\ds'\nnh$.

To show that $\,\nabla\hn\,$ has a $\,C^\infty$ extension to $\,\ds'\nnh$, 
let us note that the function $\,\psi=\lg y,p\rg\,$ and the vector field 
$\,X=\psi\hh\ev_+^{\phantom i}$ are of class $\,C^\infty$ on $\,\ds'\nnh$, 
and hence so is $\,Z=\psi^{-1}\ev_-^{\phantom i}$. (The last conclusion 
follows as 
$\,\lg X,Z\rg=\lg\ev_+^{\phantom i},\ev_-^{\phantom i}\rg=-2\,$ on 
$\,\ds''\nnh$, while $\,X,Z\,$ are both null, and $\,X\nh\ne0\,$ everywhere 
in $\,\ds'\nnh$.) Furthermore, $\,d_\ev\psi=\psi\,$ both for 
$\,\ev=\ev_+^{\phantom i}$ and $\,\ev=\ev_-^{\phantom i}$. Thus, 
$\,\nabla\hskip-2.3pt_X^{\phantom i}X=\psi^3Z-\psi X$, 
$\,\nabla\hskip-2.3pt_X^{\phantom i}Z=-\psi Z$, 
$\,\nabla\hskip-2.3pt_Z^{\phantom i}X=\psi Z$, 
$\,\nabla\hskip-2.3pt_Z^{\phantom i}Z=0$, and our assertion follows.
\end{proof}
In the above proof, the (skew-sym\-met\-ric) Ric\-ci tensor $\,\rho\,$ of 
$\,\nabla\hn\,$ is nonzero everywhere in $\,\ds'\nnh$, as 
$\,\rho\hs(X,Z)=\rho\hs(\ev_+^{\phantom i},\ev_-^{\phantom i})
=-\hh\rho\hs(\hn\eu,\ew)/6=-1\,$ by (\ref{ruw}.a). Also, 
$\,\ew=2\hh\ev_-^{\phantom i}=2\hh\psi Z\,$ vanishes on $\,\line\nnh^+$, since 
so does $\,\psi$. Thus, Proposition~\ref{oshhp} illustrates, just like 
Example~\ref{slsgp}, the necessity of the assumption that $\,\ea+\eb\ne1\,$ 
for conclusion (ii) in Proposition~\ref{wnonz}.

\section{Transversal RSTS connections}\label{trrs}
\setcounter{equation}{0}
We discuss here transversal RSTS connections having 
eve\-ry\-\hbox{where\hskip1pt-}\hskip0ptnon\-zero Ric\-ci tensor. Trans\-ver\-sal 
tor\-sion\-free connections which are {\it flat} were studied, in any 
co\-di\-men\-sion, by Wolak \cite{wolak}.

Suppose that $\,\mathcal{F}\,$ is a co\-di\-men\-sion $\,m\,$ foliation on a 
manifold $\,\ym\,$ and $\,\mathcal{V}\,$ is the co\-di\-men\-sion $\,m\,$ 
distribution tangent to $\,\mathcal{F}\nnh$. A local section of the quotient 
bundle $\,(\tm)/\hs\mathcal{V}\hs$ defined on a nonempty open set 
$\,\,U\subset\ym\,$ will be called 
$\,\mathcal{F}\nnh${\it-pro\-ject\-a\-ble\/} if it is the image, under the 
quotient projection $\,\tm\to(\tm)/\hs\mathcal{V}\nh$, of some 
$\,\mathcal{V}\nh$-pro\-ject\-a\-ble local vector field defined on $\,\,U\,$ 
(cf.\ Remark~\ref{liebr}). Following Molino \cite{molino}, by a {\it 
transversal connection\/} for $\,\mathcal{F}\,$ we mean any operation 
$\,\bna\,$ associating with every nonempty open set $\,\,U\subset\ym\,$ and 
every pair $\,\ev,\ev\hh'$ of $\,\mathcal{F}\nnh$-pro\-ject\-a\-ble local 
sections of $\,(\tm)/\hs\mathcal{V}\nh$, defined on $\,\,U\nh$, an 
$\,\mathcal{F}\nnh$-pro\-ject\-a\-ble local section 
$\,\bna_{\hskip-2pt\ev}\hn\ev\hh'$ of $\,(\tm)/\hs\mathcal{V}\nh$, defined on 
$\,\,U\nh$, in such a way that
\begin{enumerate}
  \def\theenumi{{\rm\roman{enumi}}}
\item[(i)] the dependence of $\,\bna_{\hskip-2pt\ev}\hn\ev\hh'$ on $\,\ev\,$ 
and $\,\ev\hh'$ is local, and, in particular, the operations $\,\bna\,$ 
corresponding to two intersecting open sets agree on their intersection,
%$\,\,U_1^{\phantom i},U_2^{\phantom i}$ 
%$\,\,U_1^{\phantom i}\cap U_2^{\phantom i}$,
\item[(ii)] if $\,\,U\,$ satisfies (\ref{bdl}) then, for some connection 
$\,\nabla\,$ on the base $\,\bs$, and any 
$\,\mathcal{F}\nnh$-pro\-ject\-a\-ble local sections $\,\ev,\ev\hh'$ of 
$\,(\tm)/\hs\mathcal{V}\nh$, defined on $\,\,U\nh$, the $\,\pi$-im\-age of 
$\,\bna_{\hskip-2pt\ev}\hn\ev\hh'$ is the vector field $\,\nsww\hh'$ on 
$\,\bs$, where $\,\ew,\ew\hh'$ stand for the $\,\pi$-im\-ages of $\,\ev\,$ and 
$\,\ev\hh'\nnh$.
\end{enumerate}
All local properties of connections on manifolds make sense for transversal 
connections. One can thus speak of co\-di\-men\-sion-two foliations on 
manifolds with transversal tor\-sion\-free connections, the Ric\-ci tensor of 
which is \hbox{skew\hh-}\hskip0ptsymmetric and nonzero at every point. From 
now on we refer to them as {\it transversal RSTS connections with 
eve\-ry\-\hbox{where\hskip1pt-}\hskip0ptnon\-zero Ric\-ci tensor}.
\begin{example}\label{trsts}Each of these cases leads to a 
transversal RSTS connection $\,\bna\,$ with 
eve\-ry\-\hbox{where\hskip1pt-}\hskip0ptnon\-zero Ric\-ci tensor for a 
co\-di\-men\-sion-two foliation $\,\mathcal{F}\,$ on an 
\hbox{$\,n$\hh-}\hskip0ptdi\-men\-sion\-al manifold.
\begin{enumerate}
  \def\theenumi{{\rm\alph{enumi}}}
\item[(a)] $n=4\,$ and $\,\mathcal{F}\,$ is the foliation tangent to the 
vertical distribution $\,\mathcal{V}\,$ of a type~III SDNE Walker manifold 
(immediate from Theorem~\ref{maith}).
\item[(b)] $n=2\,$ and $\,\mathcal{F}\,$ is the 
\hbox{$0$\hh-}\hskip0ptdi\-men\-sion\-al foliation on a surface $\,\bs\,$ with 
a fixed RSTS connection $\,\nabla\nh$, the Ric\-ci tensor of which is 
\hbox{skew\hh-}\hskip0ptsymmetric and nonzero everywhere (obvious).
\item[(c)] $n\ge3\,$ is arbitrary and $\,\mathcal{F}\,$ is the vertical 
foliation on the total space of a locally trivial bundle over a surface 
$\,\bs\,$ with $\,\nabla\hn\,$ as in (b) (obvious).
\item[(d)] $n=3\,$ and $\,\mathcal{F}\,$ is the horo\-cycle foliation on the 
unit tangent bundle $\,T^1\nnh\yp\,$ of the hyperbolic plane $\,\yp\,$ 
(immediate from Proposition~\ref{rstsc}, the lines preceding 
Remark~\ref{horoc}, and (c)).
\end{enumerate}
\end{example}
In cases (a), (b) and (c) above, {\it the underlying manifold cannot be 
compact}, as shown in Theorem~\ref{nocpl}(a) and 
\cite[the lines following Theorem~5.1]{derdzinski-08}. In (d), however, 
although $\,T^1\nnh\yp\,$ is noncompact, it has compact quotients to which 
$\,\mathcal{F}\,$ descends:
\begin{proposition}\label{genus}The horo\-cycle foliation on the unit tangent 
bundle $\,T^1\bs\,$ of any closed orientable surface\/ $\,\bs\,$ of genus 
greater than\/ $\,1$, for any hyperbolic metric on\/ $\,\bs$, admits a 
transversal RSTS connection with 
eve\-ry\-\hbox{where\hskip1pt-}\hskip0ptnon\-zero Ric\-ci tensor.
\end{proposition}
This is a direct consequence of Example~\ref{trsts}(d), Remark~\ref{horoc} and 
Proposition~\ref{rstsc}.
\begin{corollary}\label{hidim}Compact manifolds with co\-di\-men\-sion-two 
foliations that admit transversal RSTS connections having 
eve\-ry\-\hbox{where\hskip1pt-}\hskip0ptnon\-zero Ric\-ci tensor exist in all 
dimensions\/ $\,n\ge3$, but not in dimension\/ $\,2$.
\end{corollary}
In fact, if $\,n\ge3$, it suffices to combine Proposition~\ref{genus} with the 
obvious Car\-te\-sian-prod\-uct construction. For $\,n=2$, see the lines 
preceding Proposition~\ref{genus}.

%\section{Generic metrics in the type~III SDNE Walker class}\label{gtts}
%\setcounter{equation}{0}

\section{Type~III SDNE 
\hbox{non\hh-\hskip-1.5pt}\hskip0ptWalk\-er manifolds}\label{glpr}
\setcounter{equation}{0}
It is not known whether Theorem~\ref{nocpl} remains true without the 
assumption that $\,\gm\,$ is a Walker metric. This section presents 
a global condition unrelated to compactness, which, although 
satisfied by some type~III SDNE Walker manifolds, can never hold in 
the \hbox{non\hh-\hskip-1.5pt}\hskip0ptWalk\-er case.

Specifically, we say that a type~III SDNE manifold $\,(\ym,\gm)\,$ is 
{\it vertically complete\/} if every leaf of its vertical distribution 
$\,\mathcal{V}\,$ is complete as a manifold with the connection induced 
by the Le\-vi-Ci\-vi\-ta connection of $\,\gm$. Thus, geodesic completeness 
of $\,(\ym,\gm)\,$ implies its vertical completeness. Note that the leaves 
of $\,\mathcal{V}\,$ are totally geodesic (Section~\ref{vert}), and the 
connection induced on a leaf is always flat 
\cite[Lemma 5.2(i)]{derdzinski-09}.

All pairs $\,(\ym,\gm)=(\tab,\hs\gm\nnh^\nabla\nnh+\hs2\hh\pi^*\nnh\tau)\,$ 
described in the first part of Theorem~\ref{maith} are examples of 
vertically complete type~III SDNE {\it Walker\/} manifolds. In fact, 
by (\ref{gjk}), the connection induced on each leaf $\,\tayb\,$ is the 
standard flat connection on the vector space $\,\tayb$. On the other hand, 
no such examples are possible in the 
\hbox{non\hh-\hskip-1.5pt}\hskip0ptWalk\-er case:
\begin{theorem}\label{compl}Every vertically complete type\/ {\rm III} SDNE 
manifold has the Walker property.
\end{theorem}
\begin{proof}Suppose that, on the contrary, $\,(\ym,\gm)\,$ is a 
vertically complete type~III \hbox{non\hh-\hskip-1.5pt}\hskip0ptWalk\-er SDNE 
manifold, and so the vertical distribution is not parallel (see 
Section~\ref{vert}). Let $\,\alpha\,$ and $\,\beta\,$ be the $\,1$-forms on 
$\,\ym$, defined in \cite[Lemma 5.2]{derdzinski-09}. Thus, according to 
\cite[Theorem 6.2(ii)]{derdzinski-09}, $\,\beta\ne0\,$ somewhere in 
$\,\ym$. Formulae (8.1.g) and (8.2.i) in \cite{derdzinski-09} now imply 
that, for some leaf $\,\yn\,$ of $\,\mathcal{V}\nh$, the $\,1$-form on 
$\,\xi\,$ on $\,\yn\,$ obtained by restricting $\,\alpha\,$ to $\,\yn\,$ is 
not identically zero, while \cite[formula (8.1.b)]{derdzinski-09} states that 
$\,d_\eu\hs[\hh\alpha(\hn\ev)]=\alpha(\hn\eu)\hs\alpha(\hn\ev)\,$ for any 
local sections $\,\eu,\ev$ of $\,\mathcal{V}\,$ parallel along 
$\,\mathcal{V}\nh$. (The three formulae are established in 
\cite[Theorem 8.4]{derdzinski-09}, without using the assumption that 
$\,\beta\ne0\,$ {\it everywhere}.) Denoting by $\,\mathrm{D}\,$ the complete 
flat connection induced on the leaf $\,\yn$, we thus have 
$\,\mathrm{D}\hs\xi=\xi\otimes\hs\xi$, and, choosing a geodesic 
$\,\bbR\ni s\mapsto x(s)\in\yn\,$ with the velocity vector field $\,v(s)\,$ 
such that $\,\mu(s)=\xi_{\hs x(s)}(v(s))\,$ is nonzero for some 
$\,s\in\bbR\hh$, we obtain $\,d\hh\mu/ds=\mu^2\nnh$. Hence $\,\mu\,$ cannot be 
defined everywhere in $\,\bbR\hh$. This contradiction completes the proof.
\end{proof}

\section*{Acknowledgements}
The author wishes to thank Ta\-de\-usz Ja\-nusz\-kie\-wicz for suggesting 
the unit tangent bundles of hyperbolic surfaces as a possibility in the 
construction leading to Proposition~\ref{genus}.

\end{document}